\documentclass[11pt]{article}
\usepackage{latexsym,verbatim,ifthen,graphicx,mathrsfs}
\usepackage[english]{babel}
\usepackage[utf8]{inputenc} 	
\usepackage{amssymb}
\usepackage{amsmath}
\usepackage{amsthm}	
\usepackage{tikz-cd}
\usepackage{mathtools}  
\usepackage{microtype}		 
\usepackage[overload]{empheq}
\usepackage{graphicx}

\usepackage{calc}
\newlength\myheight
\newlength\mydepth
\settototalheight\myheight{Xygp}
\usetikzlibrary{matrix,arrows,decorations.pathmorphing}
\usepackage[all]{xy}				
\usepackage{hyperref}
\hypersetup{
	colorlinks=true,
	allcolors=blue,
}   		
\usepackage{anysize}
\usepackage{cancel}									
\marginsize{2.5cm}{2.5cm}{2cm}{2cm}
\usepackage{enumerate}

\usepackage{mathrsfs}									
\usepackage{fourier}			
\usepackage[Sonny]{fncychap}	
\usepackage{fancyhdr}

\newtheorem{theorem}{Theorem}[section]
\newtheorem{lemma}[theorem]{Lemma}
\newtheorem{proposition}[theorem]{Proposition}
\newtheorem{corollary}[theorem]{Corollary}
\theoremstyle{definition}
\newtheorem{example}[theorem]{Example}

\theoremstyle{definition}
\newtheorem{remark}[theorem]{Remark}
\newtheorem{remarks}[theorem]{Remarks}
\newtheorem{definition}[theorem]{Definition}

\newtheorem{notation}[theorem]{Notation}

\newtheorem{thm}{Theorem}

\DeclareMathOperator{\Ker}{Ker}

\DeclareMathOperator{\id}{id}
\DeclareMathOperator{\idd}{id}

\newcommand{\antishriek}{\text{\raisebox{\depth}{\textexclamdown}}}

\renewcommand{\P}{{\mathcal{P}}}

\newcommand{\K}{{\Bbbk}}

\newcommand{\Vect}{\mathsf{Vect}}

\definecolor{red}{rgb}{1,0.1,0.1}
\definecolor{blue}{rgb}{0.1,0.1,1}
\definecolor{green}{rgb}{0,100,0}

\title{Higher-order Massey products for algebras over algebraic operads}

\author{Oisín Flynn-Connolly, José M. Moreno-Fernández}

\begin{document}

\date{}

\maketitle

\abstract{We introduce higher-order Massey products for algebras over algebraic operads.
This extends the work of Fernando Muro on secondary ones.
We study their basic properties and behavior with respect to morphisms of algebras 
and operads and give some connections to formality. 
We prove that these higher-order operations represent the differentials 
in a naturally associated operadic Eilenberg--Moore spectral sequence.
We also study the interplay between particular choices of higher-order Massey products and quasi-isomorphic 
$\mathcal P_\infty$-structures on the homology of a $\mathcal P$-algebra.
We focus on Koszul operads over a characteristic zero field 
and indicate how our results generalize to the non-Koszul case.}


\section{Introduction}

In \cite{Mas69}, reprinted as \cite{Mas98},
W. S. Massey introduced the classical triple Massey product, 
a secondary operation on the (co)homology of differential graded associative algebras.
He used this new operation to show that the Borromean rings are non-trivially linked.
Similar secondary operations were defined independently by Allday and Retakh 
on the homology of differential graded Lie algebras,
see \cite{Allday73,Allday77,Ret77}.
The existence of these higher-order products is due to the vanishing of certain obstructions
that follow from the associativity and Jacobi relations at the chain level, respectively.
Recently, 
F. Muro has shown that secondary operations on the homology of differential graded algebraic structures
analogous to Massey's in the case of associative algebras are not ad-hoc at all \cite{Mur21}.
Indeed, the theory of algebraic operads explains and organizes the existence and construction of these operations.
An algebraic operad is an operad in the symmetric monoidal category of 
$\mathbb Z$-graded vector spaces over a 
characteristic zero field, and will be assumed to be Koszul.
In \emph{loc. cit.}, 
Muro defines secondary Massey products for algebras over algebraic operads.
Given an algebraic operad $\mathcal P$,
each quadratic relation in the presentation of $\mathcal P$
defines a secondary Massey-product-like operation on the homology of the $\mathcal P$-algebras.
This secondary operation takes as many inputs as the arity of the relation.
In this way, 
the associativity relation of the associative operad yields the classical triple Massey products,
while the Jacobi identity relation of the Lie operad yields the Lie--Massey brackets.
Under this new point of view, 
Muro uncovered secondary Massey-product-like operations for
many distinct types of algebras for the first time, and gave applications to hyper-commutative and Gerstenhaber algebras.

\medskip

In Muro's paradigm, there is no restriction as to the arity of the relation. 
Thus, a relation $\Gamma$ of arity $r$ in a presentation of an operad $\mathcal P$ 
produces a Massey-product-like operation  with
$r$ inputs $\langle -,\ldots,- \rangle_\Gamma$ on the homology of the $\mathcal P$-algebras.
However, this still left the definition of \emph{higher-order} Massey product operations unclear.
Our work addresses this gap.
It is well-known that the triple Massey product is just the first of an infinite series of higher-order operations
on the homology of dg associative algebras, roughly 
witnessing the different ways in which an $n$-fold product in homology vanishes as a consequence of associativity.
These higher-order products have been shown many times 
to be essential in a wide range of topics where triple Massey products are not enough,
see for example the survey \cite{Limonchenko}.
In particular, they are concrete tools for computations when a fully-fledged $A_\infty$-structure is not available.

\medskip

These higher operations refine Muro's secondary ones,
coinciding with them in all our examples,
and gather together to form the hierarchy of higher operations on the homology of algebras 
over algebraic operads mentioned before.
Our approach generalizes the fruitful framework of higher-order Massey products for 
dg associative algebras to algebras over any algebraic operad,
producing a new tool to perform computations in many kinds of dg algebras.

The importance of these higher-order operations seems to have been neglected 
due to the following widespread misconception:
whenever a higher-order Massey product $\langle x_1,\ldots,x_r\rangle$ is defined
on the homology of a dg associative algebra,
any  $A_\infty$-structure $\left\{m_r\right\}$  on its homology transferred
via the homotopy transfer theorem satisfies 
\begin{equation}
    \label{ecu : Intro}    
\pm m_r\left(x_1,\ldots,x_r\right) \in \langle x_1,\ldots,x_r\rangle.
\end{equation}
This is true only for the triple Massey product, 
but fails in general \cite{jose20}.
This fact makes the higher-order operations defined in this paper important, 
filling a fundamental gap in the understanding of the homology of dg algebraic structures.
Being slightly more precise, 
we show that if the homology of an algebra over a Koszul operad $\mathcal P$
is endowed with a $\mathcal P_\infty$-algebra structure quasi-isomorphic to the original structure,
then the $\mathcal P_\infty$-algebra structure maps recover  higher-order Massey 
products only up to lower-weight $\mathcal P_\infty$-algebra structure maps.
We also prove, however, a positive result in this direction:  
for any choice of class in a higher-order Massey product set,
under a linear independence hypothesis,
one can make appropriate choices in the homotopy transfer theorem so that the induced 
$\mathcal P_\infty$ structure on the homology of the $\mathcal P$-algebra 
recovers this choice exactly by Formula \eqref{ecu : Intro}.

\medskip

Let us briefly explain how these higher-order Massey products arise.
Let $\mathcal P$ be a Koszul operad with Koszul dual cooperad $\mathcal P^\antishriek$ 
(we indicate in Remark \ref{remark: Massey products for non-Koszul operads} how to deal with the non-Koszul case).
Each weight-homogeneous cooperation $\Gamma^c$ of $\mathcal P^\antishriek$ gives rise to a partially defined
higher operation $\langle -,\ldots,-\rangle_{\Gamma^c}$ on the homology of any $\mathcal P$-algebra.
The number of inputs of this operation is the arity $r$ of $\Gamma^c$.   
If $A$ is a $\mathcal P$-algebra, then out of homogeneous elements $x_1,\ldots,x_r \in H_*(A)$,
the operation gives a (possibly empty) set of homology classes
$$\langle x_1,\ldots,x_r\rangle_{\Gamma^c} \subseteq H_*(A).$$
The non-emptiness depends on the vanishing,
in a precise sense,
of higher operations of the same kind that arise from $\Gamma^c$ and have strictly lower weight-degree.
We call $\langle x_1,\ldots,x_r\rangle_{\Gamma^c}$ the \emph{$\Gamma^c$-Massey product} of the classes $x_1,\ldots,x_r$.
The process to construct the $\Gamma^c$-Massey product operation $\langle -,\ldots,-\rangle_{\Gamma^c}$ 
is done by a non-trivial analogy with the case of dg associative algebras.
To wit, the cooperation $\Gamma^c$ determines a set of indices $I\left(\Gamma^c\right)$ 
which is then used to form \emph{defining systems}.
Fixed a $\mathcal{P}$-algebra $A$ and homogeneous elements $x_1,\ldots,x_r\in H_*(A)$, where $r$ is the arity of $\Gamma^c$,
a defining system for the $\Gamma^c$-Massey product $\langle x_1,\ldots,x_r\rangle_{\Gamma^c}$ is a coherent choice of elements 
 $\left\{a_\alpha\right\}$ of $A$ indexed by $I\left(\Gamma^c\right)$ 
that conspire together to create a cycle.
Running over all possible choices of defining systems for $x_1,\ldots,x_r$,
we obtain all possible representatives of the 
homology classes in the set $\langle x_1,\ldots,x_r\rangle_{\Gamma^c}$.
This construction is the core of the paper, 
and it is performed in Section \ref{sec : higher-order operadic Massey products}.
Since the details are quite technical, 
we skip them for the moment and refer the reader to the mentioned section.
There, we give explicit examples,
including the case of the associative, commutative, Lie, and dual numbers operads.
We prove that our framework generalizes and refines
Muro's in Proposition \ref{prop : Our products coincide with Muro's}.

\smallskip

The remainder of the paper is organized as follows. 
We establish the fundamental properties of these new higher-order operations
in Sections \ref{sec: Elementary properties of the operadic Massey products} 
and \ref{sec : Massey products along morphisms of operads and formality}.
There, we prove
that morphisms of $\mathcal P$-algebras preserve higher-order Massey products
and that quasi-isomorphisms induce a bijective correspondence between them (Prop. \ref{prop : Morphisms preserve Massey products}).
This makes higher-order Massey products a powerful invariant for studying the homotopy types of algebras over operads,
particularly in the context of formality.
In Section \ref{sec : Massey products along morphisms of operads and formality}, 
we examine their behavior under morphisms of operads. Under mild assumptions,
higher-order Massey products can be pulled back and pushed forward 
(Prop. \ref{Prop: Pullback of Massey products} and Cor. \ref{cor: otra}). 
This allows one to relate the quasi-isomorphism class (in particular, the formality) 
of an algebra of a given type to that of a functorially associated algebra of a different type. 
A motivating example is the adjunction between the universal enveloping algebra 
of a dg Lie algebra and the commutator bracket of a dg associative algebra, 
where formality and quasi-isomorphism classes can be transferred in both directions under suitable hypotheses,
see \cite{campos2020lie}.

\medskip

Next, we apply these operations to the study of spectral sequences and homotopy transfer. 
It is a celebrated classical result that higher-order Massey products for associative algebras describe the differentials 
in the Eilenberg--Moore spectral sequence.
We generalize this by constructing an Eilenberg--Moore-type spectral
sequence for algebras over any algebraic operad, which, under mild hypotheses, computes their Quillen homology. 
In Section \ref{sec: differentials EMSS}, we prove our first main result: 
the higher-order Massey products defined in this paper provide concrete representatives 
for the differentials in this generalized spectral sequence. More precisely, we prove the following result
(Thm. \ref{Thm : Differentials EMSS and Massey products}).

\begin{thm}
\label{thm: intro EMSS}
Let $A$ be a $\mathcal P$-algebra, and $ x_1, \dots, x_r $ homology classes 
such that the Massey product set $\langle x_1,\ldots,x_r\rangle_{\Gamma^c}$ is defined 
for a cooperation $\Gamma^c \in \left(\mathcal{P}^{\antishriek}\right)^{(n)}(r)$. 
Let the total degree be $t = |\Gamma^c| + \sum_{i=1}^r |x_i|$.
Then the element 
\[
\Gamma^c\otimes x_1\otimes \cdots \otimes x_r 
\in \left(\mathcal P^{\antishriek} \right)^{(n)} (r)\otimes_{\mathbb S_r} H_*(A)^{\otimes r}
\]
survives to the $E^{n}$ page in the $\mathcal P$-Eilenberg--Moore spectral sequence,
and for any $x\in \langle x_1, \dots, x_r \rangle_{\Gamma^c}$
$$
d^{n}\left(\Gamma^c\otimes x_1\otimes \cdots \otimes x_r\right) = (-1)^{n-1}
\left[\idd\otimes x\right] \qquad \mbox{in } E^n_{0, t-1}.
$$
\end{thm}

Finally, Section \ref{sec : Massey products and P-infinity algebras}
establishes a precise relationship between the higher-order Massey products
on the homology of a $\mathcal{P}$-algebra and transferred $\mathcal{P}_\infty$-structures on it,
clarifying exactly how these structures recover higher operations. 
More precisely, we prove the following result (Thm. \ref{thm: P-infinity and recovery of Massey}).

\begin{thm}
\em
Let $A$ be an algebra over a Koszul operad $\mathcal P$,
and let $H$ be its homology endowed with a minimal $\P_\infty$-structure $\delta$.
Let $\Gamma^c\in \left(\mathcal{P}^{\antishriek}\right)^{(n)} (r)$,
and assume that $x_1,\ldots,x_r$ are homogeneous elements of $H$ for which the $\Gamma^c$-Massey product set
$\langle x_1, \dots x_r \rangle_{\Gamma^c}$ is defined.
Let $x \in \langle x_1, \dots x_r \rangle_{\Gamma^c}$. 
Then:
\begin{enumerate}[(i)]
\item For any $\P_\infty$-quasi-isomorphism $F:\P^\antishriek(H) \to \P^\antishriek(A)$
whose linear term $F_1 :H\to A$ is a cycle-choosing map inducing the identity of $H$, 
\[
\delta^{(n)}\left(\Gamma^c \otimes x_1 \otimes \cdots\otimes x_r\right) = (-1)^{n-1}x + \Phi,
\]
where
$\Phi \in \displaystyle \sum^{n-1}_{i=1}\operatorname{Im} \left(\delta^{(i)}\right).$
\item If the elements $\mu\otimes x_{i_1}\otimes \cdots \otimes x_{i_l}$ are linearly independent
in the corresponding copy of $\mathcal{P}^{\antishriek} \otimes_{\mathbb S_l} H^{\otimes l}$, 
where $(\mu,(i_1,\ldots,i_l))\in  I(\Gamma^c)$, then 
there is a choice of $\mathcal{P}_\infty$ structure $\delta$ on $H$, together with a quasi-isomorphism $F$ as above,
which recovers $x$.
\end{enumerate}
\end{thm}

\medskip
\noindent {\bf Acknowledgements:}  
The authors would like to thank Coline Emprin, 
Gr\'{e}gory Ginot, Fernando Muro and Pedro Tamaroff for useful conversations and comments;
and furthermore to the anonymous referees for their time, effort, and excellent advice.
This project has received funding from the European Union’s Horizon 2020 
research and innovation programme under the Marie Skłodowska-Curie grant agreement No 945322. 
\raisebox{-\mydepth}{\fbox{\includegraphics[height=\myheight]{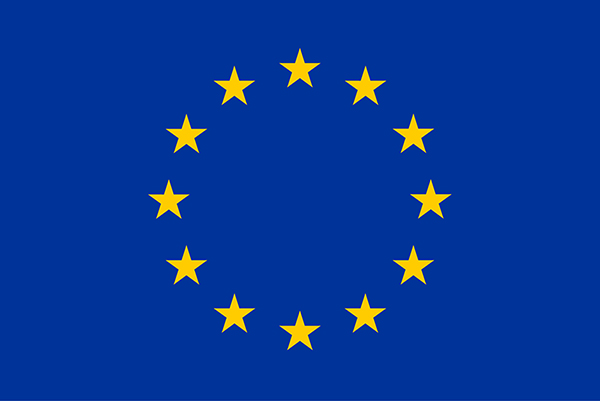}}} 
The second author has been partially supported by the MICINN grant PID2020-118753GB-I00 and
by the Junta de Andalucía grant ProyExcel-00827.

\subsubsection*{Notation and conventions}
In this paper,
all algebraic structures are taken over a base field $\Bbbk$ of characteristic zero.
We work in the category of unbounded chain complexes over $\Bbbk$ with homological convention.
That is, 
the differential $d:A_*\to A_{*-1}$ of a chain complex $\left(A,d\right)$
is of degree $-1$.
The degree of a homogeneous element $x$ is denoted by $|x|$.
The suspension  of a chain complex $\left(A,d_A\right)$ 
is the chain complex $\left(sA,d_{sA}\right)=\left(\Bbbk s\otimes A,1\otimes d\right)$,
where $s$ is a formal variable of degree $1$. 
For a homogeneous element $a\in A$, we denote $sa = s\otimes a \in sA.$
Thus, $\left(sA\right)_* \cong A_{*-1}$, and $d_{sA}\left(sa\right) = -sd_A(a)$ for every such $a\in A.$
The symmetric group on $n$ elements is denoted $\mathbb{S}_n$.
The operads in this paper are taken in the symmetric monoidal category of $\mathbb Z$-graded vector spaces, 
and therefore have zero differential.
In this monoidal category, we follow the Koszul sign rule.
That is, the symmetry isomorphism $U\otimes V \xrightarrow{\cong} V \otimes U$  
that identifies two graded vector spaces is given on homogeneous elements by
$u\otimes v\mapsto (-1)^{|u||v|} v\otimes u$.
Algebras over operads are always differential graded (dg) and homological.
We will frequently omit the adjective ``dg'' and assume it is implicitly understood.
The reason for choosing the operads to have trivial differential is that 
in this case, 
the homology of any dg $\mathcal P$-algebra is a graded (non-dg) $\mathcal P$-algebra again.
If $f:A\to B$ is a morphism of differential graded algebras over an operad, 
then we denote by $f_* :H_*(A)\to H_*(B)$ the induced map in homology.

\subsection{Preliminaries}

In this section,
we collect some of the prerequisites for understanding this paper. 
We start in Section \ref{sec: Operad prerequisites} by giving a brief 
recollection of the results of operad theory that we will make use of,
mainly to establish our notation.
We borrow most of the notation from \cite{loday12}, which is an excellent reference for algebraic operads.
A non-standard topic explained in this section is the construction of the Eilenberg--Moore-type spectral sequence 
mentioned in the introduction.
In Section \ref{sec: Higher Massey products for associative algebras},
we recall the higher-order Massey products
for differential graded associative algebras.
To finish, 
we briefly summarize in Section \ref{sec: secondary Massey products for algebras over operads} 
the construction of the secondary Massey products for algebras over algebraic operads as defined by Muro in \cite{Mur21}.

\subsubsection{Operadic background}
\label{sec: Operad prerequisites}

In this paper, 
we work with operads
in the symmetric monoidal category of graded vector spaces.
Our generic operad $\mathcal P$ is therefore arity-wise made up of $\mathbb Z$-graded vector spaces, 
but it has no differential.
That is, we work with non-dg operads.
The reason is that if $A$ is a $\mathcal P$-algebra, 
we will need $A$ and its homology $H_*(A)$ to be algebras over the same operad. 
Our operads will always be reduced, that is, they satisfy $\mathcal{P}(0) = 0$ and $\mathcal{P}(1) = \Bbbk$.

\medskip

This paper will assume familiarity with the results and notation from \cite{loday12},
and we will adopt its notation for most of the objects used in this paper 
(infinitesimal compositions, twisting morphisms, weight gradings, and Koszul duality).
We shall briefly sketch only those results that will be essential to understand this paper.

\medskip
\noindent{\bf{Quadratic and Koszul operads.}}
A symmetric sequence $E$ is \emph{reduced} if $E(0)=E(1)=0$.
An operad $\mathcal P$ is \emph{quadratic} if it is given by a presentation 
$\mathcal{F}\left(E,R\right)$, that is, 
if it is given as the quotient $\mathcal{F}(E)/(R)$
of the free operad $\mathcal{F}(E)$ on the reduced symmetric sequence $E$ 
by the operadic ideal of relations generated by a sub $\mathbb S$-module of relations $R\subseteq \mathcal F(E)^{(2)}$.
Here, $\mathcal F(E)^{(n)}$ is the sub $\mathbb S$-module of $\mathcal F(E)$ formed by elements of \emph{weight $n$},
that is, formed by combining exactly $n$ generating operations from $E$.
The free operad $\mathcal{F}(E)$ 
comes equipped with a \emph{weight grading} concentrated in non-negative degrees.
Since the operadic ideal $(R)$ is homogeneous with respect to the weight grading of $\mathcal F(E)$
and $\mathcal P$ is a quotient of $\mathcal{F}(E)$, 
the weight grading of $\mathcal F(E)$ naturally descends to $\mathcal P$.
The degree $n$ component of this weight grading on $\mathcal P$ will be denoted $\mathcal P^{(n)}$.
Similarly, 
one can construct the cofree conilpotent cooperad $\mathcal F^c(E)$.
To do so, 
consider the same underlying symmetric sequence $\mathcal F(E)$,
endowed with the same weight-grading.
Dually, we can consider the conilpotent sub-cooperad $\mathcal{F}^c(E, R)$ of 
$\mathcal F^c(E)$ which is final among the conilpotent sub-cooperads $\mathcal C$ of $\mathcal F^c(E)$ 
equipped with a morphism of $\mathbb S$-modules $\mathcal C \to E$
such that the composite
$$
\mathcal C \hookrightarrow \mathcal F^c(E) \twoheadrightarrow \mathcal F^c(E)^{(2)}/R
$$
is 0. 
The weight grading of $\mathcal F^c(E)$ restricts to the sub-cooperad $\mathcal{F}^c(E, R)$, 
and the degree $n$ component of this weight grading on $\mathcal F^c(E,R)$ 
will be denoted $\mathcal F^c(E,R)^{(n)}$.
In particular, and this will be important later, 
the weight 2 component of $\mathcal F^c(E, R)$ is precisely the submodule of co-relations $R$,
$$\mathcal F^c(E,R)^{(2)} = R.$$
We call $\mathcal F^c(E,R)$ the cofree conilpotent cooperad cogenerated by $E$ with corelations $R$.
A cooperad $\mathcal C$ is \emph{quadratic} if it is given by a presentation
$\mathcal F^c(E,R)$ as above, that is, 
if it is given as the subcooperad of $\mathcal F^c(E)$ just described.
Let $\mathcal P = \mathcal F(E,R)$ be a quadratic operad.
Its \emph{Koszul dual cooperad } is defined as 
$$\mathcal P^{\antishriek} = \mathcal F^c\left(sE, s^2R\right).$$ 
The \emph{canonical twisting morphism} is the degree $-1$ morphism of $\mathbb S$-modules
$\kappa:\mathcal P^{\antishriek} \to \mathcal P$
given by the composite
$$
\kappa: \mathcal F^c\left(sE, s^2R\right) \twoheadrightarrow sE \xrightarrow{s^{-1}} E \to \mathcal F\left(E, R\right).
$$
If $\mathcal P$ is augmented,
then we can functorially associate to it a quasi-cofree differential graded conilpotent cooperad $B\mathcal P$,  
called the \emph{bar construction} of $\mathcal P$. 
If $\mathcal P$ is quadratic, then it is naturally augmented, 
and the Koszul dual cooperad $\mathcal P^\antishriek$ is a subcooperad of $B\mathcal P$ with trivial differential.
The operad $\mathcal P$ is \emph{Koszul} if the inclusion $\mathcal P^\antishriek \hookrightarrow B\mathcal P$ is a quasi-isomorphism.
The cooperad $B\mathcal P$, being differential graded, has a homology cooperad $H_*\left(B\mathcal P\right)$. 
This homology admits an extra cohomological degree called the syzygy degree.
It can be seen that $\mathcal P$ is Koszul if, and only if, 
$H^0\left(B\mathcal P\right) \cong \mathcal P^\antishriek$ and all other cohomology groups vanish.
This follows from combining \cite[Prop. 7.3.1]{loday12} with \cite[Thm. 7.4.2]{loday12},
see also the Résumé \cite[Sec. 7.9.2]{loday12}. 
The assignment of a Koszul dual cooperad is functorial 
on weighted operads as long as the morphisms of operads preserve the weight.

\medskip
\noindent{\bf{$\mathcal P_\infty$-structures and Quillen homology. }}
In this section, we discuss several ways to present a $\mathcal P_\infty$-structure 
on a chain complex $A$ for a given  Koszul operad $\mathcal P$, and define the Quillen homology of a $\mathcal P$-algebra. 
A convenient choice of model for $\mathcal P_\infty$ is the cobar construction $\Omega \mathcal P^\antishriek$, 
where $\mathcal{P}^{\antishriek}$ is the Koszul dual cooperad of $\mathcal P$.
Recall that the cobar construction is the left adjoint of the bar construction $B$,
mapping onto the category of augmented differential graded operads.
A $\mathcal P_\infty$ structure on $A$ is therefore a morphism of 
differential graded operads $\Omega \mathcal P^\antishriek \to \operatorname{End}_A$, 
where $\operatorname{End}_A$ is the endomorphism operad of $A$.
Under this point of view, we can think of a $\mathcal P_\infty$-algebra structure on $A$ 
as a family of operations $\left\{A^{\otimes n}\to A\right\}$ parametrized by the operad $\Omega \mathcal P^\antishriek$.

By the \emph{Rosetta Stone Theorem} \cite[Theorem 10.1.13]{loday12}, 
an equivalent approach, and the one which we shall use in the rest of this document, is to define 
a \emph{$\mathcal{P}_\infty$-algebra} to be a chain complex 
$A$ along with a degree $-1$ square zero coderivation
$$
\delta: \mathcal{P}^{\antishriek}\left(A\right)\to \mathcal{P}^{\antishriek}\left(A\right).
$$
Briefly recall that if $A$ is a $\mathcal{P}$-algebra, 
then $\mathcal{P}^{\antishriek}\left(A\right)$ is a quasi-cofree $\mathcal{P}^{\antishriek}$-coalgebra
whose coderivation codifies the internal differential of $A$ as well as its $\mathcal{P}$-algebra structure.
The coderivation is meant as a $\mathcal{P}^{\antishriek}$-coalgebra,
and we explain next how to understand this.
Since it squares to zero, we might call it the codifferential of $\mathcal{P}^{\antishriek}\left(A\right)$.
It will often be convenient to present $\delta$ in two different ways. 
Firstly, as a collection of linear maps 
$\delta_r: \mathcal{P}^{\antishriek}(r)\otimes A^{\otimes r } \to A$,
for $r\geq 1,$
where each $\delta_r$ is the composition
\begin{equation*}
\mathcal{P}^{\antishriek}(r)\otimes_{\mathbb S_r} A^{\otimes r } 
\hookrightarrow \bigoplus_{k\geq 1}\mathcal{P}^{\antishriek}(k)\otimes_{\mathbb S_k} A^{\otimes k } 
= \mathcal{P}^{\antishriek}\left(A\right) \xrightarrow{\delta} \mathcal{P}^{\antishriek}\left(A\right) 
\xrightarrow{\epsilon_A} A.    
\end{equation*}
Here, $\epsilon$ is the counit of the $\mathcal{P}^{\antishriek}$ comonad. 
The coderivation $\delta$ can be reconstructed from the family $\left\{\delta_r\right\}_{r\geq 1}$ as the map
$$
\mathcal{P}^{\antishriek}\left(A\right) \xrightarrow{\Delta_{(1)}}
\left(\mathcal{P}^{\antishriek}\circ_{(1)}\mathcal{P}^{\antishriek}\right)\left(A\right) =
\mathcal{P}^{\antishriek}\circ
\left(A;  \mathcal{P}^{\antishriek}\left(A\right)\right)\xrightarrow{\idd\circ\left(\idd; m \right)} 
\mathcal{P}^{\antishriek}\circ \left(A;A\right)\to  \mathcal{P}^{\antishriek}\left(A\right). 
$$
Here, we are treating chain complexes such as $A$ as symmetric sequences concentrated in arity zero, 
and thus $\circ$ is the usual composite product of symmetric sequences, see \cite[\textsection 1.2]{Vallette2020}.
Similarly, $\circ_{(1)}$ is the infinitesimal composite product of symmetric sequences,
and $\Delta_{(1)}$ is the infinitesimal decomposition coproduct of $\mathcal{P}^{\antishriek}$, 
see \cite[\textsection 6.1]{loday12} for both notions. 
The map $m$ is
$(\delta_r)_{r\geq 1}:\bigoplus_{r\geq 1} \mathcal{P}^{\antishriek}(r)\otimes_{\mathbb S_r} A^{\otimes r } \to A$ 
induced by the universal property of the coproduct of the underlying graded vector spaces. 
Secondly, we can present $\delta$ as a collection of degree $-1$ linear maps
$\delta^{(n)}: \mathcal{P}^{\antishriek}(A)^{(n)} \to A$,
for $n\geq 1,$
where each $\delta^{(n)}$ is the composition
\begin{equation*}
\mathcal{P}^{\antishriek}(A)^{(n)}
\hookrightarrow \mathcal{P}^{\antishriek}\left(A\right) \xrightarrow{\delta} \mathcal{P}^{\antishriek}\left(A\right) 
\xrightarrow{\epsilon_A} A,    
\end{equation*}
and where $\mathcal{P}^{\antishriek}(A)^{(n)}$ consists of the weight $n$ part of $\mathcal{P}^{\antishriek}(A)$, 
\[
\mathcal{P}^{\antishriek}(A)^{(n)} = \bigoplus_{r\geq n} 
\left( \left( \mathcal{P}^{\antishriek}\right)^{(n)}(r)\otimes_{\mathbb S_r}A^{\otimes r} \right).
\]
To reconstruct  $\delta$ from the family
$\left\{\delta^{(n)}\right\}_{n\geq 1}$, 
one proceeds \emph{mutatis mutandis} as in the case of $\left\{\delta_r\right\}_{r\geq 1}$.

\medskip
The object $(\mathcal P^{\antishriek}(A), \delta)$ is called the \emph{operadic chain complex}.
The \emph{Quillen homology} of a $\mathcal{P}$-algebra $A$ 
is the homology $H_*\left(\mathcal P^\antishriek(A),\delta\right)$ 
of this operadic chain complex. 
It forms a (non-differential) graded $\mathcal P^\antishriek$-coalgebra.

\medskip

A $\mathcal{P}_\infty$-algebra $A$ is a \textit{strict} $\mathcal P$-algebra 
if the map $m$ factors through the canonical twisting morphism $\kappa : \mathcal{P}^\antishriek \to \mathcal P$. 
In other words,
if there is a commutative diagram 
$$
\begin{tikzcd}
\mathcal{P}^{\antishriek}\left(A\right) \arrow[dr, "m", swap] \arrow[r, "\kappa_A"] & 
\mathcal{P}\left(A\right) \arrow[d, dotted, "\gamma_A"]  \\
 &A
\end{tikzcd}
$$
Conversely, any $\mathcal P$-algebra $A$ can be seen as a $\mathcal P_\infty$-algebra by
pulling back its algebra structure along the morphism of operads $\Omega \mathcal{P}^{\antishriek} \to \mathcal P$.

\medskip

A \emph{$\mathcal{P}_\infty$-morphism} is a map of (dg) $\mathcal{P}^{\antishriek}$-coalgebras 
$F: \left(\mathcal{P}^{\antishriek}\left(A\right), \delta\right)\to \left(\mathcal{P}^{\antishriek}(B), \delta' \right)$.
As in the case of a codifferential on a $\mathcal{P}^{\antishriek}$-coalgebra, 
it will often be convenient to present $F$ as a collection of  linear maps 
$F_n: \mathcal{P}^{\antishriek}(n)\otimes A^{\otimes n } \to B$, for $n\geq 1$, 
where each $F_n$ is the composition
\[
\mathcal{P}^{\antishriek}(n)\otimes_{\mathbb S_n} A^{\otimes n } 
\hookrightarrow \bigoplus_{k\geq 1}\mathcal{P}^{\antishriek}(k)\otimes_{\mathbb S_k} A^{\otimes k } 
= \mathcal{P}^{\antishriek}\left(A\right) \xrightarrow{F} \mathcal{P}^{\antishriek}(B) \xrightarrow{\epsilon_B} B.
\]
The map $F$ can be reconstructed from the family $\left\{F_n\right\}_{n\geq 1}$ as the map
$$
\mathcal{P}^{\antishriek}\left(A\right) \xrightarrow{\Delta} \mathcal{P}^{\antishriek}\circ 
\mathcal{P}^{\antishriek}\left(A\right) \xrightarrow{\mathcal{P}^{\antishriek}\left(f\right)}  \mathcal{P}^{\antishriek}(B),
$$
where $f$ is the map 
$(F_i)_{i\geq 1} : \bigoplus_{i\geq 1}\mathcal{P}^{\antishriek}(i)\otimes_{\mathbb S_i} A^{\otimes i } \to B$ 
induced by the universal property of the coproduct. Similarly, 
we can decompose by weight instead of arity to produce a collection of degree $0$ linear maps 
$F^{(n)}: \mathcal{P}^{\antishriek}(A)^{(n)}\to B$.

\medskip
\noindent{\bf{The $\mathcal P$-Eilenberg--Moore spectral sequence. }}
\label{sec: construction of the EMSS}
Let $A$ be an algebra over a Koszul operad $\mathcal P$ and $H = H_\ast(A)$ be its homology.
There is a spectral sequence, 
which we call the \emph{$\mathcal P$-Eilenberg--Moore spectral sequence}, 
that computes the Quillen homology of $A$ 
as long as $A$ is positively graded of finite type 
(which is implicitly assumed whenever we speak of convergence). 
It is constructed as follows.
The operadic chain complex $\mathcal P^{\antishriek}(A)$ admits the ascending filtration
$$
F_p\mathcal P^{\antishriek}(A) = \bigoplus^{p}_{n=0}\mathcal P^{\antishriek}(A)^{(n)}.
$$
This filtration is bounded below and exhaustive.
Therefore, the associated spectral sequence converges to the operadic homology of $A$ as a graded module. 
The complex $\mathcal P^{\antishriek}(A)$ also has the 
structure of a conilpotent $\mathcal P^{\antishriek}$-coalgebra with comultiplication $\Delta$,
which respects the filtration in the sense that
$$
\Delta\left(F_p\mathcal P^{\antishriek}(A) \right) 
\subseteq \bigoplus_{\substack{k\geq 1, w, i_1, \dots i_k \geq 0 \\w+ i_1+\cdots+i_k \leq p }} 
\left(\mathcal P^{\antishriek}(k)\right)^{(w)}\otimes_{\mathbb S_k} 
\left(F_{i_1}\mathcal P^{\antishriek}(A)\otimes\cdots \otimes F_{i_k}\mathcal P^{\antishriek}(A)\right).
$$
This further implies that each page of the spectral sequence inherits a $\mathcal P^{\antishriek}$-coalgebra structure, 
and furthermore, the spectral sequence converges as a $\mathcal P^{\antishriek}$-coalgebra.
A morphism of $\mathcal P^\antishriek$-coalgebras naturally induces a morphism of the corresponding spectral sequences.
The $E^0$-page of this spectral sequence is explicitly given by
\[
E^0_{p,q} = \left(\mathcal P^\antishriek(A)^{(p)}\right)_{p+q} 
\cong \left(\bigoplus_{r\geq 1} \left(\mathcal{P}^\antishriek\right) ^{(p)}(r)
\otimes_{\mathbb S_r} A^{\otimes r}\right)_{p+q}
\]
where the $p+q$ grading is induced from the internal grading of $A.$ 
Under the isomorphism above, the differential $d^0$ is determined by the differential $d$ of $A$,
and there is an isomorphism of differential bigraded modules
\[
\left(E^0,d^0\right) \cong \left(\mathcal P^\antishriek(A),\delta_1\right),
\]
where $\delta_1$ stands for the coderivation of $\mathcal P^\antishriek(A)$ induced by
the arity 1 term of the codifferential $\delta$, that is, by the internal differential of $A$.
Taking homology of $\left(E^0,d^0\right)$, 
it follows that the $E^1$-page of the spectral sequence is
$$
E^1_{p,q} = \left(\mathcal P^\antishriek\left(H_*\left(A\right)\right)^{(p)}\right)_{p+q}
$$
and the differential on this page is therefore entirely determined by the weight 1 component of the codifferential.
In other words, $d^1= H_*\left(\delta^{(1)}\right)$.
Taking homology again, we finally have
\[
E^2_{p,q} = H_{p+q}\left(\mathcal P^\antishriek(H_\ast\left(A\right))^{(p)}\right) \;	
		\	\xRightarrow{\phantom{m}p\phantom{m}} \; \ H_*\left(\mathcal P^\antishriek(A),\delta\right).
\]
While this definition seems to be original to this paper for general operads, 
it has some very well-known special cases. 
When $\mathcal{P}$ is binary, that is, generated by operations of arity $2$,
the weight grading coincides with the arity grading up to a shift. 
So, for example, when $\mathcal P = \mathsf{Ass}$ is the associative operad, 
the $\mathcal P$-Eilenberg--Moore spectral 
sequence is exactly the classical Eilenberg--Moore spectral sequence \cite{Eil66}.  
When $\mathcal P=\mathsf{Lie}$ is the Lie operad, the $\mathcal P$-Eilenberg--Moore spectral sequence
is exactly a classical Quillen spectral sequence that appears in \cite[(6.9) p. 262]{Qui69}.

\medskip

\begin{remarks}
\begin{enumerate}
    \item If $A$ is an algebra over a Koszul operad $\mathcal P$, 
    there are several spectral sequences closely related to the 
 one defined above.
 First, we can filter $\mathcal P^\antishriek$ by weight. 
 This gives the spectral sequence we studied above.
 Second, we can filter $\mathcal P^\antishriek$ by arity.
 This produces a spectral sequence that coincides 
 with the previous one up to a shift  when the operad is binary generated,
 or more generally, when the generators of the operad are concentrated in a single arity.
 However, in general, these two spectral sequences differ.
 Third, one can replace $\mathcal P^\antishriek$ with the bar construction $B\mathcal P$
 and filter similarly.
 Since not every operad is Koszul,
 this spectral sequence will be useful in those situations.
 \item  If $A$ is a $\mathcal{P}_\infty$-algebra, 
 then the construction of the spectral sequence above goes through with straightforward adjustments.
\end{enumerate}
\end{remarks}

\medskip
\noindent{\bf{The half-reduced coproduct.}}
Let $\mathcal C$ be a counital reduced cooperad. 
There is a direct sum decomposition for each arity $r$ component 
\begin{equation}
\label{ecu: decomosition of the coproduct}
\left(\mathcal C \circ \mathcal C\right)(r) 
= \bigoplus_{k\geq 1} \mathcal C(k)\otimes_{\mathbb{S}_k}\left(\bigoplus_{i_1+\cdots + i_k=r}
\mathsf{Ind}^{\mathbb{S}_r}_{\mathbb{S}_{i_1}\times
\cdots \times \mathbb{S}_{i_k}}\left(\mathcal C(i_1)\otimes \cdots \otimes \mathcal C(i_k)\right)\right).    
\end{equation}
In this case, the decomposition map $\Delta : \mathcal C \to \mathcal C \circ \mathcal C$ can be written uniquely as 
$$\Delta(c) = \Delta^+(c) + (\id;c)$$
for every arity-homogeneous $c\in \mathcal C$. 
Here,  $\id\in \mathcal C(1)$ is the element that corresponds to the identity element $1$ of the ground field $\Bbbk$
under the linear isomorphism $\mathcal C(1)\to \Bbbk$ induced by the counit.
Since $\mathcal C$ is assumed reduced, if $c$ has arity $r$, then 
$\Delta^+(c)$ is the component of $\Delta(c)$ lying in the sum with $k\geq 2$ of the explicit decomposition written
in Equation \eqref{ecu: decomosition of the coproduct}.

We call $\Delta^+$ the \emph{half-reduced coproduct map} of $\mathcal C$.

\medskip
\noindent{\bf{The canonical planar tree representation of the free operad.}}
\phantomsection\label{sec: Free operad and planar trees}
The Massey product operations (defining systems, and indexing sets) developed in this paper 
make heavy use of the free operad $\mathcal F(E)$ on a reduced symmetric sequence $E$.
In particular, the canonical planar tree representation of the free operad will be crucial to prove some results
and will also make the constructions visually intuitive.
In this section, we briefly recall this topic. 

\smallskip

The precise formalism of symmetric tree monomials was introduced in the foundational paper \cite{Hoffbeck}.
Expanding upon it,  
the authors of \cite{Dot} provided a successful and celebrated theory of Gröbner bases for algebraic operads,
introducing shuffle operads.
There are some textbooks treating  these topics: 
an introductory account of Gröbner bases, symmetric tree monomials and shuffle operads appear in 
\cite{Bremner--Dotsenko};
a more advanced reference including applications is \cite[Chapter 8]{loday12};
and finally, a thorough treatment of symmetric trees appears in \cite[Appendix A]{Fresse1}.

\smallskip

Let $E$ be a reduced symmetric sequence, and fix a $\K$-linear basis of it.
Then there is an induced canonical basis of the free operad $\mathcal F(E)$.
A generic basis element is a planar tree with $r$ leaves labeled from $1$ to $r$
whose internal vertices are decorated by elements from the fixed basis of $E$.
For each arity $r$ component, there is a canonical linear isomorphism
\[
\mathcal F(E) (r) \cong \bigoplus_{T\in \mathcal T_r}\left(\bigotimes_{v\in V(T)} E(n_v)\right).
\]
Here,
\begin{itemize}
    \item $\mathcal T_r$ is a set of representatives of isomorphism classes 
    of (abstract) trees with $r$ leaves indexed from $1$ to $r$,
    where two trees are isomorphic if there are bijections 
    between the vertices and edges preserving the labeled tree structure,
    \item $V(T)$ is the set of internal vertices of $T$, i.e., the non-leaf and non-root vertices of $T$,
    \item $n_v$ is the number of incoming edges of $v$,
\end{itemize}
The canonical planar representation is determined once 
we explain how to order the incoming vertices of each vertex in a fixed tree $T \in \mathcal T_r$.
Recall the leaves are labeled by the set $\{1,\ldots,r\}$.
Let $v$ be a vertex of $T$, and let $I_v$ denote the set of incoming vertices of $v$, so that $n_v = |I_v|$.
To each vertex $u\in I_v$, assign the smallest integer label among the leaves linked to $u$ by a monotone path.
Then we order the vertices $u$ and the leaves adjacent to $v$ from left to right above $v$ in ascending order.
This way, there is a bijection between the set $\{1,\ldots,n_v\}$ and $I_v$.
In this paper, we will sometimes replace the set $\{1,\ldots,r\}$ of labels by a set with $r$ elements $\{k_1,\ldots,k_r\}$.

\medskip
\noindent{\bf{Operads indexed by finite sets.}}
\phantomsection\label{sec: operads indexed by finite sets}
The idea to formulate operads indexed
on finite sets rather than the standard finite sets $\{1,\ldots,r\}$ goes back to 
Joyal's theory of species \cite{Joyal1981} and Kelly's interpretation of Peter May's operads,
reprinted as \cite{Kelly2005}. 
Among others, it has proved useful in the celebrated Ginzburg--Kapranov paper on Koszul duality \cite{Gin94}, 
and the unpublished Getzler--Jones paper \cite{Getzler--Jones}.
A good modern reference is \cite[Sec. 2.5]{Fresse1}.
This approach will also prove fruitful in this paper. 
Despite the setup being quite notationally cumbersome, 
in practice it amounts to working with operations labeled with an arbitrary set of $r$ elements rather than $\{1, \dots r\}$.

\smallskip

Recall that an operad can be defined as a functor $\P : \operatorname{Fin} \to \Vect$
from the category of finite sets and bijections to vector spaces, equipped with a unit and composition 
product subject to the usual axioms.
A reduced conilpotent cooperad,
such as $\P^\antishriek$ when $\P$ is Koszul, admits an analogous description where we require it forms a comonoid.

We now, with an eye towards later applications in this paper, describe this cooperad indexed by finite sets. 
For every $r \geq 1$ and finite set $I = \{k_1, \dots, k_r\}$, 
the space $\P^\antishriek(I)$ is the $\K$-vector space obtained from the standard arity $r$ component $\P^\antishriek(r)$ 
by replacing the labels $1, \dots, r$ with the elements of $I$ via a bijection
\[
\{1, \dots, r\} \xrightarrow{\ \cong \ } I.
\]
The assignment $I \mapsto \P^\antishriek(I)$ is functorial.
The canonical planar representation of reduced labeled trees in $\P^\antishriek(r)$ 
produces a basis of $\P^\antishriek(I)$ where the leaves are now labeled by $k_1, \dots, k_r$ rather than $1, \dots, r$. 
The coproduct $\Delta$ and half-reduced coproduct
$\Delta^+$ of $\P^\antishriek$ extend in the obvious way to $\P^\antishriek(I)$.
We denote a generic basis element in $\P^\antishriek(I)$ by $(\Gamma^c, (k_1, \dots, k_r))$, 
where the tuple $(k_1, \dots, k_r)$ represents 
the bijection $i \mapsto k_i$ defining the labels of the leaves from left to right.
By construction, for every permutation $\sigma \in \mathbb S_r$ there is an identification
\begin{equation}
    \label{ecu: equivariant property}
    \left(\Gamma^c\cdot \sigma, \left(k_1,...,k_r\right)\right) 
= \left(\Gamma^c, \left(k_{\sigma^{-1}(1)},...,k_{\sigma^{-1}(r)}\right)\right).
\end{equation}

\medskip
\noindent{\bf{A version of the homotopy transfer theorem. }}
In \cite[Theorem 2]{petersen2020}, 
D. Petersen gave what probably is the most general form of T. Kadeishvili's version of the 
classical homotopy transfer theorem \cite{Kad80} for algebras over algebraic operads.
Adapted to our needs i.e., replacing arity by weight, it reads as follows. 
In the statement, $\mathcal P$ is a Koszul operad.

\begin{theorem}
\label{teo : Homotopy Transfer Theorem}
Let $(A, d)$ be a $\mathcal P$-algebra, $H$  its homology, 
and  $f : H \to A$ a cycle-choosing (and therefore necessarily degree $0$) linear map. 
Let $\delta_A$ be the degree $-1$ square-zero coderivation 
of $\mathcal{P}^{\antishriek}(A)$ representing the $\mathcal P$-algebra structure
on $A$ whose weight 0 term equals the given differential $d$. 
Then there exists noncanonically a square-zero degree $-1$ coderivation $\delta$ of $\mathcal{P}^{\antishriek}(H)$ 
whose weight 0 term vanishes, and a quasi-isomorphism of $\mathcal{P}^{\antishriek}$-coalgebras 
$F : \mathcal{P}^{\antishriek}(H ) \to \mathcal{P}^{\antishriek}(A)$ whose linear term $F_1$ is $f$ 
and which is a chain map with respect to the differentials defined by $\delta_A$ and $\delta$.
\end{theorem}

\begin{proof}[Sketch of the proof]
The homology $H$ is equipped with the structure of a $\mathcal P$-algebra descending from 
the $\mathcal P$-algebra structure on $A.$ 
This induces a degree $-1$ coderivation 
$\delta^0 : \mathcal{P}^{\antishriek}\left(H \right)\to \mathcal{P}^{\antishriek}\left(H\right)$ 
whose weight $0$ component $\delta^0_0$ is identically 0.
Now, by induction, assume that for some $n \geq 1$, 
we have a degree $-1$ coderivation 
$\delta^{n-1} : \mathcal{P}^{\antishriek}\left(H \right)\to \mathcal{P}^{\antishriek}\left(H\right)$ and 
a $\mathcal{P}^{\antishriek}$-coalgebra 
morphism $F^{n-1} : \mathcal{P}^{\antishriek}\left(H\right) \to \mathcal{P}^{\antishriek}(A )$ with $F_1 = f$,
such that the restrictions
of $\delta^{n-1}$ and $F^{n-1}$ to $F_{n-1}\mathcal{P}^{\antishriek}\left(H \right)$  satisfy 
$$
\begin{cases}
\delta^{n-1}\circ \delta^{n-1} =0 \\
F^{n-1}\circ \delta^{n-1} - \delta_A \circ F^{n-1} = 0.
\end{cases}
$$
Above, $\circ$ denotes the usual composition of maps, not the operadic circle product.
Write $F^0$ for the coalgebra map determined by $f$ in weight 0 and vanishing in higher weights.
Since $f$ takes values in cycles, $\delta^0$ and $F^0$ satisfy the identities above
on $F_0\mathcal{P}^{\antishriek}\left(H\right)$, providing the base case $n=1$ of the induction.
The idea now is to modify only the weight $n$ terms of $\delta^{n-1}$ and $F^{n-1}$
to produce new $\delta^n$ and $F^{n}$ such that the equations above are satisfied on 
$F_n\mathcal{P}^{\antishriek}(H)$. 
One can show that there are $e$ and $e'$ such that 
$$
\left(F^{n-1}\circ \delta^{n-1} -\delta_A \circ F^{n-1}\right)_{n} = f\circ e + de'
$$
where $e\in \operatorname{Hom}\left(\mathcal P^{\antishriek}(H)^{(n)}, H\right)$
and $e' \in \operatorname{Hom}\left(\mathcal P^{\antishriek}(H)^{(n)}, A\right)$.
Therefore, we can define 
$$\delta^n_i = \begin{cases}
\delta^{n-1}_i &\mbox{for } i\neq n. \\
\delta^{n-1}_n - e &\mbox{for } i = n.
\end{cases}$$ 
In fact, $e$ may be computed as the projection of 
$\left(F^{n-1}\circ \delta^{n-1} -\delta_A \circ F^{n-1}\right)_{n}$ onto $H$.  
Similarly, we can define $F_n^n$ to be
$$F^n_i = \begin{cases}
F^{n-1}_i &\mbox{ if } i\neq n. \\
F_n^{n-1}+e' &\mbox{when }  i = n.
\end{cases}$$ 
So defined, the coderivation $\delta^n$ and the coalgebra map 
$F^n$ satisfy the required conditions, and the proof is complete.
\end{proof}

\subsubsection{Higher-order Massey products for associative algebras}
\label{sec: Higher Massey products for associative algebras}

The triple Massey product for differential graded associative algebras was introduced in the fifties, 
see \cite{Uehara57}
and \cite{Mas69} (reprinted as \cite{Mas98}).
Massey himself soon realized that the triple product could be extended to $n$-fold Massey products \cite{Mas58},
see also \cite{May69}.
Our generalization of the higher-order Massey products 
to algebras over algebraic operads  
has its roots in this definition.
Therefore, we find it convenient to devote this section to recall
the higher-order Massey products for differential graded associative algebras.
Excellent references for this topic include \cite{KrainesMasseyproducts,May69,Ravenel}.

\medskip

Let $(A,d)$ be a differential graded associative algebra, 
and $x_1,x_2\in H_*(A)$ homogeneous elements. 
The Massey product $\langle x_1,x_2\rangle$ is defined as the singleton $\{x_1x_2\}$ formed by the product 
of the two classes in $H_*(A)$.
It is also possible to identify the set $\{x_1x_2\}$ with the product $x_1x_2$ itself
and define the Massey product of two homogeneous elements in homology as their ordinary product.
Let us define next the triple and higher-order Massey products.
First, we introduce the auxiliary notion of a defining system.
A defining system 
in the case of the Massey product of two homology classes $\langle x_1,x_2\rangle$ is just a choice $\{b_1,b_2\}$
of cycle representatives of $x_1$ and $x_2$.

\begin{definition}
\label{Def : associative higher Massey products} 
Let $(A,d)$ be a differential graded associative algebra, 
and  $x_1,{\dots},x_n$ be $n\geq 3$ 
homogeneous elements in $H_*(A)$.
A \emph{defining system for the $n^{th}$-order Massey product of the classes $x_1,\ldots,x_n$} 
is a set of homogeneous elements
\[
\left\{b_{ij} \right\} \subseteq A, \quad 
\textrm{ for } \quad 0 \leq i < j \leq n \quad \textrm{ and }\quad  1 \leq j-i \leq n-1,
\]
defined as follows.
\begin{itemize}
	\item (Initial step) For $i=1,{\dots},n$ the element $b_{i-1,i}$ is a cycle representative of $x_i$. 
	\item (Inductive relation) 
	For each $0 \leq i < j \leq n$  and $1 \leq j-i \leq n-1$, 
the element $b_{ij}\in A$ satisfies
 \begin{equation}
     \label{ecu: b_ij in defining system}
     d\left(b_{ij}\right)=\sum_{i<k<j} (-1)^{|b_{ik}|+1} b_{ ik }b_{kj}.
 \end{equation}
\end{itemize}
The \emph{$n^{th}$-order Massey product of the classes $x_1,\ldots,x_n$} is  the set 
\[
\left\langle x_1,{\dots},x_n \right\rangle
=\left\{\left[\sum_{0<k<n} (-1)^{|b_{0k}|+1} b_{ 0k}b_{kn} \right] \ 
\mid\  \left\{b_{ij}\right\} \textrm{ is a defining system}\right\}\subseteq H_{s+n-2}(A),
\]
 where $s=\sum_{i=1}^n |x_i|$, and the bracket $[-]$ denotes taking homology class.
\end{definition}

The elements $b_{ij}$ of Equation (\ref{ecu: b_ij in defining system}) might not exist at all,
in which case the Massey product set is empty.
The necessary and sufficient condition for $\left\langle x_1,{\dots},x_n \right\rangle$ to be non-empty
is that for all $1\leq i < j \leq n$ and $1\leq j-i\leq n-2$, 
the Massey product sets $\langle x_i,\ldots,x_j \rangle$ are non-empty and furthermore contain the zero class
in a coherent manner.

The fact that for a fixed defining system the sum 
$$\sum_{0<k<n} (-1)^{|b_{0k}|+1} b_{ 0k}b_{kn}$$
defines a cycle is a straightforward check by applying $d$ and using the inductive relations.
If there are no defining systems for the classes $x_1,\ldots,x_n$,
their Massey product $\left\langle x_1,{\dots},x_n \right\rangle$ is defined as the empty set,
and it is said to be undefined.

\medskip

A similar definition for higher Lie--Massey brackets on the homology of a differential graded Lie algebra exists, 
see \cite{Allday73,Allday77,Ret77,jose19,Tan83}.
The main purpose of this paper
is to provide a suitable generalization of Definition \ref{Def : associative higher Massey products}
to algebras over Koszul operads, see Section \ref{sec : higher-order operadic Massey products}.

\subsubsection{Secondary Massey products for algebras over algebraic operads}
\label{sec: secondary Massey products for algebras over operads}

In this section, 
we briefly outline Muro's definition of secondary Massey products for algebras over algebraic operads.
Our eventual definition of Massey products for algebras over operads,
Def. \ref{def : Higher-order Massey products for algebras over algebraic operads},
recovers Muro's whenever the inner indices of the relation are pairwise distinct,
and refines it in general. See Proposition \ref{prop : Our products coincide with Muro's}.

\begin{definition} 
{({\cite[Def. 2.1]{Mur21}})}
\label{Def: Massey product of Muro}
Let $\mathcal P = \mathcal F(E,R)$ be a Koszul operad generated by the reduced symmetric sequence $E$
with quadratic relations $R \subseteq \mathcal F(E)^{(2)}$.
Fix 
    \begin{equation*}
        \Gamma=\sum\left(\mu^{\left(1\right)}\circ_k\mu^{(2)}\right)\cdot\sigma
    \end{equation*}
a relation of arity $r$ of $R$.
Here, $\mu^{(i)}\in E(r_i)$, with $r_1+r_2=r+1$, the symbol $\circ_k$ denotes the $k$-th partial composition product,  
$1\leq k\leq r_1$, and $\sigma\in\mathbb{S}_r$.
Let $A$ be a $\mathcal P$-algebra and let $x_1,\ldots,x_r\in H_{*}(A)$ be homogeneous elements
such that
    \begin{equation}
    \label{ecu : vanishing}
        \mu^{(2)}\left(x_{\sigma^{-1}(k)},\dots, x_{\sigma^{-1}(k+r_2-1)}\right)=0
    \end{equation}
in $H_*(A)$ for each term in the relation. 
For each $1\leq i\leq r$, 
fix $y_i\in A$ a cycle representative 
of $x_i$ and, 
for each summand in the relation, 
let $\rho^{(2)}\in A$
be an element such that
\begin{equation}
\label{ecu : Condition on boundary of Fernando's}
        d\rho^{(2)}=\mu^{(2)}\left(y_{\sigma^{-1}(k)},\dots, y_{\sigma^{-1}(k+r_2-1)}\right)  
\end{equation}
in $A$. 
Such an element exists by Equation (\ref{ecu : vanishing}).
The $\Gamma$-\emph{Massey product} set $\langle x_1,\dots,x_r\rangle_\Gamma$
is the set of homology classes 
represented by cycles of the form
\begin{equation*}
        \sum(-1)^{\gamma}\mu^{\left(1\right)}
        \left(y_{\sigma^{-1}(1)},\dots,y_{\sigma^{-1}(k-1)},\rho^{(2)},y_{\sigma^{-1}(k+r_2)},
        \dots,y_{\sigma^{-1}(r)}\right),
\end{equation*}
where
    \[
    \gamma  =\alpha+ |\mu^{\left(1\right)}| \ +\left(|\mu^{(2)}|-1\right)
    \sum_{m=1}^{k-1} |x_{\sigma^{-1}(m)}|,
    \qquad \alpha  =\sum_{ \substack{i<j \\ \sigma(i) > \sigma(j)}}|x_i||x_j|.
    \]
for all possible choices of elements $\rho^{(2)}$.
\end{definition}

Muro shows that the definition above recovers the usual triple Massey products for differential graded
associative algebras when $\Gamma$ is the associativity relation of the associative operad,
and the triple Lie--Massey brackets for differential graded Lie algebras when 
$\Gamma$ is the Jacobi relation of the Lie operad.

The perspective we take to construct higher-order Massey products for algebras over algebraic operads 
differs significantly from the construction of Muro just explained.
Muro uses the form of relations defined using partial composition. 
The definition does not depend exclusively on the relation $\Gamma$,
but also on a specific choice of expansion of $\Gamma$ in terms of the partial compositions.
This choice is not unique, as it depends on the choice of a basis.
Our approach is also affected by a choice in the explicit form of the higher relations. 
To generalize, 
we prefer to see such relations as the weight 2 cooperations in the Koszul dual cooperad of $\mathcal P$,
and work with defining systems in a similar way as in 
Definition \ref{Def : associative higher Massey products}. 
This makes our formulas easier to write in the usual language of algebraic operads and Koszul duality theory.
To take into account the dependency of the higher relations on a presentation, 
we will assume all through that a $\Bbbk$-linear basis of the symmetric sequence $E$ has been fixed,
and then there is a canonical basis on $\mathcal F^c\left(sE\right)$ given by planar trees,
as briefly recalled in Section \ref{sec: Free operad and planar trees}.
We show in Proposition \ref{prop : Our products coincide with Muro's} that 
the secondary case of our definition refines Muro's, 
and coincides with it whenever the inner indices of the relation are pairwise distinct.

\section{Higher-order operadic Massey products}
\label{sec : higher-order operadic Massey products}

In this section,
we define higher-order Massey products for algebras over algebraic operads.
We focus on the case of Koszul operads 
and indicate in Remark \ref{remark: Massey products for non-Koszul operads} how to deal with the non-Koszul case.
We recommend familiarity with the classical higher-order Massey products for differential graded 
associative algebras recalled in Section \ref{sec: Higher Massey products for associative algebras}.

\medskip

Let $\mathcal P = \mathcal F(E,R)$ be a Koszul operad 
with Koszul dual cooperad $\mathcal P^\antishriek = \mathcal F^c(sE,s^2R)$.
We will assume all through the paper that a $\Bbbk$-linear basis of $E$ has been fixed.
We furthermore fix once and for all a $\K$-linear basis of $\mathcal P^{\antishriek}$,
\[
\mathcal B \ = \ \bigsqcup_{w,r}\mathcal B^{(w)}(r),
\]
where 
$\mathcal B^{(w)}(r)\subseteq
\left(\mathcal P^{\antishriek}\right)^{(w)}(r)$, with
$\mathcal B^{(0)}(1) = \{\idd\}$ and $\mathcal B^{(1)} = s\mathcal
B_E$, where $\mathcal B_E$ is the fixed basis of $E$.
The canonical planar representation of Section
\ref{sec: Free operad and planar trees} will be used to
describe and to normalize elements of $\mathcal F^c(sE)$, but every
label appearing in the constructions below is an element of this basis, 
and therefore a cooperation of $\mathcal P^{\antishriek}$.
There is a standard basis of $\mathcal P^{\antishriek}\circ\mathcal P^{\antishriek}$ induced by $\mathcal B$,
which we also fix once and for all.
The elements in this basis are classes
\[
\left(\theta;\zeta_1,\dots,\zeta_k;\sigma\right),
\]
where $\theta\in\mathcal B(k)$, $\zeta_j\in \mathcal B(i_j)$, and $\sigma$ runs over the unshuffles of type
$(i_1, \cdots , i_k)$ whose blocks are ordered in the sense that the minima of the $k$ blocks of
leaf labels appear in increasing order. 
All linear expansions in this paper are taken with respect to this basis. 
For example, the half-reduced coproduct of a cooperation $\mu\in\mathcal P^{\antishriek}$
expands uniquely as 
\begin{equation}
\label{ecu: B expansion}
\Delta^+(\mu)
\ = \
\sum_{T} \lambda_T
\left(\theta_T;\zeta_{T,1},\dots,\zeta_{T,m_T};\sigma_T\right),
\end{equation}
where $\lambda_T$ is a scalar, and similarly for the half-reduced infinitesimal coproduct 
$\Delta_{(1)}(\mu) = \Delta^+_{(1)}(\mu) + (\idd;\mu)$.

We say that the basis $\mathcal B$ is \emph{monomial} if the symmetric groups permute it up to non-zero scalars,
that is, if for every $\mu\in\mathcal B(r)$ and $\sigma\in\mathbb S_r$
there are $\mu'\in\mathcal B(r)$ and $c\in\Bbbk^{\times}$ such that $\mu\cdot\sigma = c\,\mu'$.
For instance, the bases of $\mathsf{Ass}^{\antishriek}$ and $\mathsf{Lie}^{\antishriek}$
used in Examples \ref{Ex:AssIndexingSet} and \ref{Ex:LieIndexingSet} are monomial.
Only the elements of $\mathcal B$ of weight smaller than that of a cooperation $\Gamma^c$
play a role in the indexing set and in the defining systems of $\Gamma^c$ defined below,
so, for a given $\Gamma^c$, it suffices that these elements satisfy this condition.
The equivariance properties of indexing sets and Massey products below require it,
see Remark \ref{rem: equivariance needs monomial bases}.

As mentioned in the introduction, 
each weight-homogeneous cooperation $\Gamma^c$ of $\mathcal P^\antishriek$ of arity $r$
creates a partially defined higher-order operation $\langle -,\ldots,-\rangle_{\Gamma^c}$ 
on the homology of any $\mathcal P$-algebra.
The cooperation determines an indexing set $I\left(\Gamma^c\right)$,
defining systems are coherent families of elements of $A$ indexed by $I\left(\Gamma^c\right)$, 
and the $\Gamma^c$-Massey product set collects the homology classes of the cycles associated to all defining systems.

\medskip

The section is organized as follows.
First,
we introduce the \emph{Massey inductive map}.
This map depends on the coproduct of $\mathcal P^\antishriek$ and 
a fixed twisting morphism $\kappa:\mathcal P^\antishriek \to \mathcal P$.
It is an essential ingredient when dealing with the inductive definitions that follow.
We also give an alternative, non-inductive equivalent definition.
Then, we define the indexing set $I\left(\Gamma^c\right)$ associated to an arbitrary cooperation $\Gamma^c$ 
and compute some examples.
Once the concept of indexing sets is established,
we proceed to explain what a defining system is
and give examples of them.
Then, we define the higher-order $\Gamma^c$-Massey products,
and compute examples
including the associative, commutative, Lie, Poisson, and dual numbers operads.
Later on,
we show that our higher-order Massey products framework recovers 
Muro's secondary Massey products \cite{Mur21} whenever the inner indices of the relation are pairwise distinct,
and refines them in general (Prop. \ref{prop : Our products coincide with Muro's}).
We study the elementary properties of these higher-order products in
Section \ref{sec: Elementary properties of the operadic Massey products}.
These include the behavior along morphisms of $\mathcal P$-algebras,
quasi-isomorphisms, and some connections to formality.
Some further  properties are explored in 
Section \ref{sec : Massey products along morphisms of operads and formality}.
There, we focus on the behavior of the higher-order Massey products 
along morphisms of operads and give some applications to formality.

\medskip

We are now going to define a map $D$ on the cofree conilpotent cooperad on the suspension of a given symmetric sequence $E$. 
This map will define a map on its subcooperad $\P^{\antishriek}$ by restriction.

\begin{definition}
The \emph{Massey inductive map} is the degree $-1$ map
$$
D:\mathcal F^c\left(sE\right) \xrightarrow{\Delta}
\mathcal F^c\left(sE\right)\circ\mathcal F^c\left(sE\right)
\xrightarrow{\kappa\circ \operatorname{id}} E \circ \mathcal F^c\left(sE\right).
$$
\end{definition}
This map can be explicitly described in terms of the fixed basis $\mathcal B$.
To do so, let $\mu\in\mathcal P^{\antishriek}$ be a cooperation of arity $r$,
and expand 
\[
\Delta^+(\mu)
\ = \
\sum_{T} \lambda_T
\left(\theta_T;\zeta_{T,1},\dots,\zeta_{T,m_T};\sigma_T\right)
\]
as in \eqref{ecu: B expansion}.
Applying $\kappa\circ \operatorname{id}$ kills every term whose
root label $\theta_T$ has weight different from $1$, and desuspends
the root label of the remaining terms, giving
\begin{equation}
\label{Ecu: linear expansion of D}
D\left(\mu\right) = \sum_{T\,:\,w(\theta_T)=1} \lambda_T
\left(\zeta_T;\zeta_{T,1},\dots,\zeta_{T,m_T};\sigma_T\right),
\qquad \zeta_T := \kappa\left(\theta_T\right) = s^{-1}\theta_T\in \mathcal B_E.
\end{equation}
Since $\kappa$ is injective on the weight $1$ component, no cancellation
occurs in passing to \eqref{Ecu: linear expansion of D}, which is therefore a basis
expansion in the induced basis of
$E \circ \mathcal P^{\antishriek}$.
Moreover, $D$ is injective on elements of positive weight:
in the tree monomial basis of $\mathcal F^c(sE)$, the component of $\Delta$ whose root has weight $1$
is inverse to grafting, and $\kappa:sE\to E$ is an isomorphism.
In particular $D(\mu)\neq 0$ whenever $\mu\neq0$ has weight at least $1$.
The map $D$ is inductive in the sense that every block
$\zeta_{T,i}$ of every term of $D(\mu)$ is an element of
$\mathcal B$ of weight strictly smaller than that of $\mu$.
This will allow us to establish the inductive relations of our
defining systems later on.
In terms of the ambient planar trees of $\mathcal F^c(sE)$, $D$ is trimming level $1$ edges,
see Figures \ref{Fig1} and \ref{Fig2}.
The blocks of each term carry the leaf labels of the corresponding
planar subtrees with the canonical planar order of Section \ref{sec: Free operad and planar trees}. 
This way, the block containing the smallest label comes first.

\begin{notation}
\label{not: shorthand D}
Given a tuple of labels $(k_1,\dots,k_r)$, write $v_{T,i}$ for the arity of
$\zeta_{T,i}$ and
\[
K_{T,i}:=\bigl(k_{\sigma_T^{-1}(v_{T,1}+\cdots+v_{T,i-1}+1)},\dots,
k_{\sigma_T^{-1}(v_{T,1}+\cdots+v_{T,i})}\bigr)
\]
for the labels carried by the leaves of $\zeta_{T,i}$. 
Later on, we drop the index $T$ and simply write the expansion of~\eqref{Ecu: linear expansion of D} as
\[
D(\mu)=\sum\lambda\,(\zeta;\zeta_1,\dots,\zeta_m;\sigma).
\]
Here, $v_i$ and $K_i$ denote the arity and the tuple of leaf labels of $\zeta_i$.
\end{notation}

Next, we introduce the set associated with a cooperation of
$\mathcal P^\antishriek$ that will provide the indices for our defining systems. 
It is defined by induction on the weight of arity-homogeneous cooperations of $\mathcal P^{\antishriek}$, 
with the Massey inductive map providing the necessary inductive step.
This is reminiscent of how classical Massey products of a given order are built in terms of lower order ones.
We will use the notation for operads indexed by finite sets of Section \ref{sec: operads indexed by finite sets}.
Recall that  a weight-homogeneous cooperation $\Gamma^c\in \mathcal{P}^{\antishriek}(r)$ is implicit for 
\[
\Gamma^c = \left(\Gamma^c, (1,\ldots, r)\right).
\]

\begin{definition}
\label{def : indexing sets}
Let $\Gamma^c\in \mathcal{P}^{\antishriek}(r)$ be a
weight-homogeneous cooperation and let $(k_1,\ldots,k_r)$ be a
tuple of pairwise distinct labels.
The \textit{$\Gamma^c$-indexing set}
$I\left(\Gamma^c,\left(k_1,\ldots,k_r\right)\right)$ is a set of
pairs $\left(\mu,\left(l_1,\dots,l_m\right)\right)$, where
$\mu\in\mathcal B$ has arity $m$ and
$(l_1,\dots,l_m)$ is a tuple of pairwise distinct labels from
$\{k_1,\dots,k_r\}$, defined by induction on the weight
$w\left(\Gamma^c\right)$ as follows.
\begin{itemize}
    \item If $w\left(\Gamma^c\right) = 0$, then
    $I\left(\Gamma^c,\left(k_1,\ldots,k_r\right)\right)
    = \emptyset$.
    \item If $w\left(\Gamma^c\right) = 1$, then
    $I\left(\Gamma^c,\left(k_1,\ldots,k_r\right)\right) =
    \left\{\left(\idd,(k_1)\right),\ldots,
    \left(\idd,(k_r)\right)\right\}$.
    \item If $w\left(\Gamma^c\right) \geq 2$, write
\[
D\left(\Gamma^c\right) = \sum_{T} \lambda_T
    \left(\zeta_T;\zeta_{T,1},\dots,\zeta_{T,m_T};
    \sigma_T\right)
\]
     as in \eqref{Ecu: linear expansion of D}, and let $K_{T,i}$ be the tuple
    of labels carried by the leaves of the block $\zeta_{T,i}$,
    as in Notation \ref{not: shorthand D}.
    Then
    \[
    I\left(\Gamma^c,\left(k_1,\ldots,k_r\right)\right)
    := \left\{\left(\idd,(k_j)\right)\right\}_{1\leq j\leq r}
    \ \cup\
    \bigcup_{T}\ \bigcup_{i=1}^{m_T}
    \Bigl(
    \left\{\left(\zeta_{T,i},K_{T,i}\right)\right\}
    \cup I\left(\zeta_{T,i},K_{T,i}\right)
    \Bigr),
    \]
    where the outer union runs over all terms $T$ of \eqref{Ecu: linear expansion of D}.
\end{itemize}
\end{definition}

The superscript $c$ in $\Gamma^c$ indicates that we are
seeing the corresponding element in the Koszul dual \emph{cooperad} of $\mathcal P$.
At a later place, 
we will see this same element as a relation $\Gamma$ of $\P$ in the free operad $\mathcal F(E)$.
Since we will need to distinguish between these two elements,
we keep the superscript in the notation.

\begin{remark}
As pointed out by a referee, one may compare the indexing set $I\left(\Gamma^c\right)$ 
with the set of second composition factors in $\Delta^+_{(1)}\left(\Gamma^c\right)$ whose root is different from $\idd$.
Every such factor lies in $I\left(\Gamma^c\right)$, 
but the converse does not seem to hold in general because cancellations may occur between the expansions 
attached to different terms of $D\left(\Gamma^c\right)$.
We leave the details of this claim to the interested reader. 
\end{remark}

The equivariant properties of the free operad $\mathcal F(E)$ at each arity induce the following property of indexing sets.

\begin{proposition}
\label{prop: equivariance of indexing sets}
Assume that $\mathcal B$ is monomial.
Let $\Gamma^c \equiv \left(\Gamma^c, \left(k_1,...,k_r\right)\right)\in \P^\antishriek \left(\{k_1,...,k_r\}\right)$ 
be a weight-homogeneous cooperation.
For every permutation $\sigma \in \mathbb S_r$, the indexing sets
\[
I\left(\left(\Gamma^c\cdot \sigma, \left(k_1,...,k_r\right)\right)\right)
\quad \textrm{and} \quad
I\left(\left(\Gamma^c, \left(k_{\sigma^{-1}(1)},...,k_{\sigma^{-1}(r)}\right)\right)\right)
\]
coincide up to non-zero scalars.
That is, under the identification of Equation \eqref{ecu: equivariant property},
every index of each of them is a non-zero multiple of an index of the other one.
\end{proposition}

\begin{proof}
By Equation \eqref{ecu: equivariant property},
the two cooperations are the same element of $\P^\antishriek\left(\{k_1,\dots,k_r\}\right)$,
so they have the same image under the Massey inductive map $D$.
The two indexing sets are read off from the expansions of this image
in the bases of $E\circ\mathcal P^{\antishriek}$ attached to the orders
$\left(k_1,\dots,k_r\right)$ and $\left(k_{\sigma^{-1}(1)},\dots,k_{\sigma^{-1}(r)}\right)$ of the labels.
Passing from one order to the other reorders the labels inside each block and the blocks of each term,
so it acts on the root and on the blocks by permutations.
Since $\mathcal B$ is monomial, the two bases agree up to non-zero scalars,
and the claim follows by induction on the weight.
\end{proof}

\begin{figure}[ht!]
\centering
		\includegraphics[scale=0.3]{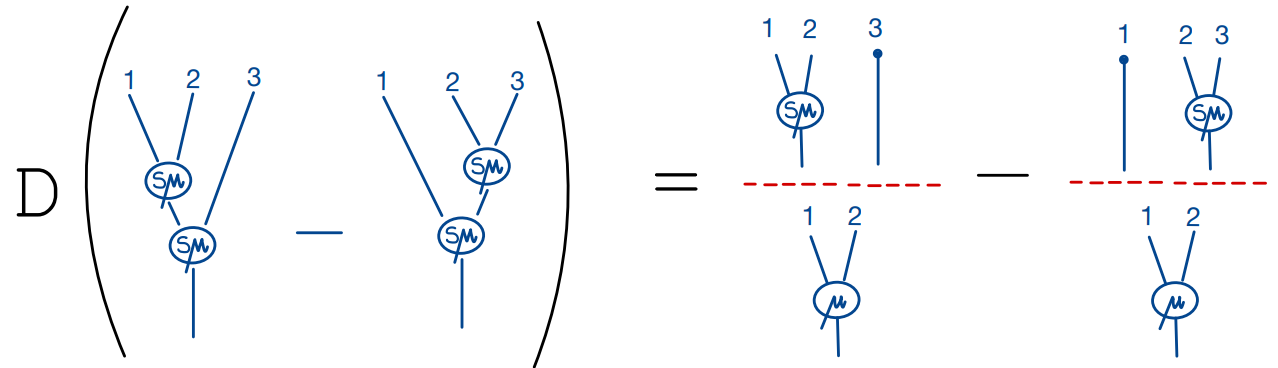}
  \caption{The Massey inductive map for $\mathsf{Ass}$}
 \label{Fig1}
\end{figure}

\begin{figure}[ht!]
\centering
	\includegraphics[scale=0.44]{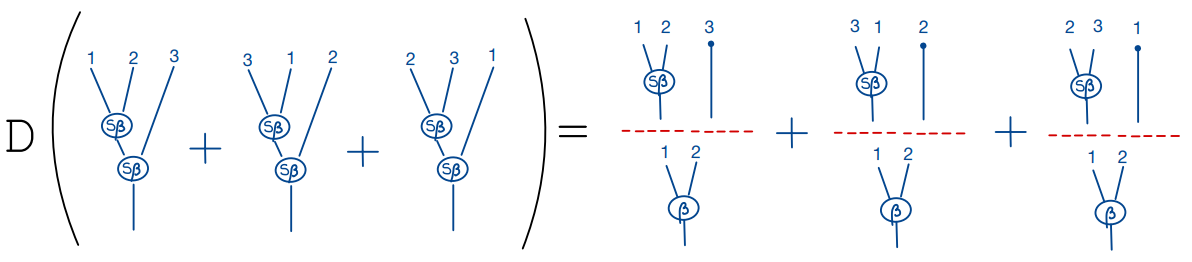}
\caption{The Massey inductive map for $\mathsf{Lie}$}
\label{Fig2}
\end{figure}

\medskip

If $\mathcal P$ is any Koszul operad and $\Gamma^c \in \left(\mathcal{P}^\antishriek\right)^{(1)}(r)=\left(sE\right)(r)$ is 
any cogenerator of arity $r$, then the $\Gamma^c$-indexing set is always given by 
$$I\left(\Gamma^c\right)  = \left\{ (\id,(1)),\ldots,(\id,(r))\right\}.$$
This is the base case of the inductive definition of defining systems below. 
We will illustrate the definition of indexing sets with some examples before continuing forward.

\begin{example}
\label{Ex:AssIndexingSet}
Let $\mathcal P = \mathsf{Ass}$.
Then the weight $n$ component of $\mathcal P^{\antishriek}$ 
is freely generated as an $\mathbb{S}_{n+1}$-module by a single generator $\mu^c_{n+1}\in \mathsf{Ass}^{\antishriek}(n+1)$. 
Recall that
$$
\Delta\left(\mu_n^c\right) = \sum_{i_1+\cdots+i_k = n} 
(-1)^{\sum(i_j+1)(k-j)}\left(\mu_k^c; \mu_{i_1}^c, \dots \mu_{i_k}^c; \idd\right).
$$
Here, we denote $\mu_1^c = \id  \in \mathsf{Ass}^{\antishriek}(1)$. 
Since $\kappa \left(\mu_2^c \right)= \mu_2$ 
and $\kappa\left(\mu_k^c\right)= 0$ for $k\geq 3,$ 
this implies that 
$$
D\left(\mu_n^c\right) = \sum_{i_1+i_2 = n} (-1)^{i_1+1}\left(\mu_2; \mu_{i_1}^c, \mu_{i_2}^c; \idd\right).
$$
This means that the indexing set  $I\left({\mu^c_n}\right)$ contains the elements 
$$
(\mu^c_{i_1}, (1, 2,\dots, i_1))
\quad \textrm{and} \quad 
(\mu^c_{i_2}, ( n-i_2 +1,\dots, n)),
$$
where $i_1+i_2=n$. 
By iterating this process, we see that 
$$
I\left({\mu^c_n}\right) 
= \left\{\left(\mu^c_{k}, (i, i+1, \dots, i+k-1)\right) \mid k < n \mbox{ and } 
i\in \{1, \ldots, n-k+1\} \right\}.
$$
\hfill$\square$
\end{example}

\begin{example}
\label{Ex:LieIndexingSet}
Let $\mathcal P = \mathsf{Lie}$.
Then the weight $n$ part of $\mathcal P^{\antishriek}$ is one-dimensional 
and generated by $\tau^c_{n+1}\in \mathsf{Lie}^{\antishriek}(n+1)$. Recall that
$$
\Delta\left(\tau_n^c\right) 
= \sum_{\substack{i_1+\cdots+i_k = n \\ \sigma \in \overline{Sh}^{-1}\left(i_1,\dots, i_k\right)}}
(-1)^{\sum(i_j+1)(k-j)}\operatorname{sgn}(\sigma)\left(\tau_k^c; \tau_{i_1}^c, \dots \tau_{i_k}^c;\sigma\right),
$$
where $\overline{Sh}^{-1}\left(i_1,\dots, i_k\right)$ is the set of reduced unshuffles. 
Here, an unshuffle is the inverse of a shuffle
and \emph{reduced} 
signifies that we are considering shuffles satisfying
\[
\sigma(1) < \sigma(i_1+1) < \cdots < \sigma(i_1+\cdots+i_{k-1}+1).
\]
For $k=2$ this reduced condition is precisely equivalent to requiring that $\sigma(1)=1$.

Since $\kappa(\tau_2^c)= \tau_2$ and $\kappa(\tau_k^c)= 0$ for $k\geq 3,$ this implies that 
$$
D\left(\tau_n^c\right) 
= \sum_{\substack{i+j = n\\ \sigma\in \overline{Sh}^{-1}\left(i,j\right)}} 
(-1)^{i+1}\operatorname{sgn}(\sigma)\left(\tau_2; \tau_{i}^c, \tau_{j}^c; \sigma\right).
$$
This means that the indexing set $I\left({\tau^c_n}\right)$ contains the elements 
$$
\left(\tau^c_{i}, \left(\sigma(1), \sigma(2),\dots, \sigma(i)\right)\right)
\quad \textrm{and} \quad 
\left(\tau^c_{j}, \left(\sigma(n-j+1), \ldots, \sigma(n)\right)\right)
$$
for each reduced shuffle $\sigma\in \overline{Sh}\left(i,j\right)$ with $i+j=n$.
In this step, we changed from using unshuffles to shuffles, 
because there is an inversion involved. 
By iterating this process, we find that 
$$
I\left({\tau^c_n}\right) = \left\{\left(\tau^c_{k}, \left(i_1, \dots i_k\right)\right) \mid k<n 
\mbox{ and }1\leq i_1 < \cdots < i_k\leq n \right\}.
$$
\hfill$\square$
\end{example}

As mentioned before, 
each cooperation $\Gamma^c$ of weight $n$ in the Koszul dual cooperad $\mathcal P^\antishriek$  of $\mathcal P$
produces a partially defined $n$-th order operation $\langle -,\ldots,-\rangle_{\Gamma^c}$ 
on the homology $H_*(A)$ of a $\mathcal P$-algebra $A$.
This higher operation has $r$ inputs, where $r$ is the arity of $\Gamma^c$,
and the output is the set of homology classes created from all possible choices of \emph{defining systems},
generalizing the case of associative algebras of 
Section \ref{sec: Higher Massey products for associative algebras}.
Our next task is to explain what the defining systems are.
Each defining system will depend on a weight-homogeneous cooperation $\Gamma^c$ of arity $r$
and $r$ homogeneous homology classes $x_1,\ldots,x_r \in H_*(A)$.
Their definition is given by induction on the weight of the cooperation.

\begin{definition}
\label{def: Defining systems for operadic Massey products}
Let $\Gamma^c\in \left(\mathcal P^{\antishriek} \right)^{(n)} (r)$ for some $n\geq 1$,
$A$ a $\mathcal P$-algebra, and $x_1,\ldots,x_r\in H_*(A)$ homogeneous elements.
A \textit{$\Gamma^c$-defining system (associated to $x_1,\ldots,x_r$)} is 
a collection $\left\{a_\alpha\right\}_{\alpha \in I\left(\Gamma^c\right)}$
of elements of $A$ indexed by $I\left(\Gamma^c\right)$
such that:
\begin{enumerate}
    \item Each
    $a_{\idd, (i)}\in A$ is a representative cycle     of $x_i$.
    \item For each index $\left(\mu,\left(k_1, \ldots, k_i\right)\right) \in I\left(\Gamma^c\right)$ with $\mu \neq \idd$,
the corresponding element $a_{\left(\mu,\left(k_1, \ldots, k_i\right)\right)}$ is such that 
\[
    d\left(a_{\mu,\left(k_1, \ldots, k_i\right)}\right)
    = \sum_T \lambda_T\,(-1)^{\gamma_T}\,
    \zeta_T\left(a_{\zeta_{T,1},K_{T,1}},\dots,a_{\zeta_{T,m_T},K_{T,m_T}}\right),
\]
where the sum runs over the terms of \eqref{Ecu: linear expansion of D},
the tuples $K_{T,j}$ are those of Notation \ref{not: shorthand D},
and the sign is given by
\[
\gamma_T = \sum_{\substack{1\leq p<q \leq i \\ \sigma_T(p) > \sigma_T(q)}} |x_{k_p}||x_{k_q}| \,
+ \, \sum_{j=2}^{m_T} \left|\zeta_{T,j}\right|
\left(\sum_{l=1}^{v_{T,1}+\cdots+ v_{T,j-1}} \left|x_{k_{\sigma_T^{-1}(l)}}\right|\right).
\]
\end{enumerate}
\end{definition}

An induction on the weight shows that
$\left|a_{\mu,\left(k_1,\dots,k_i\right)}\right| = |\mu| + |x_{k_1}| + \cdots + |x_{k_i}|$
for every index $\left(\mu,\left(k_1,\dots,k_i\right)\right)$ of $I\left(\Gamma^c\right)$,
hence the cycle \eqref{ecu : Massey cycle} below has degree $|\Gamma^c| + |x_1| + \cdots + |x_r|-1$.

Next, we use the defining systems explained above to define the $\Gamma^c$-Massey products.
If the cooperation $\Gamma^c$ is of weight $1$ and arity $r$,
that is, a cogenerator,  then $\Gamma = \kappa(\Gamma^c)$ is a generator of $\mathcal P.$
For any homogeneous elements $x_1,\ldots,x_r\in H_*(A)$,
we define their \emph{$\Gamma^c$-Massey product} as the set 
$$
\langle x_1,\ldots,x_r \rangle_{\Gamma^c} := \left\{ \Gamma(x_1,\ldots,x_r) \right\}.
$$
If $\Gamma^c$ is the zero cooperation of any weight and arity $r$, 
then $\langle x_1,\ldots,x_r \rangle_{\Gamma^c}$ is the empty set.
 We may also identify this set with its unique element $\Gamma(x_1,\ldots,x_r ) \in H_*(A)$.
Let us define the $\Gamma^c$-Massey products for elements of weight $\geq 2$.

\begin{definition}
\label{def : Higher-order Massey products for algebras over algebraic operads}
Let $A$ be a $\mathcal P$-algebra,
$\Gamma^c \in \left(\mathcal P^{\antishriek} \right)^{(n)} (r)$ with $n\geq 2$,
and $x_1,\ldots,x_r$ homogeneous elements of $H_*(A)$.
Then:
\begin{enumerate}
    \item The \emph{$\Gamma^c$-Massey product} associated to a $\Gamma^c$-defining system 
$\left\{a_\alpha\right\}$ and $x_1,\ldots,x_r$ 
is the homology class of the cycle
\begin{equation}
    \label{ecu : Massey cycle}
    a_{\Gamma^c, (1,\dots, r)}:=\sum_T \lambda_T\,(-1)^{\gamma_T}\,
    \zeta_T\left( a_{\zeta_{T,1},K_{T,1}},\dots,a_{\zeta_{T,m_T},K_{T,m_T}} \right),
\end{equation}
where the sum runs over the terms of the expansion
\eqref{Ecu: linear expansion of D} of $D\left(\Gamma^c\right)$,
the tuples $K_{T,i}$ are those of Notation \ref{not: shorthand D} for
$(k_1,\dots,k_r)=(1,\dots,r)$, and the sign is given by
\[
\gamma_T = \sum_{\substack{1 \leq p<q \leq r\\ \sigma_T(p) > \sigma_T(q)}}
|x_p||x_q| \, + \, \sum_{j=2}^{m_T} \left|\zeta_{T,j}\right|
\left(\sum_{l=1}^{v_{T,1}+\cdots+ v_{T,j-1}} \left|x_{\sigma_T^{-1}(l)}\right|\right).
\]

    \item The \emph{$\Gamma^c$-Massey product set} $\langle x_1, \dots, x_r \rangle_{\Gamma^c}$
is the (possibly empty) subset of $H_*(A)$ formed by the homology classes arising from all possible choices of 
$\Gamma^c$-defining systems $\left\{a_\alpha\right\}$ associated to $x_1,\ldots,x_r$.
\end{enumerate}
\end{definition}

If $\mathcal B$ is monomial, then, for any permutation $\sigma \in \mathbb S_r$,
it follows from Equation \eqref{ecu: equivariant property}
and a direct comparison of the signs of Equation \eqref{ecu : Massey cycle} for the two sides 
that the Massey product sets satisfy the following equivariant relation:
\[
\langle x_1, \dots, x_r \rangle_{\Gamma^c \cdot \sigma} 
= (-1)^{\varepsilon_\sigma}\, \langle x_{\sigma^{-1}(1)}, \dots, x_{\sigma^{-1}(r)} \rangle_{\Gamma^c}, 
\]
where 
\[
\varepsilon_\sigma = \sum_{\substack{1 \leq i<j \leq r\\ \sigma(i) > \sigma(j)}}|x_i||x_j|. 
\]
See Proposition \ref{prop : Elementary properties of the operadic Massey products} for the proof,
and Remark \ref{rem: equivariance needs monomial bases} for a counterexample when $\mathcal B$ is not monomial.

\smallskip

The next result shows that the proposed definition is correct.
As a consequence of it, 
we readily see from the definition of defining systems that the
$\Gamma^c$-Massey product set $\langle x_1,\ldots,x_r \rangle_{\Gamma^c}$ is non-empty 
if, and only if, for all $\left(\mu, \left(k_1,\dots k_i\right)\right) \in I(\Gamma^c)$ with $\mu \neq \id$,  
the Massey product set $\langle x_{k_1}, \dots, x_{k_i} \rangle_{\mu}$
is defined and contains the zero class in a coherent manner.

\begin{proposition}
\label{Prop: cycle}
Let $A$ be a $\mathcal P$-algebra,
$\Gamma^c \in \left(\mathcal P^{\antishriek} \right)^{(n)} (r)$ for some $n\geq 2$,
and $x_1,\ldots,x_r$ homogeneous elements of $H_*(A)$.
Then the $\Gamma^c$-Massey product $x$ associated to any $\Gamma^c$-defining system for $x_1,\ldots,x_r$ is a cycle.
\end{proposition}

\begin{proof}
Let $\{a_\alpha\}$ be a defining system,
and denote by $x$ the associated cycle given by Equation (\ref{ecu : Massey cycle}),
\begin{equation*}
    x=\sum \lambda\,(-1)^\gamma\,\zeta\left( a_{\zeta_1,K_1},\dots,a_{\zeta_m,K_m} \right).
\end{equation*}
Let us compute $dx$ in terms of the Massey inductive map $D$ and terms of the form $a_{\mu, \left(k_1,\dots k_i\right)}$. 
Recall that the differential of $A$ fits into the commutative diagram
\begin{center}
    \begin{tikzcd}
\mathcal P\circ A \arrow[d, "\operatorname{id}\circ' d", swap] \arrow[rr,"\gamma_A"]&& A \dar{d} \\
\mathcal P\circ A \arrow[rr,"\gamma_A"]&& A
\end{tikzcd}
\end{center}
where $\circ'$ is the infinitesimal composite. 
From here, it follows that 
\begin{align*}
dx &= d\left(\sum \lambda\,(-1)^{\gamma}\,\zeta\left( a_{\zeta_1,K_1},\dots,a_{\zeta_m,K_m} \right)\right)\\
&= \sum \sum_{i=1}^m \lambda\,(-1)^{\gamma+\epsilon_i}\,\zeta\left( a_{\zeta_1,K_1},\dots,
d\left(a_{\zeta_i,K_i}\right),\dots, a_{\zeta_m,K_m}\right),
\end{align*}
where 
\[
\epsilon_i = |\zeta|+ |a_{\zeta_1,K_1}|+\cdots + |a_{\zeta_{i-1},K_{i-1}}|.
\]
Each term $d\left(a_{\zeta_i,K_i}\right)$
appearing in the sum above can be rewritten in terms of $a_{\mu, \left(k_1,\dots k_s\right)}$ of lower weight,
by using the inductive relation of the defining system 
(Def \ref{def: Defining systems for operadic Massey products}, item 2).
Abbreviating the entries of such sums by $a_{\zeta_1},\dots,a_{\zeta_m}$,
we have
\begin{multline*}
\zeta\left(a_{\zeta_1},\dots,a_{\zeta_{i-1}}, \zeta'\left(a_{\eta_1},
\dots,a_{\eta_k}\right), a_{\zeta_{i+1}},\dots, a_{\zeta_m}\right)\\
= (-1)^{\left|\zeta'\right|\left(\left|a_{\zeta_1}\right|+\cdots+\left|a_{\zeta_{i-1}}\right|\right)}
\left(\zeta\circ_i \zeta'\right)\left(a_{\zeta_1},\dots, a_{\eta_1},\dots,a_{\eta_k},\dots, a_{\zeta_m}\right).
\end{multline*}
In particular, if we consider the composite
\begin{multline*}
G:\mathcal P^{\antishriek}\xrightarrow{\Delta}\mathcal P^{\antishriek}\circ \mathcal P^{\antishriek}
\xrightarrow{\kappa \circ\operatorname{id}} \mathcal P \circ 
\mathcal P^{\antishriek} \xrightarrow{f} \mathcal P \circ 
\left(\mathcal P^{\antishriek}; \mathcal P^{\antishriek}\right)
\xrightarrow{\operatorname{id}\circ \left(\operatorname{id}; \Delta\right)} 
\mathcal P \circ(\mathcal P^{\antishriek}; \mathcal P^{\antishriek}\circ 
\mathcal P^{\antishriek})  
\xrightarrow{\operatorname{id}\circ \left(\operatorname{id}; \kappa\circ \operatorname{id}\right)}\\
\mathcal P \circ\left(\mathcal P^{\antishriek}; \mathcal P\circ \mathcal P^{\antishriek}\right) \xrightarrow{p}
\mathcal P \circ\left(\mathcal P^{\antishriek}; \mathcal P^{\antishriek}\right)
\xrightarrow{q}
\mathcal P\circ \mathcal P^{\antishriek},
\end{multline*}
where $f$ is the choice of a distinguished factor of $\P^{\antishriek}$ in the composite $\P\circ \P^{\antishriek}$, 
$p$ is induced by the partial composition in $\mathcal P$, 
and $q$ is the forgetful map,
then the computation above exhibits the element $dx$ as the image of $G\left(\Gamma^c\right)$
under the linear map that evaluates each tree monomial of the expansion
$G(\Gamma^c) = \sum \lambda \left(\xi ; \xi_1,\dots,\xi_m; \sigma\right)$
against the defining system,
\[
\left(\xi ; \xi_1,\dots,\xi_m; \sigma\right)
\longmapsto
(-1)^\gamma \,\xi\left( a_{\xi_1, \left(\sigma^{-1}\left(1\right),\dots, 
\sigma^{-1}\left(v_1\right)\right)},\dots,a_{\xi_m,\left(\sigma^{-1}\left(v_1+\cdots+ v_{m-1}+1\right),\dots, 
\sigma^{-1}\left(r\right)\right)} \right),
\]
up to a sign which is the same for all terms.
Here, $\gamma$ is given as in Definition \ref{def: Defining systems for operadic Massey products} for
$\left(\xi ; \xi_1,\dots,\xi_m; \sigma\right)$.

To prove the result,
it suffices to show that $G$ is identically 0. 
We shall do this via the following observation: 
recall that $\mathcal P^{\antishriek}$ can be thought of as a subset of the
tree module and all the maps defining $G$ descend from
maps on or between the free operad on $E$ and the cofree conilpotent cooperad on $sE$.  
It thus suffices to show that $ \operatorname{Im}G\subseteq R\circ \mathcal P^{\antishriek}$,
where $R$ is the sub-module of relations in the fixed presentation $\mathcal P = \mathcal F(E,R)$. 

It follows that we may describe $G$ combinatorially by giving 
its action on individual basis tree monomials $T$ of $\mathcal{F}^c(sE)$. 
This action is as follows, signs play no role in this description.
\begin{enumerate}
    \item 
    First, apply the Massey inductive map $D$. 
    This is sending the tree monomial $T$ to a sum of tree monomials of the form
    $\left(s^{-1}e; T_1, \dots T_m\right)$, 
    where $e\in (sE)(m)$ is the label of the first vertex and $T_1,\dots, T_m$ are its children.
    \item
    Now, repeat this procedure on each $T_i$ individually,
    thereby obtaining sums of tree monomials of the form $\left(s^{-1}e_i; T_{i,1}, \dots T_{i,m_i}\right)$,  
    and take for each individual tree monomial $T_i$ the sum over the results to obtain
    \[
    \sum_{i=1}^m \left(s^{-1}e; T_1, \dots,\left(s^{-1}e_i; T_{i,1}, \dots T_{i,m_i}\right),\dots  T_m\right).
    \]
    Here, each $e_i$ is the first vertex of the corresponding $T_i$.
    \item
    The final step is to apply the partial composition in the free operad $\mathcal F(E)$ in order to obtain 
    $$
    \sum_{i=1}^m \left(s^{-1}e\circ _i s^{-1}e_i ; T_1, \dots, T_{i-1}, T_{i,1}, \dots T_{i,m_i}, T_{i+1 },\dots  T_m\right).
    $$
    This time, without the suspension.
\end{enumerate}

From this description, it follows that there is another equivalent way to describe $G$:

\begin{itemize}
    \item 
    First, directly apply the cooperadic reduced decomposition map $\Delta$ to $T$ to obtain
    \[
    \Delta(T) = \sum \left(S; S_1, \dots  S_k\right). 
    \]
    \item
    Then, project the first component of $\mathcal{F}^c\left(sE\right) \circ \mathcal{F}^c\left(sE\right)$ into weight 2. 
    That is, map $S$ to itself if it is in weight 2, and map it to 0 otherwise.
    This produces
    \[
    \sum_{i=1}^m \left(e\circ _i e_i ; T_1, \dots, T_{i-1}, T_{i,1}, \dots T_{i,m_i}, T_{i+1 },\dots  T_m\right).
    \]
    \item Desuspend the tree monomial $e\circ _i e_i$ twice. 
\end{itemize}
From this description, it follows that 
\[
\sum_{i=1}^m \left(e\circ _i e_i ; T_1, \dots, T_{i-1}, T_{i,1}, \dots T_{i,m_i}, T_{i+1 },\dots
T_m\right) \in \left(\mathcal P^{\antishriek}\right)^{(2)}\circ \mathcal P^{\antishriek},
\]
and thus that
\[
    \sum_{i=1}^m \left(s^{-1}e\circ _i s^{-1}e_i ; 
    T_1, \dots, T_{i-1}, T_{i,1}, \dots T_{i,m_i}, T_{i+1 },\dots  T_m\right) \in R\circ \mathcal P^{\antishriek}.
\]
This is exactly what we wanted to prove.
\end{proof}

In the next result, we show that our higher-order Massey products
recover the secondary Massey products defined by Muro in \cite{Mur21}
when restricting to cooperations of weight 2 whose inner indices are pairwise distinct,
and refine them in general.
The construction of Muro is recalled in Section \ref{sec: secondary Massey products for algebras over operads},
and we stick to the notation used there.

\begin{proposition}
    \label{prop : Our products coincide with Muro's}
    Let $\mathcal P$ be a Koszul operad with fixed presentation $\mathcal F(E,R)$.
    Let 
    \[
    \Gamma=\sum\left(\mu^{\left(1\right)}\circ_l\mu^{(2)}\right)\cdot\sigma\in R(r)
    \]
    be a quadratic relation of arity $r$,
    written as its expansion in the canonical basis of $\mathcal F(E)$ induced by $\mathcal B_E$:
    each $\mu^{(2)}$ lies in $\mathcal B_E$,
    each $\mu^{(1)}$ is a non-zero multiple of an element of $\mathcal B_E$,
    each $\sigma$ is an unshuffle with ordered blocks as in the basis of
    $\mathcal P^{\antishriek}\circ\mathcal P^{\antishriek}$ fixed at the beginning of this section,
    so that each summand is in canonical planar form,
    and no two summands are proportional.
    Denote the corresponding weight $2$ element of the Koszul dual cooperad $\mathcal P^{\antishriek}$ by $\Gamma^c$,
    so that
    \[
    \Gamma^c := s^2\left(\Gamma\right) = \sum (-1)^{|\mu^{(1)}|}    
    \left(s\mu^{\left(1\right)}\circ_ls\mu^{(2)}\right)\cdot\sigma.
    \]
    Let $A$ be a $\mathcal P$-algebra, and let $x_1,\ldots,x_r\in H_*(A)$ be homogeneous elements.
    Then the $\Gamma$-Massey product set $\langle x_1,\dots,x_r\rangle_\Gamma$ of  
     Def. \ref{Def: Massey product of Muro}
    and the $\Gamma^c$-Massey product set $\langle x_1,\dots,x_r\rangle_{\Gamma^c}$ of 
    Def. \ref{def : Higher-order Massey products for algebras over algebraic operads} 
    are related with the following inclusion,
        \[
        \langle x_1,\dots,x_r\rangle_{\Gamma^c} \subseteq \langle x_1,\dots,x_r\rangle_\Gamma 
        \]
and one is non-empty if and only if the other is.
Moreover, if the inner indices
\[
\left(s\mu^{(2)},\left(\sigma^{-1}(l),\dots,\sigma^{-1}(l+r_2-1)\right)\right)\in \mathcal P^{\antishriek}(\{1,\dots,r\})
\]
of the summands of the expansion are pairwise distinct, then the inclusion is an equality.
\end{proposition}
\begin{remark}
    In every example that we consider in the remainder of this paper, 
    the inner indices will be pairwise distinct. 
    In examples where this fails, 
    the $\Gamma^c$-Massey product set may be strictly smaller than the $\Gamma$-Massey product set of \cite{Mur21},
    so that our secondary Massey products may be finer invariants.
        The hypothesis on the expansion of $\Gamma$ cannot be dropped.
    Let $\mathcal P = \mathsf{Ass}$ with the basis $\left\{\mu,\mu+\mu\cdot(12)\right\}$ of $E(2)$,
    and let $\Gamma = \left(\mu\circ_1\mu\right)\cdot(12) - \left(\mu\circ_2\mu\right)\cdot(12)$,
    which is not of this form.
    In the associative algebra of the discussion following the proof below,
    $\langle x_1,x_2,x_3\rangle_\Gamma = \{0\}$,
    but $\langle x_1,x_2,x_3\rangle_{\Gamma^c}$ is empty.
    \end{remark}

\begin{proof}
By the hypothesis on the expansion of $\Gamma$,
the terms on the right-hand side below are non-zero multiples of pairwise distinct elements
of the basis of $\mathcal P^{\antishriek}\circ\mathcal P^{\antishriek}$,
and one can directly verify that, up to primitive terms,
\[
\Delta\left(\Gamma^c\right) = \sum (-1)^{|\mu^{(1)}|}(s\mu^{(1)}; 
\idd, \dots, \idd, s\mu^{\left(2\right)},\idd, \dots, \idd; \sigma).
\]
Since $\mu^{(1)}$ has weight 1, 
it follows that $\kappa(s\mu^{(1)}) = \mu^{(1)}$.
Therefore, a cycle representing the $\Gamma^c$-Massey product associated to the elements $x_1, \dots x_r$ is of the form
\begin{multline*}
\sum \lambda (-1)^\gamma \mu^{(1)}\left(a_{\idd, \left(\sigma^{-1}\left(1\right)\right)}, 
\dots, a_{\idd, \left(\sigma^{-1}\left(l-1\right)\right)},\right.\\
\left. a_{s \mu^{\left(2\right)}, 
\left(\sigma^{-1}\left(l\right),\dots,\sigma^{-1}\left(l+r_2-1\right)\right)},
a_{\idd, \left(\sigma^{-1}\left(l+r_2\right)\right)} , \dots, a_{\idd, \left(\sigma^{-1}\left(r\right)\right)}  \right),
\end{multline*}
as in Equation \eqref{ecu : Massey cycle}.
Note that summands of the expansion sharing the same inner index
\[
\left(s\mu^{(2)},\left(\sigma^{-1}(l),\dots,\sigma^{-1}(l+r_2-1)\right)\right)
\]
contribute the same element of the defining system, 
this is not the case in the setup of Def. \ref{Def: Massey product of Muro}, where they are chosen independently .
On each term, the permutation is $\sigma$ and the only contribution to the sign comes from the second block,
$s\mu^{(2)}$, sitting in position $l$,
so the sign of Equation \eqref{ecu : Massey cycle} evaluates to
\[
\gamma = \alpha + \left|s\mu^{(2)}\right| \sum_{t=1}^{l-1}\left|x_{\sigma^{-1}(t)}\right|
= \alpha + \left(\left|\mu^{(2)}\right|+1\right)\sum_{t=1}^{l-1}\left|x_{\sigma^{-1}(t)}\right|,
\qquad 
\alpha  =\sum_{ \substack{p<q \\ \sigma(p) > \sigma(q)}}\left|x_p\right|\left|x_q\right|.
\]
Now, the $a_{\idd,(i)}$ are just cycle representatives of the $x_i$,
so they correspond to the cycle representatives $y_i$ in Muro's construction (Def. \ref{Def: Massey product of Muro}).
To finish,
we check that the element 
$a_{s \mu^{\left(2\right)}, \left(\sigma^{-1}\left(l\right),\dots,\sigma^{-1}\left(l+r_2-1\right)\right)}$ 
satisfies exactly Condition (\ref{ecu : Condition on boundary of Fernando's})
in Muro's construction, with $k = l$ in the notation used there,
so it corresponds to the element $\rho^{(2)}$.
Indeed, since $s\mu^{\left(2\right)}$ has weight 1, it follows that 
$\Delta\left(s\mu^{\left(2\right)}\right) 
= \left(s\mu^{\left(2\right)};\idd,  \ldots, \idd;\idd\right) + \left(\id; s\mu^{(2)}\right)$,
and so $D\left(s\mu^{\left(2\right)}\right) = \left(\mu^{\left(2\right)};\idd,  \dots, \idd ;\idd\right)$. 
On this term, the sign of the inductive relation of Definition 
\ref{def: Defining systems for operadic Massey products} vanishes,
because the permutation is trivial and every block has degree zero.
Therefore, 
\[
da_{s \mu^{\left(2\right)}, \left(\sigma^{-1}\left(l\right),\dots,\sigma^{-1}\left(l+r_2-1\right)\right)} 
= \mu^{\left(2\right)}\left(a_{\idd, \left(\sigma^{-1}\left(l\right)\right)}, 
\dots, a_{\idd, \left(\sigma^{-1}\left(l+r_2-1\right)\right)}\right),
\]
which is Condition (\ref{ecu : Condition on boundary of Fernando's}).
Finally, taking into account the coefficient $\lambda =(-1)^{|\mu^{(1)}|}$ in the expansion of $\Gamma^c$ above,
the total sign of each term of the cycle is
\[
\alpha + \left|\mu^{(1)}\right| + \left(\left|\mu^{(2)}\right|+1\right)\sum_{t=1}^{l-1}\left|x_{\sigma^{-1}(t)}\right|
\equiv
\alpha + \left|\mu^{(1)}\right| +
\left(\left|\mu^{(2)}\right|-1\right)\sum_{t=1}^{l-1}\left|x_{\sigma^{-1}(t)}\right| \pmod 2,
\]
which is precisely the sign of Def. \ref{Def: Massey product of Muro}.
Setting $y_i := a_{\idd,(i)}$,
every $\Gamma^c$-defining system gives a $\Gamma$-defining system, which gives the desired inclusion.
To see the non-emptiness,
a $\Gamma$-defining system provides a primitive for every inner index,
and any choice of one primitive per index is a $\Gamma^c$-defining system, 
so the two sets must be non-empty simultaneously.
If the inner indices of the summands are pairwise distinct, 
the assignment above is a bijection between the two kinds of defining systems, and so the sets will coincide. 
This concludes the proof.
\end{proof}

Definition \ref{def : Higher-order Massey products for algebras over algebraic operads} 
depends on the fixed basis of $\mathcal P^\antishriek$.
This dependence is also present in \cite[Def. 2.1]{Mur21},
where it is addressed by fixing a concrete linear expansion for the relations,
which helps mitigate its effects.
Such a linear expansion is not unique, 
and the Massey product defined in loc. cit. depends on this choice.
In fact, a change of basis may change the Massey product sets, and even whether they are defined.
For instance, let $\mathcal P = \mathsf{Ass}$, 
and let $\Gamma^c = s^2\left(\mu\circ_1\mu - \mu\circ_2\mu\right)\cdot(12)$, 
which corresponds to the relation $(x_2x_1)x_3 - x_2(x_1x_3)$.
With respect to the basis $\left\{\mu, \mu\cdot(12)\right\}$ of $E(2)$,
the Massey product set $\langle x_1,x_2,x_3\rangle_{\Gamma^c}$ is defined
if, and only if, $x_2x_1$ and $x_1x_3$ vanish in homology.
With respect to the basis $\left\{\mu, \mu+\mu\cdot(12)\right\}$,
it is defined if, and only if, $x_1x_2$, $x_2x_1$ and $x_1x_3$ vanish in homology.
In the associative algebra $\Bbbk\langle x_1,x_2,x_3\rangle/(x_2x_1,x_1x_3)$,
with zero differential and generators of degree $0$,
the first Massey product set is $\{0\}$, while the second one is empty.
In particular, the Massey products created by different choices of basis need not be in bijection.
All the Massey products in this paper are taken with respect to the fixed basis.

\smallskip

In the following examples, 
we explain how our operadic framework for defining systems 
recovers the classical framework in the associative and Lie cases,
and then explain how it creates completely new higher-order operations.

\begin{example}
\label{example : Operadic def systems for Ass}
When $\mathcal P = \mathsf{Ass}$ is the associative operad,
our framework recovers the classical definition of higher-order Massey products as
in Def. \ref{Def : associative higher Massey products},  
up to an explicit global sign determined below.
To see this, 
recall from Example \ref{Ex:AssIndexingSet}
that 
the weight $n$ component of $\mathsf{Ass}^{\antishriek}$ 
is freely generated as an $\mathbb{S}_{n+1}$-module by a single generator $\mu^c_{n+1}$,
and that the $\mu^c_{n}$-indexing set attached to a cooperation is given by
\[
\left\{\left(\mu^c_{k}, (i, i+1, \dots, i+k-1)\right)\  \mid \ 1\leq k < n \textrm{ and } 
i\in \left\{1, \ldots, n-k+1\right\}  \right\}.
\]
We show next that fixing a particular differential graded associative algebra $(A,d)$ 
and homogeneous homology classes $x_1,\ldots,x_n \in H_*(A)$,
there is a bijective correspondence between the classical 
defining systems $\{b_{ij}\}$ for the higher-order Massey product $\langle x_1,\ldots,x_n\rangle$,
and the defining systems $\{a_\alpha\}$ for the $\mu_n^c$-Massey product 
$\langle x_1,\ldots,x_n\rangle_{\mu_n^c}$ as defined in this paper.
Indeed, the correspondence is given by 
\[
b_{i,j} = (-1)^{(j-i-1) \ + \ \sum_{t=i+1}^{j}\left(j-t\right)\left|x_t\right|}\ a_{\mu^c_{j-i},
\left(i+1,i+2,\dots, j\right)},
\qquad 0 \leq i < j \leq n, \quad  1 \leq j-i \leq n-1.
\]
The sign is $+1$ when $j = i+1$,
so the initial steps of the two definitions agree:
both $b_{i-1,i}$ and $a_{\idd, (i)}$ are cycle representatives of $x_i$.
To finish, it suffices to 
compare the boundaries of the elements in these sets.
Recall that
\[
D\left(\mu_r^c\right) = \sum_{l_1+l_2 = r} (-1)^{l_1+1}\left(\mu_2; \mu_{l_1}^c, \mu_{l_2}^c; \idd\right),
\]
where each $\mu^c_r$ is normalized so that its expansion in tree monomials contains the right comb 
$s\mu_2\circ_2\left(s\mu_2 \circ_2\left( \cdots \circ_2 s\mu_2\right)\right)$ with coefficient $+1$.
On the term of $D\left(\mu^c_{j-i}\right)$ with $\left(l_1, l_2\right) = \left(k-i, j-k\right)$,
the permutation is trivial and the only block of nonzero degree is the second one,
of degree $|\mu^c_{j-k}| = j-k-1$,
so that, by directly applying Definition \ref{def: Defining systems for operadic Massey products}, we see that
\[
d a_{\mu^c_{j-i}, \left(i+1,i+2,\dots, j\right)}
= \sum_{k=i+1}^{j-1} (-1)^{(k-i+1) \ +\  \left(j-k-1\right)\left(\left|x_{i+1}\right|+\cdots + \left|x_{k}\right|\right)}\\ 
a_{\mu^c_{k-i},(i+1,i+2,\ldots,k)}\cdot a_{\mu^c_{j-k},(k+1,k+2,\ldots,j)}.
\]
Multiplying by the sign of the correspondence and rewriting each factor in terms of the $b$'s,
the coefficient of $b_{ik}b_{kj}$ in $db_{ij}$ has exponent
\begin{multline*}
(j-i-1) + \sum_{t=i+1}^{j}\left(j-t\right)\left|x_t\right|
+ (k-i-1) + \sum_{t=i+1}^{k}\left(k-t\right)\left|x_t\right|\\
+ (j-k-1) + \sum_{t=k+1}^{j}\left(j-t\right)\left|x_t\right|
+ (k-i+1) + \left(j-k-1\right)\sum_{t=i+1}^{k}\left|x_t\right|.
\end{multline*}
The coefficient of $\left|x_t\right|$ in this expression is $2(j-t)-1 \equiv 1$ for $i < t\leq k$
and $2\left(j-t\right)\equiv 0$ for $k< t \leq j$,
while the constant part is congruent to $k-i$ modulo 2.
Since 
\[
|b_{ik}| = |a_{\mu^c_{k-i},(i+1,i+2,\ldots,k)}| = |x_{i+1}| + |x_{i+2}|+\cdots + |x_{k}| + \left(k-i-1\right),
\]
the exponent above is congruent to $|b_{ik}|+1$ modulo $2$, and therefore
\[
    db_{ij}     = \sum_{k=i+1}^{j-1} (-1)^{|b_{ik}|+1} b_{ik}b_{kj},
\]
which is the classical inductive relation of Definition \ref{Def : associative higher Massey products}.
Thus, the correspondence above is a bijection between the two kinds of defining systems.
Finally, the same congruence with $\left(i,j\right) = \left(0,n\right)$ compares the associated cycles:
the cycle of Equation \eqref{ecu : Massey cycle} for $\mu^c_n$ equals 
$(-1)^{(n-1)+\sum_{t=1}^{n}\left(n-t\right)\left|x_t\right|}$ times the classical cycle 
$\sum_{0<k<n} (-1)^{|b_{0k}|+1} b_{0k}b_{kn}$,
so that
\[
\left\langle x_1, \dots, x_n\right\rangle_{\mu^c_n} 
= (-1)^{(n-1)+\sum_{t=1}^{n}\left(n-t\right)\left|x_t\right|}\left\langle x_1, \dots, x_n\right\rangle.
\]
\hfill$\square$
\end{example}

 The observant reader will likely have spotted that the above is
 just one of the several linearly independent Massey products 
 that $\mathsf{Ass}$ possesses.
 In fact, there are different, linearly independent Massey products for each permutation $\sigma\in \mathbb{S}_n$, 
 since $\mu^c_{n}\cdot\sigma$ is linearly independent of $\mu^c_{n}$ when $\sigma \neq \id$. 
 Up to a sign, these are related by
 $\langle x_1, \dots x_n\rangle_{\mu^c_{n}\cdot\sigma} 
 = \left\langle x_{\sigma^{-1}(1)}, \dots x_{\sigma^{-1}(n)}\right\rangle_{\mu^c_{n}}$,
 see Prop. \ref{prop : Elementary properties of the operadic Massey products}.
Similarly, different presentations of an operad 
(in the associative case, one could take for example the Livernet--Loday presentation \cite[Prop. 9.1.1]{loday12})
give rise to seemingly distinct Massey products,
which are just the same expressed with respect to a different basis.

\begin{example}
\label{example: Lie--Massey brackets}
When $\mathcal P = \mathsf{Lie}$ is the Lie operad, 
our framework recovers the classical definition of higher Lie--Massey brackets as in \cite{Allday77,Ret77} 
(see also \cite{Tan83,jose17}).
To see this,
recall that the weight $n$ part of $\mathsf{Lie}^{\antishriek}$ is one-dimensional 
and generated by $\tau^c_{n+1}\in \mathsf{Lie}^{\antishriek}(n+1)$.
Recall also from Example \ref{Ex:LieIndexingSet} that in this case, 
the $\tau^c_{n}$-indexing set is
\[
I\left({\tau^c_n}\right):= \left\{(\tau^c_{l}, 
(i_1, \dots i_l))\mid \ l< n \mbox{ and }1\leq i_1< \cdots < i_l\leq n \right\}.
\]
We show next that given a particular differential graded Lie algebra $(L,d)$ 
and homogeneous elements $x_1,\ldots,x_n\in H_*(L)$,
there is a bijective correspondence between the classical defining systems 
$\left\{x_{j_1,\ldots,j_l}\right\}$ of  \cite{Allday77},
and the defining systems $\left\{a_\alpha\right\}$
for the $\tau^c_{n}$-Massey product as defined in this paper.
Indeed,
the correspondence is given by 
\[
x_{j_1,\dots,j_l} = a_{\tau^c_{l}, \left(j_1,\dots,j_l\right)} \quad \textrm{ for all } \ 1\leq j_1<\cdots< j_l\leq n.
\]
Recall from Example \ref{Ex:LieIndexingSet} that 
\[
D\left(\tau_n^c\right) 
= \sum_{\substack{r_1+r_2 = n\\ \sigma\in \overline{Sh}^{-1}\left(r_1, r_2\right)}} 
(-1)^{r_1+1}\operatorname{sgn}(\sigma)\left(\tau_2; \tau_{r_1}^c, \tau_{r_2}^c; \sigma\right).
\]
On each term, the only block of nonzero degree is the second one, of degree $\left|\tau^c_{r_2}\right| = r_2 - 1$.
Therefore, by directly applying Definition \ref{def: Defining systems for operadic Massey products}, we see that
\[
    d a_{\tau^c_{l}, \left(j_1,\dots,j_l\right)}
    = \sum_{\substack{r_1+r_2 =l\\ 
    \sigma\in \overline{Sh}^{-1}(r_1, r_2)}}(-1)^{r_1+1}\operatorname{sgn}(\sigma)\, (-1)^{\gamma_\sigma}\\
    \tau_2\left( a_{\tau_{r_1}^c, 
    \left(j_{\sigma^{-1}(1)}, j_{\sigma^{-1}(2)},\dots, j_{\sigma^{-1}(r_1)}\right)}, 
    a_{\tau_{r_2}^c, \left(j_{\sigma^{-1}\left(r_1+1\right)},\dots, j_{\sigma^{-1}\left(l\right)}\right)} \right),
\]
where 
\[
\gamma_\sigma = \sum_{\substack{p<q \\ \sigma(p)>\sigma(q)}}|x_{j_p}||x_{j_q}| 
\ + \ \left(r_2-1\right)\left(|x_{j_{\sigma^{-1}(1)}}|+\cdots+|x_{j_{\sigma^{-1}(r_1)}}|\right).
\]
Under the correspondence above, 
and re-indexing the sum by the reduced shuffles $\sigma \in \overline{Sh}\left(p, l-p\right)$ 
inverse to the unshuffles, with $p = r_1$,
this relation reads
\[
    dx_{j_1,\dots,j_l} = \sum_{p=1}^{l-1}  \sum_{\sigma\in \overline{Sh}(p, l-p)}
    \epsilon(\sigma) \left[x_{j_{\sigma(1)},\dots, j_{\sigma(p)}}, x_{j_{\sigma(p+1)},\dots, j_{\sigma(l)}}\right],
\]
with $\epsilon(\sigma) = (-1)^{p+1}\operatorname{sgn}(\sigma)\, (-1)^{\gamma_{\sigma^{-1}}}$, 
where
\[
\gamma_{\sigma^{-1}} = \sum_{\substack{s\leq p< t\\ 
\sigma(s)>\sigma(t)}}\left|x_{j_{\sigma(s)}}\right|\left|x_{j_{\sigma(t)}}\right|
\ +\ \left(l-p-1\right)\left(\left|x_{j_{\sigma(1)}}\right|+\cdots+\left|x_{j_{\sigma(p)}}\right|\right).
\]
This is the defining-system relation for the classical higher Lie--Massey brackets;
the conventions of \cite{Allday77,Ret77} differ from the one displayed here at most by 
a rescaling $x_{j_1,\dots,j_l}\mapsto \pm\, x_{j_1,\dots,j_l}$ of the elements of a defining system,
which induces a bijection of defining systems as in the associative case.
We leave this verification to the interested reader.
\hfill$\square$
\end{example}

As likely expected, the higher-order Massey products for commutative differential graded associative algebras
coincide with those formed by forgetting that the structure is commutative. 
This can be seen as a consequence of the theory developed in the next section, 
see Example \ref{Example: Commutative Massey Products}.

\medskip

In \cite{Mur21},
Muro contributed a new kind of triple Massey-product operation for Gerstenhaber and/or Poisson algebras.
Our framework recovers this triple operation as a consequence of  Proposition \ref{prop : Our products coincide with Muro's},
since the inner indices of the Poisson relation are pairwise distinct.
It follows from our results that all the higher-order analogs of this new
operation also exist. 
We furthermore give a new example: a four-fold higher-order Massey product for Poisson algebras.
We hope these higher products will be successfully applied in the future in cases where the triple-product operations 
defined by Muro do not suffice.

\begin{example}
Recall that the Poisson operad $\mathsf{Pois}$ is Koszul self-dual up to operadic suspension,
generated by a commutative associative product $\wedge$ and a Lie bracket $[-,-]$,
both of degree zero, which are compatible via the Poisson relation,
\[
[x\wedge y,z] = x\wedge [y,z] + (-1)^{|y||z|} [x,z]\wedge y.
\]
The inner combinatorics of this operad are complex,
and multiplying base elements frequently involves complicated
rewriting procedures.
Therefore, we cannot hope to write down formulas that are quite
as clean as
in examples \ref{example : Operadic def systems for Ass} and
\ref{example: Lie--Massey brackets}.
Nonetheless, it is possible to compute Poisson Massey products
inductively in low weight.

The relation sub-module $R(3)$ contains the two quadratic
relations,
\begin{align*}
\rho_1 &= [x\wedge y,z] - x\wedge [y,z]
- (-1)^{|y||z|} [x,z]\wedge y,\\[0.25cm]
\rho_2 &= [x, y\wedge z] - [x,y]\wedge z
- (-1)^{|x||y|}\, y\wedge [x,z].
\end{align*}
Each one determines a weight $2$ cooperation
$\Gamma^c_i := s^2\rho_i \in
\left(\mathsf{Pois}^{\antishriek}\right)^{(2)}(3)$
with its fixed expansion, as in Proposition
\ref{prop : Our products coincide with Muro's}.
For $\Gamma^c_1$, a defining system for the Massey product
$\langle x_1,x_2,x_3\rangle_{\Gamma^c_1}$ in a Poisson algebra
is a set of elements
$\left\{z_1,z_2,z_3,y_1,y_2, y_3\right\}$
where each $z_i$ is a cycle representative of $x_i$, and
\[
dy_1 = z_1\wedge z_2, \qquad dy_2 = [z_2, z_3], \qquad
dy_3 = [z_1, z_3 ].
\]
The cycle representative associated to this defining system is
$$
[y_1, z_3] -(-1)^{|z_1|} z_1\wedge y_2
- (-1)^{|z_2|+|z_1||z_2|}z_2\wedge y_3.
$$

Similarly, for $\Gamma^c_2$, a defining system for
$\langle x_1,x_2,x_3\rangle_{\Gamma^c_2}$
is a set of elements
$\left\{z_1,z_2,z_3,y_1,y_2, y_3\right\}$
where each $z_i$ is a cycle representative of $x_i$, and
\[
dy_1 = z_2 \wedge z_3, \qquad dy_2 = [z_1, z_2], \qquad
dy_3 = [z_1, z_3].
\]
The cycle representative associated to this defining system is
$$
(-1)^{|z_1|}[z_1, y_1] - y_2\wedge z_3
- (-1)^{|z_2|+|z_1||z_2|}\, z_2\wedge y_3.
$$
\hfill$\square$
\end{example}

\begin{example}
We continue the previous example by computing a higher-order Massey product 
for Poisson algebras. To our knowledge, this product has not appeared in the literature before.
Let 
\[
\Gamma^c = \wedge\circ([-,-],[-,-])\in\left(\mathsf{Pois}^{\antishriek}\right)^{(3)}(4).
\]
Let $m=s\wedge$ and $\ell=s[-,-]$.
Then 
\[
D(\Gamma^c)= ([-,-];m,m)
 -(\wedge;\id,\Gamma^c_2)
 -(\wedge;\Gamma^c_2,\id)\cdot(2,4,3)
 +(\wedge;\Gamma^c_1,\id)
 +(\wedge;\Gamma^c_1,\id)\cdot(3,4).
\]
Here, the permutation $(2,4,3)$ gives the ordered input blocks
$(1,3,4),(2)$, and $(3,4)$ gives $(1,2,4),(3)$.
The cycle representative associated to a defining system is
\begin{align}
\label{4thorderMasseyproduct}
\begin{split}
 &(-1)^{|z_1|+|z_2|}
 [y_{\wedge, (1,2)},y_{\wedge, (3,4)}]
 -z_1\wedge x_{b, (2,3,4)}\\
 &-(-1)^{|z_2|(|z_3|+|z_4|)}x_{b, (1,3,4)}\wedge z_2
 +x_{a, (1,2,3)}\wedge z_4\\
 &+(-1)^{|z_3||z_4|}x_{a, (1,2,4)}\wedge z_3,
\end{split}
\end{align}
where each $z_i$ is a cycle representative of $x_i$, and
\[
 d y_{\wedge, (i,j)}=z_i\wedge z_j,\qquad
 d y_{[-], (i,j)}=[z_i,z_j],
\]
\begin{align*}
 d x_{a, (i,j,k)}
 ={}&[y_{\wedge, (i,j)},z_k]
 -(-1)^{|z_i|}z_i\wedge y_{[-], (j,k)}\\
 &-(-1)^{|z_j|+|z_i||z_j|}z_j\wedge y_{[-], (i,k)},\\[0.2cm]
 d x_{b, (i,j,k)}
 ={}&(-1)^{|z_i|}[z_i,y_{\wedge, (j,k)}]
 -y_{[-], (i,j)}\wedge z_k\\
 &-(-1)^{|z_j|+|z_i||z_j|}z_j\wedge y_{[-], (i,k)}.
\end{align*}
The required indices are $(1,2),(3,4)$ for $y_{\wedge}$,
$(1,3),(1,4),(2,3),(2,4)$ for $y_{[-]}$,
$(1,2,3),(1,2,4)$ for $x_a$, and
$(1,3,4),(2,3,4)$ for $x_b$.
Each element is shared by all occurrences of its index.
The degrees are
\begin{align*}
 |y_{\wedge, (i,j)}|=|y_{[-], (i,j)}|&=|z_i|+|z_j|+1,\\
 |x_{a, (i,j,k)}|=|x_{b, (i,j,k)}|&=|z_i|+|z_j|+|z_k|+2.
\end{align*}
The right-hand sides defining $d x_a$ and $d x_b$ are exactly the
triple cycles in the previous example. 
\hfill$\square$
\end{example}

Our final example will illustrate the close connection of Massey products with spectral sequences.
We point the reader to \cite[10.3.7]{loday12} for some of the basic background on this example. 
We remark that technically, the operad in this example is not generated by a reduced symmetric sequence, 
contrary to our convention in this paper. 
Nonetheless, the techniques apply verbatim in this special case, and we outline our theory therein.

\begin{definition}
The \emph{dual numbers operad} is the quadratic operad $\mathcal D$ presented as
$$
\mathcal D := \mathcal F\left(\Bbbk \triangle, \triangle\circ  \triangle \right),
$$
where $ \triangle$ is an arity 1 element of homological degree 1. 
\end{definition}

Algebras over this operad are precisely the mixed complexes, i.e.,
chain complexes $(A,d)$ equipped with an operation 
$\triangle:A\to A$ such that $\triangle^2=0$ and $d\triangle +\triangle d = 0$.
A particular case of this is the bicomplexes $\left(A_{*,*},d',d''\right)$,
seen as mixed complexes with $A_n=\bigoplus_{p-q=n}A_{p,q}$, $d=d'$ and $\triangle=d''$.
The dual numbers operad is Koszul,
and its Koszul dual cooperad is cofree conilpotent on a single generator,
$$
\mathcal D^{\antishriek} = \mathcal F^c \left(s\triangle\right).
$$
In particular, this cooperad has no corelations and is concentrated in arity 1.

\begin{example}
We shall compute the Massey products of the dual numbers operad. 
The arity $1$ component of $\mathcal D^\antishriek$ is
$$
\mathcal D^\antishriek(1) = \bigoplus \Bbbk \delta_n,
$$
where $\delta_n$ has weight $n$ and degree $2n.$ 
Since
\[
\Delta\left(\delta_n\right) = \sum_{\substack{k+l =n\\ k,l \geq 0}} \left(\delta_k; \delta_l\right),
\]
where $l\geq 0$ and $k\geq 0$, it follows that
$$
D(\delta_n)= (\triangle; \delta_{n-1}) \quad \mbox{ for all } n\geq 2.
$$
Therefore,
the $\delta_n$-defining system is given by $\left\{a_{\delta_i}: 0 \leq i< n\right\},$ 
with the relation $da_{\delta_i} = \triangle\left(a_{\delta_{i-1}}\right)$. 
This is almost the definition of the $d^{n}$-differential in the spectral 
sequence associated to the bicomplex $(A,d,\triangle)$. 
More precisely, the proof of Theorem \ref{Thm : Differentials EMSS and Massey products}
applies verbatim to this non-reduced operad.
Hence, if $x\in\langle y\rangle_{\delta_n}$, then $\delta_n\otimes y$ survives to the $E^{n}$-page
of the $\mathcal D$-Eilenberg--Moore spectral sequence,
that is, the spectral sequence of the weight filtration of $\bigoplus_{k\geq 0}\delta_k\otimes A$,
and $d^{n}\left(\delta_n\otimes y\right) = (-1)^{n-1}\left[\idd\otimes x\right]$.
\hfill$\square$
\end{example}

\begin{remark}
\label{remark: Massey products for non-Koszul operads}
In our higher-order Massey products framework for Koszul operads, 
there is nothing special about the Koszul dual cooperad $\mathcal P^{\antishriek}$ 
aside from it being a very useful resolution.
In principle, 
starting
with any conilpotent cooperad $\mathcal C$ together with a choice of twisting morphism 
$\tau:\mathcal C \to \mathcal P$,
it is possible to define a \textit{relative Massey inductive map} 
\[
D:\mathcal C\xrightarrow{\Delta}\mathcal C\circ \mathcal C \xrightarrow{\tau\circ \operatorname{id}}
\mathcal P \circ \mathcal C.
\]
From this, one defines \textit{relative Massey products} following, 
\emph{mutatis mutandis}, twisting the same recipe we gave in the Koszul case.
Taking $\mathcal C = B\mathcal P$ to be the bar construction of $\mathcal P$
and $\tau:B\mathcal P\to \mathcal P$ the canonical twisting morphism,
this allows for defining Massey products for non-Koszul operads. 
We leave this as possible future work.
\end{remark}

\subsection{Elementary properties of the operadic Massey products}
\label{sec: Elementary properties of the operadic Massey products}

In this section, 
we collect some elementary properties of the operadic Massey products.
First, 
we show 
that Massey product sets do not depend on the initial choice of cycles in the defining system
(Prop. \ref{prop : Massey products do not depend on cycles}).
Then, 
that morphisms of $\mathcal P$-algebras preserve Massey products (Prop. \ref{prop : Morphisms preserve Massey products}).
In particular, quasi-isomorphisms induce bijections of the corresponding Massey product sets.
This provides an obstruction 
for two $\mathcal P$-algebras to be weakly equivalent. 
In particular, a non-trivial Massey product provides an obstruction to formality.
At the end of the section, we collect
a few elementary properties of the Massey products that might be 
useful elsewhere (Prop. \ref{prop : Elementary properties of the operadic Massey products}).

\medskip

For the next few results, 
we fix a cooperation $\Gamma^c\in \left(\mathcal{P}^\antishriek \right)^{(n)} (k)$.
We say that a Massey product set $\langle x_1,\ldots,x_k\rangle_{\Gamma^c}$ is
\emph{defined} if it is non-empty, that is, if there is some defining system for the Massey product; 
\emph{trivial} if it contains the zero homology class;
and \emph{non-trivial} if it is defined and does not contain the zero homology class. 

First, we shall show that Massey product sets do not depend on the initial choice of cycles in a defining system.

\begin{proposition}
\label{prop : Massey products do not depend on cycles}
Let $A$ be a $\mathcal P$-algebra, and $x_1,\ldots,x_k \in H_*(A)$ be homogeneous elements.
Then the set $\langle x_1,\ldots,x_k\rangle_{\Gamma^c}$ 
does not depend on the choice of representative cycles for the classes $x_i$.
Being more precise, 
for any choice of representative cycles $\overline{x_i}$ for $x_i$ 
and for each $x \in \langle x_1,\ldots,x_k\rangle_{\Gamma^c}$,
there is a defining system $\left\{a_{\beta}\right\}$ for $x$ such that $a_{\id, (i)} = \overline{x_i}$. 
\end{proposition}

\begin{proof}
Let $\left\{b_{\beta}\right\}$ be a defining system for a 
Massey product $x \in \left\langle x_1,\ldots,x_k\right\rangle_{\Gamma^c}$. 
We shall construct, by induction on the weight of the elements of the defining system, 
a defining system for $x$ such that $a_{\id, (i)} = \overline{x_i}$ 
and that $a_{\Gamma^c, (1,\dots, k)}$ is homologous to $b_{\Gamma^c, (1,\dots, k)}$.
It suffices to consider the case where only one representative cycle is modified;
assume  without loss of generality this is $x_1$. Let $a_{\id, (1)} = \overline{x_1}$ be any cycle representative.
We claim there is a defining system $\left\{a_\beta\right\}$ for $x$ with
\[
a_{\id, (1)} = \overline{x_1}
\qquad \textrm{ and } \qquad
a_{\id, (i)} = b_{\id, (i)} \quad \textrm{ for } i \geq 2.
\]
We say that an index $\left(\mu, (j_1,\dots, j_r)\right) \in I\left(\Gamma^c\right)$ is \emph{marked} 
if $1 \in \left\{j_1,\dots, j_r\right\}$.
If an index is marked, then for every term $\left(\zeta ; \zeta_1,\dots,\zeta_m; \sigma\right)$ of $D(\mu)$
exactly one of the cooperations $\zeta_1,\dots,\zeta_m$ carries the label $1$, and it is $\zeta_1$,
since in the canonical planar 
tree representation the entries of each vertex are ordered by the minimum of the leaves linked to them.
If an index is not marked, then none of the $\zeta_i$ is marked either.

For the first step, simply fix $a_{\id, (i)} = \overline{x_i}$. 
Since $a_{\id, (1)}$ and $b_{\id, (1)}$ are both representatives
for $x_1,$ their difference $a_{\id, (1)} -b_{\id, (1)}$ is nullhomologous. 
This means that there is a $c \in A$ such that
\[
dc = a_{\id, (1)} -b_{\id, (1)},
\]
while $a_{\id, (i)} = b_{\id, (i)}$ for every $i \geq 2$.
The family $\left\{a_{\id, (i)}\right\}$, together with the choice of the element $c$, gives the first inductive step.
Let $n$ be the weight of $\Gamma^c$, and fix $N$ such that $1 \leq N < n$. 
Now, suppose we have chosen elements $a_{\mu, (i_1,\dots, i_r)} \in A$
for all indices $\left(\mu, (i_1,\dots, i_r)\right) \in I\left(\Gamma^c\right)$ 
where the weight of $\mu$ is strictly less than $N$, 
such that
\begin{equation*}
d\left(a_{\mu,\left(i_1, \ldots, i_r\right)}\right) 
    = \sum \lambda (-1)^\gamma \zeta \left(a_{\zeta_1,\left(i_{\sigma^{-1}\left(1\right)},\dots 
    i_{\sigma^{-1}\left(v_1\right)}\right)},
    \dots,a_{\zeta_m,\left(i_{\sigma^{-1}\left(v_1+\cdots+ v_{m-1}+1\right)},\dots, i_{\sigma^{-1}(r)}\right)} \right),
\end{equation*}
where $D\left(\mu\right) 
= \sum \lambda \left(\zeta ; \zeta_1,\dots,\zeta_m; \sigma\right)$ and $v_i$ is the arity of $\zeta_i$.
From now on, we abbreviate the entries of such sums by $a_{\zeta_1},\dots, a_{\zeta_m}$, 
and likewise for the elements $b$ and $Q$. 
We will also omit the scalars $\lambda (-1)^\gamma$ of
Def. \ref{def: Defining systems for operadic Massey products}, and 
we further suppose that
\begin{equation*}
    a_{\mu, (i_1,\dots, i_r)} = b_{\mu, (i_1,\dots, i_r)} + Q_{\mu, (i_1,\dots, i_r)},
\end{equation*}
where $Q_{\mu,(i_1,\dots,i_r)}$ vanishes if the index is not marked, and is otherwise the sum
\[
Q_{\mu,(i_1,\dots,i_r)} = \sum \zeta\left(c, b_{\zeta_2},\dots,b_{\zeta_m}\right),
\]
indexed by those terms $(\zeta;\zeta_1,\dots,\zeta_m;\sigma)$ of $D(\mu)$ in which the
first block $\zeta_1 = \id$ is the identity cooperation carrying the label $1$.
Now, let $\left(\mu, (j_1,\dots, j_r)\right)\in I\left(\Gamma^c\right)$ have $\mu$ of weight $N$.
Then, $Q_{\mu, (j_1,\dots, j_r)}$ is well defined,
because the cooperations appearing in its defining tuple have weight strictly less than $\mu$. 
Its boundary is as follows:
\begin{equation}
\label{ecu : differential of Q}
dQ_{\mu, (j_1,\dots, j_r)} = 
\sum \zeta \left( a_{\zeta_1}, \dots, a_{\zeta_m} \right)
- \sum \zeta \left( b_{\zeta_1}, \dots, b_{\zeta_m} \right).
\end{equation}
Indeed, by Leibniz's rule, $dQ_{\mu, (j_1,\dots, j_r)}$ is the sum of two types of terms: 
one containing $dc$, and those containing a single $db_{\zeta_i}$.
The term containing $dc$ gives the right-hand side of the equation above, 
up to the terms containing some $Q_{\zeta_1}$, 
since $a_l = b_l$ for every unmarked index $l$, 
and $a_{\zeta_1} - b_{\zeta_1}$ equals $dc$ when $\zeta_1 = \id$ and $Q_{\zeta_1}$ otherwise.
Now, observe that the terms containing some $Q_{\zeta_1}$ are two-vertex
tree monomials with $c$ at the leaf labeled $1$, attached to the second vertex. 
Rewriting each $db_{\zeta_i}$ by means of the defining system relations
for $\left\{b_\beta\right\}$, the terms containing $db_{\zeta_i}$ become two-vertex 
tree monomials with $c$ at the leaf labeled $1$, attached to the root vertex. 
Together, these are precisely the terms of $G(\mu)$
in which the leaf labeled $1$ is a singleton block, evaluated with $c$ in
that slot and elements of $\left\{b_\beta\right\}$ elsewhere.
Since $\operatorname{Im} G \subseteq R\circ\mathcal P^{\antishriek}$ and, moreover,
this inclusion is compatible with the decomposition by lower factors, their sum vanishes,
as in the proof of Proposition~\ref{Prop: cycle}.

Since $\left\{b_\beta\right\}$ is a defining system, the second sum on the right-hand side of
Equation (\ref{ecu : differential of Q}) is $db_{\mu, (j_1,\dots, j_r)}$.
Therefore, setting
\[
a_{\mu, (j_1,\dots, j_r)} := b_{\mu, (j_1,\dots, j_r)} + Q_{\mu, (j_1,\dots, j_r)},
\]
so that in particular $a_{\mu, (j_1,\dots, j_r)} = b_{\mu, (j_1,\dots, j_r)}$ whenever the index is not marked,
we obtain
\[
d\left(a_{\mu,\left(j_1, \cdots, j_r\right)}\right) 
= db_{\mu, (j_1,\dots, j_r)} + dQ_{\mu, (j_1,\dots, j_r)}
= \sum \zeta \left(a_{\zeta_1},\dots,a_{\zeta_m} \right),
\]
where the sum ranges over
$D\left(\mu\right) = \sum \left(\zeta ; \zeta_1,\dots,\zeta_m; \sigma\right)$.
This concludes the induction step. 
To finish, consider the element $Q_{\Gamma^c, (1,\dots, k)}$.
This is defined by the same logic as above, and Equation (\ref{ecu : differential of Q}) applied to it gives
\[
dQ_{\Gamma^c, (1,\dots, k)} = a_{\Gamma^c, (1,\dots, k)} - b_{\Gamma^c, (1,\dots, k)}.
\]
Therefore, the elements $a_{\Gamma^c, (1,\dots, k)}$ and $b_{\Gamma^c, (1,\dots, k)}$
are homologous, as we wanted to prove.
\end{proof}

\begin{proposition}
\label{prop : Morphisms preserve Massey products}
	A morphism  of $\mathcal P$-algebras $f:A\to B$ preserves Massey products. 
	That is, 
	if $x_1,\ldots,x_k \in H_*(A)$ are homogeneous elements such that
 the Massey product set $\langle x_1,\ldots,x_k\rangle_{\Gamma^c}$ is defined, 
	then $\left\langle f_*\left(x_1\right),\ldots,f_*\left(x_k\right)\right\rangle_{\Gamma^c}$ is also defined, 
	and moreover
	$$f_*\langle x_1,\ldots,x_k\rangle_{\Gamma^c} \subseteq \left\langle 
    f_*\left(x_1\right),\ldots,f_*\left(x_k\right)\right\rangle_{\Gamma^c}.$$
	If furthermore $f$ is a quasi-isomorphism, 
	then $f_*$ induces a bijection between the corresponding Massey product sets.
\end{proposition}

\begin{proof}
Assume that the Massey product set $\langle x_1,\ldots,x_k\rangle_{\Gamma^c}$ is defined.
Then, any defining system $\left\{a_{\alpha}\right\}$ for $\langle x_1,\ldots,x_k\rangle_{\Gamma^c}$
produces a defining system $\left\{f(a_{\alpha})\right\}$ for 
$\left\langle f_*\left(x_1\right),\ldots,f_*\left(x_k\right)\right\rangle_{\Gamma^c}$,
because $f$ commutes with the operadic structure maps and the differentials.
Therefore, if $\langle x_1,\ldots,x_k\rangle_{\Gamma^c}$ is defined, then 
$\left\langle f_*\left(x_1\right),\ldots,f_*\left(x_k\right)\right\rangle_{\Gamma^c}$ is also defined,
and the containment $f_*\langle x_1,\ldots,x_k\rangle_{\Gamma^c} \subseteq \left\langle 
f_*\left(x_1\right),\ldots,f_*\left(x_k\right)\right\rangle_{\Gamma^c}$
follows.

Next, assume that $f$ is a quasi-isomorphism and let us prove that the 
corresponding Massey product sets are in bijective correspondence.
Since $f_*$ is injective, 
it suffices to show that $f_*$ maps $\langle x_1,\ldots,x_k\rangle_{\Gamma^c}$ onto 
$\left\langle f_*\left(x_1\right),\ldots,f_*\left(x_k\right)\right\rangle_{\Gamma^c}$.
We first treat the case in which $f$ is furthermore surjective, 
and then reduce the general case to this one.

Assume that $f$ is a surjective quasi-isomorphism. 
Then, $\Ker f$ is an acyclic chain complex.
Let $\left\{b_{\beta}\right\}$ be a defining system for a 
Massey product $y \in \left\langle f_*\left(x_1\right),\ldots,f_*\left(x_k\right)\right\rangle_{\Gamma^c}$. 
We shall construct, by induction on the weight of the elements of the defining system, 
a defining system $\left\{a_{\beta}\right\}$ for 
$\langle x_1,\ldots,x_k\rangle_{\Gamma^c}$
such that $f\left(a_{\beta}\right) = b_{\beta}$ for every index $\beta\in I\left(\Gamma^c\right)$.

For the first step, fix $1\leq i\leq k$ and choose $\widetilde{a}_{\idd,(i)}\in A$ 
such that $f\left(\widetilde{a}_{\idd,(i)}\right) = b_{\idd,(i)}$.
Since $b_{\idd,(i)}$ is a cycle, 
$d\widetilde{a}_{\idd,(i)}$ is a cycle in $\Ker f$.
Therefore, there is some $h_{\idd,(i)}\in \Ker f$ such that $dh_{\idd,(i)} = d\widetilde{a}_{\idd,(i)}$, 
and the element
\[
a_{\idd,(i)} := \widetilde{a}_{\idd,(i)} - h_{\idd,(i)}
\]
is a cycle satisfying $f\left(a_{\idd,(i)}\right) = b_{\idd,(i)}$.
Let $n$ be the weight of $\Gamma^c$, and fix $N$ such that $1 \leq N < n$. 
Now, suppose we have chosen elements $a_{\mu, (i_1,\dots, i_r)} \in A$
for all indices $\left(\mu, (i_1,\dots, i_r)\right) \in I\left(\Gamma^c\right)$ 
where the weight of $\mu$ is strictly less than $N$, 
satisfying the relations of 
Definition \ref{def: Defining systems for operadic Massey products}
and such that $f\left(a_{\mu, (i_1,\dots, i_r)}\right) = b_{\mu, (i_1,\dots, i_r)}$.
Let $\left(\mu, (j_1,\dots, j_r)\right)\in I\left(\Gamma^c\right)$ have $\mu$ of weight $N$,
and suppose $D\left(\mu\right) = \sum \lambda\left(\zeta ; \zeta_1,\dots,\zeta_m; \sigma\right)$. 
As in the proof of Proposition \ref{prop : Massey products do not depend on cycles}, 
we abbreviate the entries of the corresponding sums by $a_{\zeta_1},\dots, a_{\zeta_m}$, 
and likewise for the elements $b$.
Choose $\widetilde{a}_{\mu, (j_1,\dots, j_r)}\in A$ 
such that $f\left(\widetilde{a}_{\mu, (j_1,\dots, j_r)}\right) = b_{\mu, (j_1,\dots, j_r)}$, 
and consider the element
\[
e_{\mu, (j_1,\dots, j_r)} := d\widetilde{a}_{\mu, (j_1,\dots, j_r)} 
- \sum \lambda (-1)^\gamma \zeta \left(a_{\zeta_1}, \dots, a_{\zeta_m}\right).
\]
The element $e_{\mu, (j_1,\dots, j_r)}$ is a cycle,
because the sum on the right-hand side is a cycle by Proposition \ref{Prop: cycle} 
applied to the $\mu$-defining system formed by the elements already chosen.
Moreover, since $f$ commutes with the differentials and the operadic structure maps,
\[
f\left(e_{\mu, (j_1,\dots, j_r)}\right) = db_{\mu, (j_1,\dots, j_r)} 
- \sum \lambda (-1)^\gamma \zeta \left(b_{\zeta_1}, \dots, b_{\zeta_m}\right) = 0,
\]
because $\left\{b_{\beta}\right\}$ is a defining system.
Therefore, there is some $h_{\mu, (j_1,\dots, j_r)}\in \Ker f$ 
such that 
\[
dh_{\mu, (j_1,\dots, j_r)} = e_{\mu, (j_1,\dots, j_r)}.
\]
Let
\[
a_{\mu, (j_1,\dots, j_r)} := \widetilde{a}_{\mu, (j_1,\dots, j_r)} - h_{\mu, (j_1,\dots, j_r)}.
\]
Then, 
\[
d\left(a_{\mu,\left(j_1, \dots, j_r\right)}\right) 
= \sum \lambda (-1)^\gamma \zeta \left(a_{\zeta_1},\dots,a_{\zeta_m} \right),
\qquad \textrm{and} \qquad
f\left(a_{\mu, (j_1,\dots, j_r)}\right) = b_{\mu, (j_1,\dots, j_r)}.
\]
This concludes the induction step. 
To finish, since $f$ commutes with the operadic structure maps, 
it sends the cycle $a_{\Gamma^c,(1,...,k)}$ associated to $\left\{a_{\beta}\right\}$ 
to the cycle $y$ associated to $\left\{b_{\beta}\right\}$, finishing the proof.

\medskip

For the general case, recall that $f$ admits an explicit factorization 
as a quasi-isomorphism followed by a degree-wise surjection,
\[
\begin{tikzcd}
A \arrow[r, hook, "i"] & A\sqcup E(B) \arrow[r, two heads, "p"] & B.
\end{tikzcd}
\]
Here, $i$ is the canonical map into the coproduct, 
and $p$ is induced by the universal property of the coproduct 
applied to $f$ and the degree-wise surjective quasi-isomorphism $E(B)\to B$,
with $E(B)$ is the quasi-free contractible $\P$-algebra on $B$.
Since $f$ and $i$ are quasi-isomorphisms, so is $p$.
Furthermore, $p$ is a degree-wise surjection since so is its restriction $E(B)\to B$.
Moreover, the morphism $r : A\sqcup E(B)\to A$ induced by the identity of $A$ 
and the zero morphism $E(B)\to A$ satisfies $r\circ i = \idd_A$.
Hence, $r$ is a quasi-isomorphism and $i_* = r_*^{-1}$.
Now, let $y \in \left\langle f_*\left(x_1\right),\ldots,f_*\left(x_k\right)\right\rangle_{\Gamma^c}$.
Since $f_*\left(x_i\right) = p_*\left(i_*\left(x_i\right)\right)$ for each $i$, 
the surjective case applied to $p$ provides a Massey product 
$z\in \left\langle i_*\left(x_1\right),\ldots,i_*\left(x_k\right)\right\rangle_{\Gamma^c}$ 
such that $p_*(z) = y$.
By the first part of this proof applied to $r$, 
\[
r_*(z)\in \left\langle r_*i_*\left(x_1\right),\ldots,r_*i_*\left(x_k\right)\right\rangle_{\Gamma^c} 
= \langle x_1,\ldots,x_k\rangle_{\Gamma^c}.
\]
Finally, $f_*\left(r_*(z)\right) = p_*i_*r_*(z) = p_*(z) = y$.
This shows that
\[
f_*\langle x_1,\ldots,x_k\rangle_{\Gamma^c}
= \left\langle f_*\left(x_1\right),\ldots,f_*\left(x_k\right)\right\rangle_{\Gamma^c}.
\]
\end{proof}

Recall that two $\mathcal P$-algebras are \emph{weakly-equivalent}, or \emph{quasi-isomorphic},
if there is a zig-zag of $\mathcal P$-algebra quasi-isomorphisms between them. 
From the two previous results, we can deduce the following.

\begin{corollary}
\label{cor:QuasiIsoMasseyProducts}
 There is a bijection between the Massey product sets of weakly-equivalent $\mathcal P$-algebras.
\end{corollary}

Recall that, for an operad $\mathcal P$ without differential, 
the homology of a $\mathcal P$-algebra is also a $\mathcal P$-algebra.  
We say that a $\mathcal P$-algebra is \emph{formal} if it is weakly equivalent to its homology 
endowed with the induced $\mathcal P$-algebra structure (and trivial differential). 
In general, the Massey products of $\mathcal P$-algebras with trivial differential 
are always trivial because they have no relations that exist at the chain level but not at the homological level. 
In particular, the Massey products of the homology of a $\mathcal P$-algebra are all trivial whenever they are defined.
From this, we immediately deduce the following result.

\begin{corollary}
If a $\mathcal P$-algebra has a non-trivial Massey product, then it is not formal. 
\end{corollary}

\begin{proof}
Assume that a $\mathcal P$-algebra $A$ has a non-trivial Massey product. 
Since the homology of $A$ has a zero differential, 
it must be that all of its Massey products are trivial. 
Therefore, by Corollary \ref{cor:QuasiIsoMasseyProducts} it cannot be quasi-isomorphic to $A.$
\end{proof}

Next,
we collect some elementary properties satisfied by the operadic Massey products.
These are similar to some of those explained in \cite{kraines72} and, more recently, 
in \cite[p. 325]{Ravenel}.

 \begin{proposition}
 \label{prop : Elementary properties of the operadic Massey products}
 Let $A$ be a $\mathcal P$-algebra, $\Gamma^c\in \left(\mathcal{P}^{\antishriek}\right)^{(n)} (r)$,
 and $x_1, \dots, x_r \in H_*(A)$ 
 be homogeneous elements such that $\left\langle x_1,\ldots,x_r\right\rangle_{\Gamma^c}$ is defined.
 Then the following assertions hold.
     \begin{enumerate}
         \item 
         (Scalar homogeneity) 
         If $\lambda \in \Bbbk$ is a scalar, then for all $1 \leq i \leq r$,
         \[
         \lambda\langle x_1,\dots,x_r\rangle_{\Gamma^c} \subseteq \langle x_1,\dots, \lambda x_i, \dots 
         ,x_r\rangle_{\Gamma^c}.
         \]
         This inclusion is an equality if $\lambda$ is non-zero.
         \item (Equivariance)
        If $\mathcal B$ is monomial, then for every permutation $\sigma\in \mathbb S_r$,
         \[
         \langle x_1, \dots, x_r\rangle_{\Gamma^c\cdot\sigma} 
         =(-1)^{\varepsilon_\sigma}\langle x_{\sigma^{-1}(1)}, \dots x_{\sigma^{-1}(r)}\rangle_{\Gamma^c},
         \] 
         where $\displaystyle \varepsilon_\sigma = \sum_{\substack{1 \leq i<j \leq r\\ \sigma(i) > \sigma(j)}}|x_i||x_j|.$
     \end{enumerate}
 \end{proposition}

 \begin{proof}
        \begin{enumerate}
         \item 
         Let $\{a_\alpha\}$ be a defining system for a Massey product  
         $x \in \langle x_1,\dots,x_r\rangle_{\Gamma^c}$. Consider a new defining system
         $$
         b_{\zeta, (j_1, \dots j_s)} = \begin{cases}
             \lambda a_{\zeta, (j_1, \dots j_s)} & \mbox{if } j_l = i \mbox{ for some } l.
             \\
             a_{\zeta, (j_1, \dots j_s)} & \mbox{otherwise.}
         \end{cases}
         $$
         In particular, one has 
         $$
         b_{\id, (j)} = \begin{cases}
             \lambda a_{\id, (i)} &\mbox{for } i = j.
             \\
             a_{\id, (j)} &\mbox{otherwise.}
         \end{cases}
         $$
         so $\{b_\alpha\}$ is a defining system for a Massey product in
         $\langle x_1,\dots, \lambda x_i, \dots ,x_r\rangle_{\Gamma^c}.$  
         Furthermore, $b_{\Gamma^c, (1, 2,\dots, r)} = \lambda a_{\Gamma^c, (1, 2,\dots, r)}$ 
         so the corresponding Massey product is $\lambda x$.

         \item         Let $\{ a_\alpha \}$ be a defining system for 
        $x \in \langle x_1, \dots, x_r \rangle_{\Gamma^c \cdot \sigma}$.
        By Proposition \ref{prop: equivariance of indexing sets},
        relabeling each index $\beta = \left(\zeta,\left(l_1,\dots,l_m\right)\right)$ of $I\left(\Gamma^c\right)$ as
        $\left(\zeta,\left(\sigma^{-1}(l_1),\dots,\sigma^{-1}(l_m)\right)\right)$
        gives a non-zero multiple $c_\beta\,\alpha$ of a unique index $\alpha$ of $I\left(\Gamma^c\cdot\sigma\right)$,
        and $\beta\mapsto\alpha$ is a bijection.
        Let $b_\beta := \pm c_\beta\, a_\alpha$, where $\pm$ is the Koszul sign of putting the inputs
        $x_{\sigma^{-1}(l_1)},\dots,x_{\sigma^{-1}(l_m)}$ in increasing order of their indices.
        Since the signs $\gamma_T$ of Definition \ref{def: Defining systems for operadic Massey products}
        are precisely the Koszul signs of permuting the inputs and moving the blocks past them,
        the inductive relations for $\{a_\alpha\}$ and $\{b_\beta\}$ coincide term by term, up to these scalars.
        Hence $\{ b_\beta \}$ is a defining system for
        $\langle x_{\sigma^{-1}(1)}, \dots x_{\sigma^{-1}(r)}\rangle_{\Gamma^c}$,
        and the associated cycles differ exactly by the Koszul sign $(-1)^{\varepsilon_\sigma}$.
        The construction can be reversed, and the result follows.
         \end{enumerate}
 \end{proof}

\begin{remark}
\label{rem: equivariance needs monomial bases}
Proposition \ref{prop : Elementary properties of the operadic Massey products} (2)
fails if $\mathcal B$ is not monomial.
Consider again $\mathcal P = \mathsf{Ass}$ with the basis $\left\{\mu,\mu+\mu\cdot(12)\right\}$ of $E(2)$,
which is not monomial,
let $\Gamma^c = s^2\left(\mu\circ_1\mu-\mu\circ_2\mu\right)$,
and let $A$ be the associative algebra of the discussion following
Proposition \ref{prop : Our products coincide with Muro's}.
Then $\langle x_2,x_1,x_3\rangle_{\Gamma^c} = \{0\}$,
since it only requires $x_2x_1$ and $x_1x_3$ to vanish,
but $\langle x_1,x_2,x_3\rangle_{\Gamma^c\cdot(12)}$ is empty.
\end{remark}

\subsection{Massey products along morphisms of operads and formality}
\label{sec : Massey products along morphisms of operads and formality}

In this section, we shall discuss pullbacks and pushforwards of Massey products along morphisms of operads
and give some examples.

\medskip

Before we begin, it will be helpful to explain some preliminary observations.   
Let $f:\mathcal P \to \mathcal Q$ be a weight preserving morphism of operads.
In this case,
taking Koszul dual cooperad is functorial, 
and therefore there is an induced map 
$f^{\antishriek}:\mathcal F^c(sE,s^2R) = \mathcal P^{\antishriek} \to \mathcal Q^{\antishriek} = \mathcal F^c(sF,s^2S).$ 
Moreover,  there is a commutative diagram 
$$
\begin{tikzcd}[row sep = large, column sep = huge]
\mathcal P^{\antishriek} \arrow[r, "\Delta"] \arrow[d, "f^{\antishriek}"] &
\mathcal P^{\antishriek}\circ \mathcal P^{\antishriek} \arrow[d, "f^{\antishriek}\circ f^{\antishriek}"] 
\arrow[r, "\ \kappa\circ\operatorname{id} \ "] &
\mathcal P\circ \mathcal P^{\antishriek} \arrow[d, "f\circ f^{\antishriek}"]
\\
\mathcal Q^{\antishriek}\arrow[r, "\Delta"] &
\mathcal Q^{\antishriek}\circ \mathcal Q^{\antishriek} \arrow[r, "\ \kappa'\circ\operatorname{id}\ "] &
\mathcal Q\circ \mathcal Q^{\antishriek}
\end{tikzcd}
$$
From this, we conclude that the Massey inductive map $D$ commutes with $f^{\antishriek}.$ 
Secondly, on the level of algebras, $f$ descends to an adjoint pair
$$
f_! : \mathcal P-\mathsf{Alg} \leftrightarrows \mathcal Q-\mathsf{Alg} : f^*.
$$
See for instance \cite[Section 5.2.12]{loday12}.
The functor $f^\ast$ preserves the underlying chain complex of the $\mathcal Q$-algebras,
and therefore there is a chain map $f^\ast (A) \to A$ which is just the identity morphism.
We define next another chain map $h : A \to f_!(A)$. 
Given a $\mathcal P$-algebra $A$, 
the unit of the adjunction above is a morphism of $\mathcal P$-algebras
$$
A\to f^\ast f_!(A).
$$
Forgetting the $\mathcal P$-algebra structure and recalling that $f^\ast$ preserves the underlying chain complex, 
there is a chain map
\begin{equation}
    \label{ecu: chain map h}
    h: A\to f_!(A).
\end{equation}

Our second observation concerns bases.
Since $\mathcal P$-Massey products in general depend on a choice of basis on $\P$, 
it is natural that there will also be basis requirements on the morphisms $f:\mathcal P \to \mathcal Q$. 
Let $\mathcal B_{\mathcal P}$ and $\mathcal B_{\mathcal Q}$ be the fixed bases of
$\mathcal P^{\antishriek}$ and $\mathcal Q^{\antishriek}$,
with $\mathcal B_{\mathcal P}^{(1)} = s\mathcal B_E$ and $\mathcal B_{\mathcal Q}^{(1)} = s\mathcal B_F$.
We shall say that $f$ is \emph{monomial in weight one} if $f\left(\mathcal B_E\right)\subseteq\mathcal B_F\cup\{0\}$,
and that $f$ is \emph{monomial} if moreover
$f^{\antishriek}\left(\mathcal B_{\mathcal P}\right)\subseteq\mathcal B_{\mathcal Q}\cup\{0\}$.
The standard morphism $\mathsf{Ass}\to \mathsf{Com}$ is monomial in weight one
for the usual choices of bases, but it is not monomial.
If $f$ is injective and monomial in weight one, then $f^{\antishriek}$ is injective,
and $\mathcal B_{\mathcal Q}$ can be chosen so that $f$ is monomial.
Given a monomial map $f:\mathcal P \to \mathcal Q$,
then the fact that the Massey inductive map $D$ commutes with $f^{\antishriek}$ implies there is an identity
\begin{equation}
\label{eq:commutes}
\left(f \circ f^\antishriek\right)
D\left(\Gamma^c\right)
= \left(f \circ f^\antishriek\right) \left(\sum \lambda\left(\zeta; \zeta_1,\dots,\zeta_m; \sigma \right)\right)
=  \sum \lambda\left(f(\zeta); f^\antishriek(\zeta_1),\dots,f^\antishriek(\zeta_m); \sigma \right)
= D\left(  f^\antishriek\left(\Gamma^c\right)\right).
\end{equation}
The last equality holds because $f(\zeta)\in\mathcal B_F\cup\{0\}$
and $f^{\antishriek}(\zeta_i)\in\mathcal B_{\mathcal Q}\cup\{0\}$,
so that the middle sum is a basis expansion for $D\left(  f^\antishriek\left(\Gamma^c\right)\right)$,
once the terms with the same image are collected.
If $f$ is only monomial in weight one, the same holds for every $\Gamma^c$ of weight $2$,
since the blocks $\zeta_i$ then have weight at most $1$.

We think that an extension of the theory of this section to all maps $f \colon \P \to \mathcal Q$ is possible, 
but is likely to rely on the development of matric $\P$-Massey products.

\qquad

\noindent{\bf{Pullbacks of Massey products. }} 
For any $\mathcal Q$-algebra $A$, 
the $\mathcal Q$-Massey products on $A$ of the form $\langle x_1,\ldots,x_r\rangle_{f^{\antishriek}(\Gamma^c)}$
are, under the hypotheses below, $\mathcal P$-Massey products on $f^{\ast}(A)$.
Since the underlying chain complex of both algebras is the same,
we can prove the following result.

\begin{proposition}
\label{Prop: Pullback of Massey products}
    Let $f:\mathcal P \to \mathcal Q$ be a weight preserving
    morphism of operads, $A$ a $\mathcal Q$-algebra, and
    $\Gamma^c\in \left(\mathcal P^{\antishriek} \right)^{(n)} (r)$.
    Suppose that $x_1,\ldots,x_r \in H_*\left(f^*\left(A\right)\right)$
    are homogeneous elements such that the Massey product set
    $\langle x_1,\ldots,x_r\rangle_{\Gamma^c}$ is defined. Under the
    identification of $f^\ast(A)$ and $A$ as chain complexes:
    \begin{enumerate}
    \item if $f$ is injective and monomial, then
    $\langle x_1,\ldots,x_r\rangle_{f^{\antishriek}(\Gamma^c)}
    = \langle x_1,\ldots,x_r\rangle_{\Gamma^c}$;
    \item if $n = 2$ and $f$ is monomial in weight one, then
    $\langle x_1,\ldots,x_r\rangle_{f^{\antishriek}(\Gamma^c)}
    \subseteq \langle x_1,\ldots,x_r\rangle_{\Gamma^c}$.
    \end{enumerate}
\end{proposition}

\begin{proof}
Both assertions are trivial when $f^{\antishriek}\left(\Gamma^c\right)
= 0$, since the left-hand set is then empty by convention, so assume
$f^{\antishriek}\left(\Gamma^c\right) \neq 0$. We suppress the signs and scalars
$\lambda (-1)^\gamma$ of Definition
\ref{def: Defining systems for operadic Massey products} throughout,
as in the proof of Proposition
\ref{prop : Massey products do not depend on cycles}: every step below
regroups whole terms, so these are carried along unchanged.
In both cases, for every index $\left(\mu,\left(i_1,\dots,i_k\right)\right)$ of $I\left(\Gamma^c\right)$,
$f^{\antishriek}(\mu)$ is either zero or an element of $\mathcal B_{\mathcal Q}$ of the same weight as $\mu$:
in case $(1)$ because $f$ is monomial,
and in case $(2)$ because $\mu$ then has weight at most $1$ and $f$ is monomial in weight one.

$(1)$ If $f$ is injective, then so is $f^{\antishriek}$, and distinct
elements of $\mathcal B_{\mathcal P}$ have distinct images in $\mathcal B_{\mathcal Q}$,
so no two terms of the expansion of $D\left(\Gamma^c\right)$ can cancel against one
another in the image. Hence $f^\antishriek$ induces a bijection from
$I\left(\Gamma^c\right)$ onto
$I\left(f^{\antishriek}(\Gamma^c)\right)$, compatible with $D$ by
Equation \eqref{eq:commutes}. Since the $\mathcal P$-algebra structure
of $f^\ast(A)$ is the structure of $A$ restricted along $f$, a family
$\left\{a_\alpha\right\}$ satisfies the relations of Definition
\ref{def: Defining systems for operadic Massey products} for
$\Gamma^c$ if and only if the family
\[
b_{f^{\antishriek}(\mu),\left(i_1,\dots,i_k\right)} := a_{\mu,\left(i_1,\dots,i_k\right)}
\]
satisfies them for $f^{\antishriek}\left(\Gamma^c\right)$, the two
relations are equal on a term by term basis; and so the associated cycles
\eqref{ecu : Massey cycle} are then the same element of $A$. Thus we can conclude that the two
Massey product sets therefore coincide.

$(2)$ Assume $n = 2$. Every block of every term of
$D\left(\Gamma^c\right)$ then has weight at most $1$, so
$I\left(\Gamma^c\right)$ consists of the indices $(\idd,(i))$ for
$i = 1,\dots,r$ together with indices $\left(\mu,\left(k_1,\dots,k_s\right) \right)$, with $\mu$ of weight $1$. By definition, a
$\Gamma^c$-defining system is a family $\left\{a_\alpha\right\}$
in which $a_{\idd,(i)}$ is a cycle representing $x_i$ and
\begin{equation}
\label{ecu: weight two defining relation}
d\, a_{\mu,\left(k_1,\dots,k_s\right)} = \kappa(\mu)\left(a_{\idd,(k_1)},\dots,
a_{\idd,(k_s)}\right).
\end{equation}
Fix $x\in\langle
x_1,\ldots,x_r\rangle_{f^{\antishriek}(\Gamma^c)}$ with a defining
system $\left\{b_\beta\right\}$ and, using that $\langle
x_1,\ldots,x_r\rangle_{\Gamma^c}$ is defined, choose some choice of a $\Gamma^c$-defining
system $\left\{c_\alpha\right\}$ not necessarily defining $x$. It follows from Proposition
\ref{prop : Massey products do not depend on cycles} that we may assume without loss of generality that $c_{\idd,(i)} = b_{\idd,(i)}$ for each $i$. Set $a_{\idd,(i)} :=
b_{\idd,(i)}$. Computed in $f^\ast(A)$, the right-hand side of
\eqref{ecu: weight two defining relation} equals
$\kappa'\left(f^{\antishriek}(\mu)\right)\left(b_{\idd,(k_1)},\dots,
b_{\idd,(k_s)}\right)$ computed in $A$; it therefore vanishes when
$f^{\antishriek}(\mu) = 0$, and otherwise depends only on
$\left(f^{\antishriek}(\mu),\left(k_1,\dots,k_s\right)\right)$. 
Define
$$
a_{\mu,\left(k_1,\dots,k_s\right)} :=
\begin{cases}
0 & \mbox{if } f^{\antishriek}(\mu) = 0;\\
b_{\left(f^{\antishriek}(\mu),\left(k_1,\dots,k_s\right)\right)} & \mbox{if }
\left(f^{\antishriek}(\mu),\left(k_1,\dots,k_s\right)\right) \in
I\left(f^{\antishriek}(\Gamma^c)\right);\\
c_{\mu_0,\left(k_1,\dots,k_s\right)} & \mbox{otherwise,}
\end{cases}
$$
where in the last case $(\mu_0,\left(k_1,\dots,k_s\right))$ is a fixed index chosen in the
fiber of $\left(f^{\antishriek}(\mu),\left(k_1,\dots,k_s\right)\right)$.
Equation
\eqref{ecu: weight two defining relation} holds in each case: in the
first because both sides vanish, in the second because it is the
relation satisfied by 
$b_{f^{\antishriek}(\mu),\left(k_1,\dots,k_s\right)}$,
and in the third because every index of the fiber prescribes the same right-hand side.

It remains to compute the Massey product computed by this defining system.
Equation (\ref{ecu : Massey cycle}) for $\left\{a_\alpha\right\}$ evaluates in
$f^\ast(A)$ as $\zeta\left(a_{\zeta_1},\dots,a_{\zeta_m}\right) =
f(\zeta)\left(a_{\zeta_1},\dots,a_{\zeta_m}\right)$ in $A$, and the
values $a$ depend only on the images of the blocks. Two terms of
$D\left(\Gamma^c\right)$ with the same image under $f\circ
f^{\antishriek}$ thus contribute equally, their coefficients add, and
terms whose image coefficients cancel contribute nothing to the sum. By Equation
\eqref{eq:commutes}, we conclude that the product is the cycle
\eqref{ecu : Massey cycle} associated to $\left\{b_\beta\right\}$,
whose class is $x$. Therefore $x \in \langle
x_1,\ldots,x_r\rangle_{\Gamma^c}$ as required.
\end{proof}

\begin{remark}
    It is natural to ask about the missing $n>2$ case of the above result, 
    i.e., about the functorial properties of the higher Massey products.
 We believe the statement of the result is not true in this generality and that explicit,
    rather involved counterexamples can be given. 
    We leave open the possibility that more complex formulae involving matrix Massey products can be found.
\end{remark}

\begin{example}
\label{Example: Commutative Massey Products}
We can use Prop \ref{Prop: Pullback of Massey products} to compute the Massey products for the commutative operad $\mathsf{Com}.$
Consider the canonical weight preserving operad map $f:\mathsf{Ass}\to \mathsf{Com}.$ 
The map  $f^{\antishriek}:\mathsf{Ass}^{\antishriek}\to \mathsf{Com}^{\antishriek}$ 
is the linear dual of the natural operad morphism $\mathsf{Lie}\to \mathsf{Ass}$. 
This last operad map is an embedding, 
so it follows that $f^{\antishriek}$ is surjective.  
Thus, for any $\tau \in \left(\mathsf{Com}^{\antishriek}\right)^{(2)}(r)$ there
exists $\mu \in \left(\mathsf{Ass}^{\antishriek}\right)^{(2)}(r)$ such that $f^{\antishriek}(\mu) = \tau$,
and, since $f$ is monomial in weight one, it follows from Proposition \ref{Prop: Pullback of Massey products} (2) that
$$
  \langle x_1,\ldots,x_k\rangle_{\tau} \subseteq \langle x_1,\ldots,x_k\rangle_{\mu} 
$$
whenever the products above make sense.   
However as the map $f^\antishriek$ is not injective, 
equality may not hold. Consider the commutative algebra $A$ with zero differential 
and basis $\{a,b,c, y, cy\}$  where $|a|=|b|=|c| =2$ and $|y|=5$ with the obvious product.  
Then $\Bbbk \cdot cy = \langle a,b, c \rangle_{\mu_3^c - \mu_3^c\cdot (1,2)}.$ 
However  $\{0\} = \langle a,b,c \rangle_{f^{\antishriek}\left(\mu_3^c - \mu_3^c\cdot (1,2)\right)}$. 
We have an inclusion $\{0\} \subsetneq \Bbbk \cdot cy$ but this is manifestly not an equality. 
This is linked to the fact that $\left(f^\antishriek(\mu), (i_1, \dots, i_r)\right)$ need not belong to 
$I\left(f^\antishriek(\Gamma^c), (1, \dots, r)\right)$.
\hfill$\square$
\end{example}

\medskip

\noindent{\bf{Pushforwards of Massey products. }} 
When $f$ is injective, for any $\mathcal P$-algebra $A$,
the $\mathcal P$-Massey products on $A$ induce $\mathcal Q$-Massey products on $f_!(A)$.
Recall that $h$ is the chain map $h: A\to f_!(A)$ defined in \eqref{ecu: chain map h}.

\begin{corollary}
\label{cor: otra}
Let $f:\mathcal P \to \mathcal Q$ be an injective, weight preserving, monomial morphism of operads,
 $A$ a $\mathcal P$-algebra, and $\Gamma^c\in \left(\mathcal P^{\antishriek} \right)^{(n)} (r)$.
Suppose that  $x_1,\ldots,x_r \in H_*(A)$ are homogeneous elements such that
the Massey product set $\langle x_1,\ldots,x_r\rangle_{\Gamma^c}$ is defined. 
Then the $\mathcal Q$-Massey product set 
$\langle h_\ast(x_1),\ldots,h_\ast(x_r)\rangle_{f^{\antishriek}(\Gamma^c)}$ is also defined, 
and 
$$h_\ast\langle x_1,\ldots,x_r\rangle_{\Gamma^c} \subseteq \langle h_\ast(x_1),\ldots,h_\ast(x_r)\rangle_{f^{\antishriek}(\Gamma^c)}  .$$
\end{corollary}

\begin{proof}
This is a straightforward consequence of previous results.
Indeed, one applies Prop. \ref{prop : Morphisms preserve Massey products} to the unit map $u:  A \to f^*f_!(A) $,
which implies that $u_\ast \langle x_1,\ldots,x_r\rangle_{\Gamma^c} \subseteq \langle u_*(x_1),\ldots,u_*(x_r) \rangle_{\Gamma^c}$.
Then, the injective case of Prop. \ref{Prop: Pullback of Massey products} equates this Massey product sets with  
the associated $\mathcal Q$-Massey product set of $f_!(A)$, concluding the proof.
\end{proof}

\section{Differentials in the \texorpdfstring{$\mathcal P$}{P}-Eilenberg--Moore spectral sequence}
\label{sec: differentials EMSS}

Aside from providing obstructions to formality, 
one of the major uses of higher Massey products is in providing a concrete description of 
the differentials in the classical Eilenberg--Moore spectral sequence.  
The following is a classical result of May \cite{may66}, 
compare also \cite{Sta67},
but adapted to the notation of this paper.

\begin{theorem}
Let $A$ be a differential graded associative algebra, 
and let $ x_1, \dots x_n $  be homology classes 
such that the Massey product set $\langle x_1,\ldots,x_n\rangle$ is non-empty.
Then, the element $[sx_1 \mid \cdots \mid sx_n]$ survives to the $E^{n-1}$-page 
of the Eilenberg--Moore spectral sequence of $A$, 
and furthermore, the suspension $sx$ of any representative of $x\in \langle x_1, \dots x_n \rangle$
is a representative for $d^{n-1}[sx_1 \mid \cdots \mid sx_n]$.
\end{theorem}

An analogous statement for differential graded Lie algebras appears in \cite{Allday77}. 
Our following result generalizes these statements to all algebras over a Koszul (in fact, quadratic) operad.
Recall from Section \ref{sec: Operad prerequisites} the construction of the spectral sequence.
We will sometimes confuse homology classes with representatives to lighten the notation.

\begin{theorem}
\label{Thm : Differentials EMSS and Massey products}
Let $A$ be a $\mathcal P$-algebra, and $ x_1, \dots, x_r $ homology classes 
such that the Massey product set $\langle x_1,\ldots,x_r\rangle_{\Gamma^c}$ is defined 
for a cooperation $\Gamma^c \in \left(\mathcal{P}^{\antishriek}\right)^{(n)}(r)$. 
Let the total degree be $t = |\Gamma^c| + \sum_{i=1}^r |x_i|$.
Then the element 
\[
\Gamma^c\otimes x_1\otimes \cdots \otimes x_r \in \left(\mathcal P^{\antishriek} \right)^{(n)} (r)\otimes_{\mathbb S_r} H_*(A)^{\otimes r}
\]
survives to the $E^{n}$ page in the $\mathcal P$-Eilenberg--Moore spectral sequence,
and for any $x\in \langle x_1, \dots, x_r \rangle_{\Gamma^c}$
$$
d^{n}\left(\Gamma^c\otimes x_1\otimes \cdots \otimes x_r\right) = (-1)^{n-1} \left[\idd\otimes x\right] \qquad \mbox{in } E^n_{0, t-1}.
$$
\end{theorem}

Our proof of this is an adaptation of the classical one, 
so we shall therefore make use of the Staircase Lemma \cite[Lemma 2.1]{kraines72}, 
which we briefly recall next.

\begin{lemma} 
\label{lemma : Staircase lemma}
Let $A=\left(A_{*,*}, d', d''\right)$ be a bicomplex,
denote by $d$ the differential on its total complex, 
and fix $c_1,\dots, c_n$ homogeneous elements in $A$. 
Suppose that $d'c_s = d''c_{s+1}$ for $1 \leq s\leq n-1$,
and define $c := c_1-c_2+\cdots +(-1)^{n-1}c_n.$ 
Then, $dc = d'c+d''c = d''c_1+ (-1)^{n-1}d'c_n,$ and furthermore,
in the spectral sequence $\left\{\left(E^r,d^r\right)\right\}$ associated to the bicomplex, 
if $d''c_1=0$ then $c_1$ survives to $E^n$, and $d^n[c_1]= (-1)^{n-1}[d'c_n].$
\end{lemma}

Our approach to proving Theorem \ref{Thm : Differentials EMSS and Massey products} is therefore to construct a sequence $c_1,\dots, c_{n}$
satisfying the conditions of the Staircase Lemma.

\begin{proof}[Proof of Theorem \ref{Thm : Differentials EMSS and Massey products}]
 First, fix a defining system $\left\{a_{\mu,(k_1,\dots,k_i)}\right\}$ 
 for the element $x\in \langle x_1, \dots, x_r \rangle_{\Gamma^c}$.
For each $s$ between $1$ and $n$, 
we will define $c_s$ in terms of this defining system and the auxiliary maps
\[
\Delta_s: \mathcal P^{\antishriek} \xrightarrow{\Delta} \mathcal P^{\antishriek} \circ \mathcal P^{\antishriek} \xrightarrow{p_s\circ \operatorname{id}} \left(\mathcal P^{\antishriek}\right)^{(s)} \circ \mathcal P^{\antishriek},
\]
where $p_s$ is the projection onto the weight $s$ component. 
More precisely, the element $c_s$ is defined as
\[
    c_s := \sum \lambda\,(-1)^{\gamma}\left[\theta \otimes a_{\zeta_1,K_1}\otimes \cdots \otimes a_{\zeta_m,K_m}\right],
\]
where $\Delta_{n-s+1}\left(\Gamma^c\right) = \sum \lambda\left(\theta; \zeta_1, \dots, \zeta_m; \sigma\right)$,
with $w(\theta)=n-s+1$, the $K_i$ are as in Notation \ref{not: shorthand D} for $(1,\dots,r)$,
and $\gamma$ is given by the formula of 
Definition \ref{def : Higher-order Massey products for algebras over algebraic operads} 
for the term $\left(\theta; \zeta_1, \dots, \zeta_m; \sigma\right)$.
The elements $a_{\zeta_i,K_i}$ above are defined.
Indeed, applying $\kappa\circ\idd\circ\idd$ to the coassociativity identity 
$\left(\Delta\circ\idd\right)\Delta = \left(\idd\circ\Delta\right)\Delta$,
and using that $D$ is injective on elements of positive weight,
we obtain that every block of a term of $\Delta\left(\Gamma^c\right)$ whose root is different from $\idd$ 
is a block of a term of $\Delta\left(\zeta_j\right)$ for some block $\left(\zeta_j,K_j\right)$ of $D\left(\Gamma^c\right)$.
Such a block is either $\left(\zeta_j,K_j\right)$ itself, which lies in $I\left(\Gamma^c\right)$, 
or a block of a term of $\Delta\left(\zeta_j\right)$ whose root is different from $\idd$, 
which lies in $I\left(\zeta_j,K_j\right)\subseteq I\left(\Gamma^c\right)$ by induction on the weight.
In particular, $c_1 = [\Gamma^c \otimes a_{\idd,(1)}\otimes \cdots \otimes a_{\idd, (r)}],$
and
\begin{equation*}
    c_{n} = \sum \lambda\,(-1)^{\gamma}\, s\zeta \otimes a_{\zeta_1,K_1}\otimes\cdots\otimes a_{\zeta_m,K_m},
\end{equation*}
where
$D\left(\Gamma^c\right) = \sum \lambda\left(\zeta; \zeta_1,\dots,\zeta_m; \sigma \right)$
with $\zeta\in\mathcal B_E$, so that $\theta = s\zeta\in \mathcal B^{(1)}$.
To finish, we must verify that the conditions of the Staircase Lemma \ref{lemma : Staircase lemma} are met.
Denote by $\partial$  the external differential on $\mathcal P^{\antishriek}(A)$, and by $d^\bullet$ its internal differential.
Then, since $d(a_{\idd, (i)}) = 0$ for each $i$, it follows that $d^\bullet c_1 = 0$.
A routine calculation shows that $d^\bullet c_{s+1} = \partial c_s$ for $1\leq s\leq n-1$.
It follows from the Staircase Lemma that 
\[
d^{n}[c_1] = (-1)^{n-1}[\partial c_{n}] = (-1)^{n-1}[\idd\otimes x].
\]
In the expression of $c_{n}$, each root $s\zeta$ has weight $1$, so that
\[
\partial c_{n} = \idd\otimes\sum \lambda\,(-1)^{\gamma}\,\zeta\left(a_{\zeta_1,K_1},\dots,a_{\zeta_m,K_m}\right) 
= \idd\otimes a_{\Gamma^c,(1,\dots,r)}.
\]
The class of this element is $\idd\otimes x$.
This finishes the proof.
\end{proof}

The formality of a dg algebra of some type is well-known to be related to 
the collapse of the associated Eilenberg--Moore-type spectral sequence, 
see for instance \cite{HalperinStasheff} for the commutative case,
and \cite{DanielYves} for the Lie case.
The following statement generalizes these results.

\begin{theorem}
\label{thm : formal implies collapse}
The Eilenberg--Moore spectral sequence of a 
formal $\mathcal P$-algebra  over a Koszul operad collapses at the $E^2$-page. 
The same is true for formal $\mathcal P_\infty$-algebras.
\end{theorem}

\begin{proof}
Since every $\mathcal P$-algebra is a $\mathcal P_\infty$-algebra,
we prove the result for $\mathcal P_\infty$-algebras.
Let $A$ be a formal $\mathcal P_\infty$-algebra, and denote by $H=H_*(A)$ its homology as a $\mathcal P$-algebra. 
Then there are $\mathcal P_\infty$-quasi-isomorphisms $A\leftrightarrows H$, 
or equivalently, 
$\mathcal P^{\antishriek}$-coalgebra quasi-isomorphisms
\[
\mathcal P^{\antishriek}(A) \leftrightarrows  \mathcal P^{\antishriek}(H).
\]
Recall that the codifferential $\delta_H(\mu, -)$ of $\mathcal P^{\antishriek}(H)$ vanishes unless $\mu\in \mathcal{P}^{\antishriek}$ has weight 1.
By comparison, 
both Eilenberg--Moore spectral sequences are isomorphic from the first page. 
Therefore, 
it suffices to consider the case where $A$ has no internal differential 
and $\delta_A^{(i)}$ vanishes when $i\neq 1$. 
We now check that the differential $d^i$ in the Eilenberg--Moore spectral sequence vanishes for $i\geq 2.$
To do so, 
we will use the standard relative cycles and boundaries spaces,
\[
Z^r_p=F_p \mathcal{P}^{\antishriek}(A)\cap \delta^{-1}\left(F_{p-r} \mathcal{P}^{\antishriek}(A)\right)\quad \textrm{and} \quad D_p^r=F_p \mathcal{P}^{\antishriek}(A) \cap \delta\left(F_{p+r}\mathcal{P}^{\antishriek}(A)\right).
\]
The differential $d^r$ in the successive pages of the spectral sequence is induced by the restrictions
of the differential $\delta$ of $\mathcal{P}^{\antishriek}(A)$ to 
$Z^r_p$, as shown below:
\begin{center}
\begin{tikzcd}[row sep=huge, column sep=huge]
Z^r_p \ar{r}{\delta} 	
\ar[two heads]{d}		&		Z^r_{p-r} \ar[two heads]{d}\\
E^r_p = Z^r_p / Z^{r-1}_{p-1}+D^{r-1}_{p} \ar[dashed]{r}{d^r} 
&   E^r_{p-r}= Z^r_{p-r} / Z^{r-1}_{p-r-1}+D^{r-1}_{p-r}
\end{tikzcd} 
\end{center}
Fix some $r\geq 2$. 
Since $A$ is strict with zero internal differential, 
the codifferential $\delta$ lowers the weight filtration by exactly one. 
To check that $d^r=0$, 
fix $x\in Z^r_p$ and write $x=x_1+\cdots+x_p$ with $x_i\in \mathcal{P}^{\antishriek}(A)^{(i)}$.
Since $\delta(x)\in F_{p-r}\mathcal{P}^{\antishriek}(A)$ and $\delta(x_i)$ is concentrated in weight $i-1$, 
it follows that $\delta(x_i)=0$ for every $i\geq p-r+2$.
In particular, $\delta(x_p)=0$, as $r\geq 2$.
Take $y:=x_p$. 
Then $x-y=x_1+\cdots+x_{p-1}\in F_{p-1}\mathcal{P}^{\antishriek}(A)$ and $\delta(x-y)\in F_{p-r}\mathcal{P}^{\antishriek}(A)$.
Therefore, $x-y\in Z^{r-1}_{p-1}$. So, $y$ represents the class $[x]\in E^r_p$.
Since $\delta(y)=0$, we conclude $d^r[x]=[\delta(y)]=0$, finishing the proof.
\end{proof}

The converse to Theorem \ref{thm : formal implies collapse} in general is not true, it fails even in the associative case and some
examples are computed in some of the references given before the statement of the theorem.

\section{Higher-order operadic Massey products and \texorpdfstring{$\mathcal{P}_\infty$}{P-infinity}-structures}
\label{sec : Massey products and P-infinity algebras}

In this section, we  fix a Koszul operad $\mathcal{P}$.
In this case, there is a natural relationship
between the higher-order operadic Massey products and $\mathcal P_\infty$-structures 
on the homology of the $\mathcal P$-algebras.

\medskip

Let $A$ be a $\mathcal P$-algebra, and denote by $H$ its homology.
Since $\mathcal P$ has no operadic differential, $H$ is a $\mathcal P$-algebra in a natural way.
It is well-known that the homotopy transfer theorem (in its various forms)
extends this $\mathcal P$-algebra structure on $H$ to a $\mathcal P_\infty$-structure 
that retains the quasi-isomorphism class of $A$ as a $\mathcal P_\infty$-algebra.
In this paper, 
we mainly focus on D. Petersen's extension \cite{petersen2020}
of T. Kadeishvili's classical transfer theorem \cite{Kad80}, 
which is recalled in Theorem \ref{teo : Homotopy Transfer Theorem}.
See also \cite[Section 10.3]{loday12}.
It is a common misconception to expect that higher-order Massey product sets of the sort $\langle x_1,\ldots,x_r\rangle$ 
are related to $\mathcal P_\infty$-structure maps $\theta_r$ induced on the homology $H$ via the homotopy transfer theorem 
by the clean formula
\[
\pm\theta_r(x_1,\ldots,x_r) \in \langle x_1,\ldots,x_r\rangle.
\]
At this level of generality, the assertion is incorrect.
However, it is true for secondary Massey products,
as shown in  \cite[Theorem 3.9]{Mur21}.

Let us make the connection between infinity structures and higher-order Massey products more precise. 
First, recall that codifferentials on the cofree conilpotent $\mathcal P^\antishriek$-coalgebra $\mathcal P^\antishriek(A)$ 
are in bijective correspondence with $\mathcal P_\infty$-structures 
on the chain complex $A$ \cite[Theorem 10.1.13]{loday12}.

\begin{definition}
Let $A$ be a $\mathcal{P}$-algebra,
$\Gamma^c\in \left(\mathcal{P}^{\antishriek}\right)^{(n)} (r)$, and  $x_1,\ldots,x_r$ homogeneous elements of $H=H_*(A)$
for which the $\Gamma^c$-Massey product set 
$\langle x_1, \dots, x_r \rangle_{\Gamma^c}$ is defined.
A given $\mathcal{P}_\infty$ structure $\delta$ on $H$ for which $A$ and $H$ are quasi-isomorphic 
is said to \emph{recover} the Massey product element $x \in \langle x_1,\dots, x_r\rangle_{\Gamma^c}$ if, 
up to sign, 
\[
\delta_r\left(\Gamma^c \otimes x_1\otimes \cdots \otimes x_r\right) = x.
\]
\end{definition}

We begin by showing below that given a higher-order Massey product $x \in \langle x_1,\dots, x_r\rangle_{\Gamma^c}$, 
there is always a choice of $\mathcal P_\infty$ structure on $H$ quasi-isomorphic to $A$ which recovers $x$.
In general, however, 
an arbitrary $\mathcal P_\infty$ structure on $H$ quasi-isomorphic to $A$ 
only recovers a given higher-order Massey product element up to multiplications of lower weight.  
Our proof strategy is very similar to the proof in \cite{jose20}, where
the authors demonstrated this result in the associative case.

\begin{theorem} 
\label{thm: P-infinity and recovery of Massey}
Let $A$ be an algebra over a Koszul operad $\mathcal P$,
and let $H$ be its homology endowed with a minimal $\P_\infty$-structure $\delta$.
Let $\Gamma^c\in \left(\mathcal{P}^{\antishriek}\right)^{(n)} (r)$,
and assume that $x_1,\ldots,x_r$ are homogeneous elements of $H$ for which the $\Gamma^c$-Massey product set
$\langle x_1, \dots x_r \rangle_{\Gamma^c}$ is defined.
Let $x \in \langle x_1, \dots x_r \rangle_{\Gamma^c}$. 
Then:
\begin{enumerate}[(i)]
\item For any $\P_\infty$-quasi-isomorphism $F:\P^\antishriek(H) \to \P^\antishriek(A)$
whose linear term $F_1 :H\to A$ is a cycle-choosing map inducing the identity of $H$, 
\[
\delta^{(n)}\left(\Gamma^c \otimes x_1 \otimes \cdots\otimes x_r\right) = (-1)^{n-1}x + \Phi,
\]
where
$\Phi \in \displaystyle \sum^{n-1}_{i=1}\operatorname{Im} \left(\delta^{(i)}\right).$
\item If the elements $\mu\otimes x_{i_1}\otimes \cdots \otimes x_{i_l}$ are linearly independent
in the corresponding copy of $\mathcal{P}^{\antishriek} \otimes_{\mathbb S_l} H^{\otimes l}$, 
where $(\mu,(i_1,\ldots,i_l))\in  I(\Gamma^c)$, then 
there is a choice of $\mathcal{P}_\infty$ structure $\delta$ on $H$, together with a quasi-isomorphism $F$ as above,
which recovers $x$.
\end{enumerate}
\end{theorem}

\begin{proof}
$(i)$ Let $F \colon H\xrightarrow{\simeq} A$ be a $\mathcal P_\infty$-quasi-isomorphism as in the statement, 
and consider the corresponding quasi-isomorphism of $\mathcal P^{\antishriek}$-coalgebras
$$
\mathcal P^{\antishriek}(H) \xrightarrow{\simeq}\mathcal P^{\antishriek}(A).
$$
The induced morphism of $\mathcal P$-Eilenberg--Moore spectral sequences is, 
at the $E^1$ level,  
the identity on $ \mathcal P^{\antishriek}(H)$. 
By comparison, all the terms in both spectral sequences are also isomorphic.

Now, it follows from Theorem \ref{Thm : Differentials EMSS and Massey products} that 
if $\langle x_1, \dots x_r \rangle_{\Gamma^c}$ is nonempty,
then the element $[\Gamma^c\otimes x_1\otimes\cdots\otimes x_r]$ survives to the $n$-page $(E^{n}, d^{n})$.
Moreover, given any $x\in \langle x_1, \dots x_r \rangle_{\Gamma^c}, $ one has 
\[
d^{n}\overline{\Gamma^c\otimes x_1\otimes\cdots\otimes x_r} = (-1)^{n-1}\overline{x}. 
\]
Here, $\overline{\hspace{0.1cm} \cdot \hspace{0.1cm}}$ denotes the class in $E^{n}$.
In other words, there exists $\Phi\in F_{n-1}\mathcal P^{\antishriek}(H)$ such that
$$
\delta_H\left(\Gamma^c\otimes x_1\otimes\cdots\otimes x_r + \Phi\right) = (-1)^{n-1}x.
$$
Applying the counit $\epsilon_H:\mathcal P^{\antishriek}(H) \to  H$ to both sides, we obtain
$$
m_H\left(\Gamma^c\otimes x_1\otimes\cdots\otimes x_r + \Phi\right) = (-1)^{n-1}x.
$$

Write $m_H = \sum_{i\geq 1} \delta_H^{(i)}$, 
and decompose $\Phi=\sum^{n-1}_{i=0}\phi_i$ with $\phi_i\in \mathcal{P}^{\antishriek}(H)^{(i)}$.
By a word length argument, 
\[
\delta_H^{(n)}\left(\Gamma^c\otimes x_1\otimes\cdots\otimes x_r\right) + \sum_{i=1}^{n-1}\delta_H^{(i)}(\phi_i) = (-1)^{n-1}x.
\]
This completes the proof.   

\medskip
    
$(\mathit{ii})$   We will construct a $\mathcal{P}_\infty$ structure on $H$ recovering $x$ 
via the procedure established in the proof of Theorem \ref{teo : Homotopy Transfer Theorem}. 
We shall continue to use the notation of that proof. 
First, 
choose a defining system $\left\{a_{\alpha}\right\}$ for the Massey product element $x\in \langle x_1, \dots, x_r \rangle_{\Gamma^c}$. 
We proceed by induction on the weight, 
starting in weight $0$ with $\delta^0$ initially defined as the coderivation 
corresponding to the strict $\mathcal P$-algebra structure on $H$ induced from it being the homology of $A$, 
and 
defining $f$ as any cycle-choosing map $H\to A$ extending the 
choice 
\[
f(x_1) = a_{\idd, (1)}, \quad  f(x_2) = a_{\idd, (2)}, \quad \dots,\quad  f(x_r) = a_{\idd,(r)}.
\]
This choice is possible, since $x_1,\dots,x_r$ are linearly independent, and defines the weight $0$ component of $F$. 
This is the first step of the induction, 
and we include it in detail because we think it sheds light on the general case.
Recall that the algorithm of Theorem \ref{teo : Homotopy Transfer Theorem} automatically determines 
the weight $1$ component of $\delta$, which is the $\mathcal P$-algebra structure on $H$, 
but there are choices for the weight $1$ component of $F$. 
First, we make the following observation. 
If $\left(s\mu, (i_1,\dots,i_m)\right)$ appears in the $\Gamma^c$-indexing set, 
then $\overline{\gamma_A}(\mu; x_{i_1},\dots,x_{i_m}) =0$,
where $\overline{\gamma_A}$ is the induced structure in homology.
This is because $D\left(s\mu\right) = \left(\mu; \idd,\dots, \idd\right)$, and 
therefore $$da_{s\mu, (i_1,\dots,i_m)} = \kappa\left(s\mu\right)\left(a_{\idd, (i_1)},\dots,a_{\idd, (i_m)}\right),$$ 
which implies that $\overline{\gamma_A}(\mu; x_{i_1},\dots,x_{i_m})\in H$ admits a lift to $A$ which is a boundary, 
hence it vanishes on homology. 
Therefore,
\[
\left(F^{0}\circ \delta^{0}\right)_{1} [s\mu\otimes x_{i_1}\otimes\cdots\otimes x_{i_m}]=0.
\]
On the other hand, 
\[
\left(\delta_A \circ F^{0}\right)_{1}[s\mu\otimes x_{i_1}\otimes\cdots\otimes x_{i_m}]
= \kappa \left(s\mu\right)\left(f\left(x_{i_1}\right),\dots, f\left(x_{i_m}\right)\right) = da_{s\mu, (i_1,\dots,i_m)}.
\]
Hence, the obstruction $\left(F^{0}\circ \delta^{0}-\delta_A \circ F^{0}\right)_{1}$ 
takes the value $-da_{s\mu, (i_1,\dots,i_m)}$ on this element, its projection $e$ to $H$ vanishes, 
and we may choose $e'$ with $e'\left(s\mu\otimes x_{i_1}\otimes\cdots\otimes x_{i_m}\right) = -a_{s\mu, (i_1,\dots,i_m)}$.
So, the weight $1$ component of $F$ maps $s\mu\otimes x_{i_1}\otimes\cdots\otimes x_{i_m}$ to $-a_{s\mu, (i_1,\dots,i_m)}$.
The general case is similar. 
Our inductive hypothesis for the weight $N$ step, $N\geq 2$, has the following two parts:
\begin{equation}
\label{ecu: induction 1}
    \epsilon_A F^{N-1}\left[\mu\otimes x_{i_1}\otimes \cdots \otimes x_{i_l}\right] = (-1)^{w(\mu)}a_{\mu,\left(i_1,\dots, i_l\right)}, \textrm{ where }
\left(\mu, \left(i_1,\dots i_l\right)\right) \in I\left(\Gamma^c\right) \textrm{ and } w(\mu) < N,
\end{equation}
and 
\begin{equation}
\label{ecu: induction 2}
    \epsilon_H \delta^{N-1}\left[\mu\otimes x_{i_1}\otimes\cdots\otimes x_{i_l}\right]= 0 \textrm{, where } \left(\mu, \left(i_1,\dots i_l\right)\right) \in I\left(\Gamma^c\right)\textrm{ and } 1\leq w(\mu) < N.
\end{equation}
We verified these two items in the weight $1$ case in the previous paragraph. 
Let $\left(\mu, \left(i_1,\dots i_l\right)\right) \in I\left(\Gamma^c\right)$ with $w(\mu)=N$.
Next, we shall compute $\left(\delta_A \circ F^{N-1}\right)_{N}\left[\mu\otimes x_{i_1}\otimes \dots\otimes x_{i_l}\right]$.
The map $\left(\delta_A \circ F^{N-1}\right)_{N}$ is precisely the composite
\[
\mathcal{P}^{\antishriek}\left(H\right) \xrightarrow{\Delta(H)} \mathcal{P}^{\antishriek}\circ \mathcal{P}^{\antishriek}\left(H\right) \xrightarrow{\mathcal{P}^{\antishriek}\left(f\right)}  \mathcal{P}^{\antishriek}\left(A\right) \xrightarrow{\kappa(A)}\mathcal{P}(A)\xrightarrow{\gamma_A}A.
\]
The weight $N$ component of $F^{N-1}$ is 0, 
and in particular $\mathcal{P}^{\antishriek}\left(f\right)\left(\idd; \mu\otimes x_{i_1}\otimes \dots\otimes x_{i_l} \right) = 0$. 
It follows that $\left(\delta_A \circ F^{N-1}\right)_{N}$ sends the class of 
$\mu\otimes x_{i_1}\otimes \dots\otimes x_{i_l}$ to the same element as the following map does:
\[
\mathcal{P}^{\antishriek}\left(H\right) \xrightarrow{\Delta} \mathcal{P}^{\antishriek}\circ \mathcal{P}^{\antishriek}\left(H\right) \xrightarrow{\mathcal{P}^{\antishriek}\left(f\right)}  \mathcal{P}^{\antishriek}\left(A\right) \xrightarrow{\kappa(A)}\mathcal{P}(A)\xrightarrow{\gamma_A}A.
\]
The map above is tightly related to the Massey inductive map $D$. 
Indeed, the image of $\mu\otimes x_{i_1}\otimes \dots\otimes x_{i_l}$ is given by 
\[
\sum \lambda(-1)^{\gamma}\zeta\left(
f\left(\zeta_1\otimes x_{i_{\sigma^{-1}\left(1\right)}}\otimes
\dots
\otimes
x_{i_{\sigma^{-1}\left(v_1\right)}}\right),
\dots,
f\left(\zeta_m
\otimes
x_{i_{\sigma^{-1}\left(v_1+\cdots+ v_{m-1}+1\right)}}\otimes\cdots\otimes x_{i_{\sigma^{-1}(l)}}\right)
\right),
\]
where 
\[
D\left(\mu\right) = \sum \lambda\left(\zeta ; \zeta_1,\dots,\zeta_m; \sigma\right),
\]
and $\gamma$ is the sign of 
Definition \ref{def: Defining systems for operadic Massey products}.
By the first assumption of our inductive hypothesis \eqref{ecu: induction 1},
and since $w(\zeta_1)+\cdots+w(\zeta_m) = N-1$,
this is equal to 
\[
(-1)^{N-1}\sum \lambda(-1)^{\gamma} \zeta \left(a_{\zeta_1,\left(i_{\sigma^{-1}\left(1\right)},\dots i_{\sigma^{-1}\left(v_1\right)}\right)},\dots,a_{\zeta_m,\left(i_{\sigma^{-1}\left(v_1+\cdots+ v_{m-1}+1\right)},\dots, i_{\sigma^{-1}(l)}\right)} \right).
\]
It follows from the definition of a defining system that this is equal to 
\[
(-1)^{N-1}da_{\mu, (i_1, \dots, i_l)}.
\]
On the other hand, since $\delta^{N-1}$ is a coderivation, 
$\delta^{N-1}\left[\mu\otimes x_{i_1}\otimes\cdots\otimes x_{i_l}\right]$ is obtained by applying $\epsilon_H\delta^{N-1}$ 
to the inner factors of the terms of $\Delta_{(1)}(\mu)$.
The inner factors of the terms whose root is different from $\idd$ are blocks of terms of $\Delta(\mu)$ whose root is different from $\idd$, 
so they lie in $I\left(\mu,\left(i_1,\dots,i_l\right)\right)\subseteq I\left(\Gamma^c\right)$ by the argument in the proof of 
Theorem \ref{Thm : Differentials EMSS and Massey products}.
Therefore, the second assumption of our inductive hypothesis \eqref{ecu: induction 2}, 
together with $\delta^{(0)} = 0$ and the vanishing of the weight $N$ component of $\delta^{N-1}$, implies 
that
\[
\left(F^{N-1}\circ \delta^{N-1}\right)_{N}\left[\mu\otimes x_{i_1}\otimes\cdots\otimes x_{i_l}\right]=0.
\]
Therefore,
\[
\left(F^{N-1}\circ \delta^{N-1} -\delta_A \circ F^{N-1}\right)_{N}\left[\mu\otimes x_{i_1}\otimes \dots\otimes x_{i_l}\right]
= (-1)^{N}da_{\mu, (i_1, \dots, i_l)}.
\]
This implies $e$ vanishes on this element, 
and there is 
no obstruction to obtaining a lift $F^N$ such 
that $\epsilon_A F^N(\mu\otimes x_{i_1}\otimes \dots\otimes x_{i_l})=(-1)^N a_{\mu, (i_1, \dots, i_l)}$,
since these elements are linearly independent by hypothesis.
Notice that the algorithm also tells us 
that $\epsilon_H\delta^N\left[\mu\otimes x_{i_1}\otimes \dots\otimes x_{i_l}\right] = 0$ (the projection of a boundary in homology).
This completes the induction step.

If $n=1$, then $\delta_r\left(\Gamma^c \otimes x_1\otimes \cdots \otimes x_r\right)$ is the induced product, which equals $x$.
If $n\geq 2$, applying the computation above with $\mu = \Gamma^c$ and $N=n$,
the inductive hypotheses \eqref{ecu: induction 1} and \eqref{ecu: induction 2} 
identify 
\[
\left(F^{n-1}\circ \delta^{n-1}-\delta_A \circ F^{n-1}\right)_{n}\left[\Gamma^c\otimes x_{1}\otimes \dots\otimes x_{r}\right] = 
(-1)^{n} a_{\Gamma^c,(1,\dots,r)}.
\]
The obstruction term $e$ of the algorithm of Theorem \ref{teo : Homotopy Transfer Theorem} at this
input is the projection to $H$ of $(-1)^{n} a_{\Gamma^c,(1,\dots,r)}$, that is, $(-1)^{n} x$. 
The resulting transferred structure hence satisfies 
\[
\delta_r\left(\Gamma^c \otimes x_1\otimes \cdots \otimes x_r\right) = -e = (-1)^{n-1} x,
\]
so $\delta$ recovers $x$, completing the proof.
\end{proof}

\bibliographystyle{plain}
\bibliography{MyBib}

\bigskip

\noindent\sc{Oisín Flynn-Connolly}\\ 
\noindent\sc{Leiden Institute of Advanced Computer Science, \\ Leiden University, Leiden, The Netherlands}\\
\noindent\tt{oisinflynnconnolly@gmail.com}\\

\noindent\sc{José Manuel Moreno Fernández}\\ 
\noindent\sc{Departamento de Álgebra, Geometría y Topología, \\ Universidad de Málaga, 29080, Málaga, Spain}\\
\noindent\tt{josemoreno@uma.es}\\

\end{document}